\documentclass[10pt,a4paper,dvipsnames]{article}

\usepackage{fullpage,amssymb,amsmath,amsfonts,amsthm,paralist,graphicx,color,float, mathtools,tikz,xfrac}
\usepackage[utf8]{inputenc}
\usepackage[T1]{fontenc}
\usepackage[english]{babel}
\usepackage[colorlinks=true,linkcolor=BrickRed,citecolor=blue]{hyperref}
\usepackage[normalem]{ulem}
\usepackage{centernot}
\usepackage{ifthen,xkeyval, tikz, calc, graphicx}
\usepackage{xcolor}
\usepackage{framed}
\usepackage[textsize=scriptsize, backgroundcolor=blue!20!white,bordercolor=red]{todonotes}
\usepackage{dsfont}
\usepackage{mathrsfs}
\usepackage{comment}
\usepackage{stmaryrd}
\usepackage{caption}
\usepackage{subcaption}
\usepackage{todonotes}
\usepackage{thmtools}
\usetikzlibrary{plotmarks}
\usetikzlibrary{shapes.misc}
\usepackage{orcidlink}

\definecolor{blueish}{HTML}{1A85FF}
\definecolor{redish}{HTML}{D41159}

\newtheorem{theorem}{Theorem}[section]
\newtheorem{lemma}[theorem]{Lemma}
\newtheorem{prop}[theorem]{Proposition}

\theoremstyle{definition}
\newtheorem{definition}[theorem]{Definition}

\theoremstyle{remark}
\newtheorem{remark}[theorem]{Remark}

\definecolor{shadecolor}{named}{GreenYellow}

\makeatletter
\newcommand{\pushright}[1]{\ifmeasuring@#1\else\omit\hfill$\displaystyle#1$\fi\ignorespaces}
\newcommand{\pushleft}[1]{\ifmeasuring@#1\else\omit$\displaystyle#1$\hfill\fi\ignorespaces}
\makeatother

\newcommand{\p}{\mathbb P}
\newcommand{\pla}{\mathbb P_\lambda}
\newcommand{\plaq}{\mathbb{P}_{\lambda,q}}
\newcommand{\ela}{\mathbb E_{\lambda}}
\newcommand{\elaq}{\mathbb E_{\lambda,q}}

\newcommand{\Scal}{\mathcal{S}}
\newcommand{\Ycal}{\mathcal{Y}}
\newcommand{\Mcal}{\mathcal{M}}
\newcommand{\Ucal}{\mathcal{U}}
\newcommand{\Wcal}{\mathcal{W}}

\newcommand{\E}{\mathbb E}

\newcommand{\R}{\mathbb R}

\newcommand{\Rd}{\mathbb R^d}

\newcommand{\N}{\mathbb N}
\newcommand{\B}{\mathcal B}

\newcommand{\HypTwo}{{\mathbb{H}^2}}
\newcommand{\HypThree}{{\mathbb{H}^3}}
\newcommand{\HypDim}{{\mathbb{H}^d}}
\newcommand{\dd}{\mathrm{d}} 
\newcommand{\C}{\mathscr {C}}

\newcommand{\piv}[1]{\mathsf {Piv}(#1)}

\DeclareMathOperator*{\esssup}{ess\,sup}
\DeclareMathOperator*{\essinf}{ess\,inf}

\newcommand{\LandauBigO}[1]{\mathcal{O}\left(#1\right)}

\newcommand{\convex}[1]{{\mathrm{conv}\left(#1\right)}}
\newcommand{\Gcal}{\mathcal{G}}

\newcommand{\dequal}{\,{\buildrel d \over =}\,}

\DeclareMathOperator{\arcosh}{arcosh}

\newcommand{\conn}[3]{#1 \longleftrightarrow #2\textrm { in } #3}
\newcommand{\adja}[3]{#1 \sim #2\textrm { in } #3}

\newcommand{\nconn}[3]{#1 \centernot\longleftrightarrow #2\textrm { in } #3}
\newcommand{\dconn}[3]{#1 \Longleftrightarrow #2\textrm { in } #3}

\newcommand{\tlam}{\tau_\lambda}

\newcommand{\Id}{\mathds 1}

\DeclarePairedDelimiter\abs{\lvert}{\rvert}

\DeclarePairedDelimiter\ceil{\lceil}{\rceil}

\DeclarePairedDelimiterX{\inner}[2]{\langle}{\rangle}{#1, #2}

\newcommand{\habs}[1]{\abs*{#1}_\HypTwo}
\newcommand{\habst}[1]{\abs*{#1}_\HypThree}
\newcommand{\habsd}[1]{\abs*{#1}_\HypDim}

\newcommand{\dist}[1]{\mathrm{dist}\left(#1\right)}

\newcommand{\orig}{o}
\newcommand{\e}{\mathrm{e}}

\newcommand{\connf}{\varphi}

\definecolor{darkorange}{RGB}{255,165,0}
\definecolor{altviolet}{RGB}{139,0,139}
\definecolor{turquoise}{RGB}{64,224,208}
\definecolor{lblue}{RGB}{173,216,230}
\definecolor{violet}{RGB}{238,130,238}
\definecolor{darkgreen}{RGB}{0,100,0}
\definecolor{lgreen}{RGB}{144,238,144}

\tikzset{cross/.style={cross out, draw=black, minimum size=2*(#1-\pgflinewidth), inner sep=0pt, outer sep=0pt},
cross/.default={1pt}}

\newcommand{\GregFourEg}{\raisebox{-21.5pt}{
    \begin{tikzpicture}[scale=1]
        \draw (0,0) -- (0.6,0) -- (0.6,0.6);
        \draw (0.6,0) -- (0.6,-0.6);
        \filldraw[fill=blueish!40] (0.6,0.6) circle (3pt) node[left]{$1$};
        \filldraw[fill=blueish!40] (0,0) circle (3pt)node[left]{$2$};
        \filldraw[fill=blueish!40] (0.6,-0.6) circle (3pt) node[left]{$3$};
        \filldraw[fill=redish] (0.6,0) circle (3pt);
    \end{tikzpicture}
}}

\newcommand{\GregFourEgStepOneVOne}{\raisebox{-21.5pt}{
    \begin{tikzpicture}[scale=1]
        \draw (0,0) -- (0.6,0) -- (0.6,0.6);
        \draw (0.6,0) -- (0.6,-0.6) -- (1.2,-0.6);
        \filldraw[fill=blueish!40] (0.6,0.6) circle (3pt) node[left]{$1$};
        \filldraw[fill=blueish!40] (0,0) circle (3pt)node[left]{$2$};
        \filldraw[fill=blueish!40] (0.6,-0.6) circle (3pt) node[left]{$3$};
        \filldraw[fill=blueish!40] (1.2,-0.6) circle (3pt)node[right]{$4$};
        \filldraw[fill=redish] (0.6,0) circle (3pt);
    \end{tikzpicture}
}}

\newcommand{\GregFourEgStepOneVTwo}{\raisebox{-21.5pt}{
    \begin{tikzpicture}[scale=1]
        \draw (0,0) -- (0.6,0) -- (0.6,0.6);
        \draw (1.2,0) -- (0.6,0) -- (0.6,-0.6);
        \filldraw[fill=blueish!40] (0.6,0.6) circle (3pt) node[left]{$1$};
        \filldraw[fill=blueish!40] (0,0) circle (3pt)node[left]{$2$};
        \filldraw[fill=blueish!40] (0.6,-0.6) circle (3pt) node[left]{$3$};
        \filldraw[fill=blueish!40] (1.2,0) circle (3pt)node[right]{$4$};
        \filldraw[fill=redish] (0.6,0) circle (3pt);
    \end{tikzpicture}
}}

\newcommand{\GregFourEgStepTwo}{\raisebox{-21.5pt}{
    \begin{tikzpicture}[scale=1]
        \draw (0,0) -- (0.6,0) -- (0.6,0.6);
        \draw (0.6,0) -- (1.2,0) -- (1.2,-0.6);
        \filldraw[fill=blueish!40] (0.6,0.6) circle (3pt) node[left]{$1$};
        \filldraw[fill=blueish!40] (0,0) circle (3pt)node[left]{$2$};
        \filldraw[fill=blueish!40] (1.2,-0.6) circle (3pt) node[left]{$3$};
        \filldraw[fill=blueish!40] (1.2,0) circle (3pt)node[right]{$4$};
        \filldraw[fill=redish] (0.6,0) circle (3pt);
    \end{tikzpicture}
}}

\newcommand{\GregFourEgStepThree}{\raisebox{-21.5pt}{
    \begin{tikzpicture}[scale=1]
        \draw (0,0) -- (0.6,0) -- (0.6,0.6);
        \draw (0.6,0) -- (1.2,0) -- (1.2,-0.6);
        \draw (1.2,0) -- (1.2,0.6);
        \filldraw[fill=blueish!40] (0.6,0.6) circle (3pt) node[left]{$1$};
        \filldraw[fill=blueish!40] (0,0) circle (3pt)node[left]{$2$};
        \filldraw[fill=blueish!40] (1.2,-0.6) circle (3pt) node[left]{$3$};
        \filldraw[fill=redish] (1.2,0) circle (3pt);
        \filldraw[fill=redish] (0.6,0) circle (3pt);
        \filldraw[fill=blueish!40] (1.2,0.6) circle (3pt) node[right]{$4$};
    \end{tikzpicture}
}}

\newcommand{\GregFourEgStepFour}{\raisebox{-21.5pt}{
    \begin{tikzpicture}[scale=1]
        \draw (0,0) -- (0.6,0) -- (0.6,0.6);
        \draw (0.6,0) -- (0.6,-0.6);
        \filldraw[fill=blueish!40] (0.6,0.6) circle (3pt) node[left]{$1$};
        \filldraw[fill=blueish!40] (0,0) circle (3pt)node[left]{$2$};
        \filldraw[fill=blueish!40] (0.6,-0.6) circle (3pt) node[left]{$3$};
        \filldraw[fill=blueish!40] (0.6,0) circle (3pt) node[right]{$4$};
    \end{tikzpicture}
}}

\numberwithin{equation}{section}
\allowdisplaybreaks

\title{Mean-field behaviour of the random connection model on hyperbolic space}
\author{Matthew Dickson\footnote{University of British Columbia, Department of Mathematics, Vancouver, BC, Canada, V6T 1Z2; Email: dickson@math.ubc.ca; \orcidlink{0000-0002-8629-4796}~https://orcid.org/0000-0002-8629-4796} \and Markus Heydenreich\footnote{Universität Augsburg, Institut für Mathematik, Universitätsstr.\ 2, 86135 Augsburg, Germany; Email: markus.heydenreich@uni-a.de; \orcidlink{0000-0002-3749-7431}~https://orcid.org/0000-0002-3749-7431}
}
 
\date{}

\begin{document}
\maketitle

\vspace{-1em}

{\centering{ \today}\par}

\vskip-3em

\begin{abstract}
We study the random connection model on hyperbolic space $\HypDim$ in dimension $d=2,3$. Vertices of the spatial random graph are given as a Poisson point process with intensity $\lambda>0$. Upon variation of $\lambda$ there is a percolation phase transition: there exists a critical value $\lambda_c>0$ such that for $\lambda<\lambda_c$ all clusters are finite, but infinite clusters exist for $\lambda>\lambda_c$. We identify certain critical exponents that characterise the clusters at (and near) $\lambda_c$, and show that they agree with the mean-field values for percolation. 
We derive the exponents through isoperimetric properties of critical percolation clusters rather than via a calculation of the triangle diagram.
\end{abstract}


\smallskip


{\footnotesize
}

\section{Introduction}

\paragraph{Motivation.}
We study the random connection model on a Borel space $\mathcal S$. 
This is a spatial random graph model, where vertices are given by a locally finite point process $\eta$ on $\Scal$. Edges are sampled randomly for any pair of points $u,v\in\eta$ independently with probability $\connf(u,v)$, where $\connf\colon \Scal\times\Scal\to[0,1]$ is some connectivity function. 
The overarching question that we address is: \emph{how is the geometry of $\Scal$ linked to the connectivity properties of the random connection model?}

We zoom in on the case where $\Scal$ is hyperbolic and $\eta$ is the Poisson point process with respect to the hyperbolic tensor measure. We found it useful to think in terms of the Poincar\'e disc model for $\Scal=\HypTwo$ or Poincar\'e ball model for $\Scal=\HypThree$, but this choice of model is not relevant for the results.

In our main result, we derive critical exponents that describe behaviour of the random connection model \emph{at} the phase transition point. 
We further derive criteria under which such phase transition exists. 
As we shall work out, the value of these critical exponents is intimately related to isoperimetric properties of random sets in the hyperbolic space. 
Interestingly, the same critical exponents show up in high-dimensional Euclidean space, but proof techniques clearly differ.

\paragraph{The model.}
We adopt the Poincar\'e disc model for $\HypTwo$ resp.\ the Poincar\'e ball model for $\HypThree$, both with constant curvature -1, but that specific value is of no importance for our result. 
We write $\orig$ for the origin, and let $\dist{x,y}$ denote \emph{hyperbolic} distance between points $x$ and $y$. The hyperbolic area (resp.\ volume) measure $\dd x$ on the Poincar\'e disc model has metric tensor $\frac{4\|dy\|^2}{(1-\|y\|^2)^2}$, where $y$ is from the unit ball in \emph{Euclidean coordinates}. 
Given a hyperbolic-measurable set $E\subset \HypDim$, let $\habsd{E}=\int_\HypDim\mathbf1_E \,\dd x$ denote the hyperbolic measure of that set. 
A few relations for hyperbolic triangles are listed in Appendix~\ref{app:hyperbolictriangles}. 

Every hyperplane (for $\HypTwo$ these are simply bi-infinite geodesics) separates the space in two half-spaces. Mind that, in contrast to Euclidean spaces, on $\HypDim$ we can have more than two disjoint half-spaces; see Figure~\ref{fig:SteppingStonesmagnetisation} for an example. 

Given $u,v\in\HypDim$, there exists a unique geodesic $\gamma$ that passes through $u$ and $v$. There is then a unique isometry on $\HypDim$, which we call a \emph{translation isometry}, that preserves orientation, maps $\gamma\mapsto\gamma$, and maps $u\mapsto v$. We denote this translation isometry map by $t_{u,v}\colon\HypDim\to\HypDim$. We assume that $\connf\colon \HypDim\times\HypDim\to \left[0,1\right]$ satisfies
\begin{align}
    \connf(x,y) &= \connf(y,x),\\
    \label{eq:translInv}
    \connf(x,y) &= \connf\left(t_{u,v}(x),t_{u,v}(y)\right),
\end{align}
for all $x,y,u,v\in\HypDim$.
However, every even isometry on $\HypDim$ extends to a unique Möbius transformation, and hence invariance of $\connf$ under translation is equivalent to invariance under rotation  or horolation (i.e. ``rotation centred at a boundary point''), cf.~\cite[Exercise~29.13]{martin2012foundations}. In particular, this means that all the $\connf$ we consider are of the form
\begin{equation*}
    \connf(x,y) = \connf\left(\dist{x,y}\right).
\end{equation*}

\medskip
The random connection model is a continuum percolation model. We let $\eta$ be a homogeneous Poisson process on $\HypDim$ with intensity measure $
\lambda \,\dd x$, where $\lambda>0$ is a model parameter. We identify $\eta$ with its support. 
For $x,y\in\HypDim$, we also consider the Palm versions $\eta^x=\eta\cup\{x\}$ and $\eta^{x,y}=\eta\cup\{x,y\}$. 

We consider a spatial random graph whose vertices is given by $\eta$ (resp.\ $\eta^x$ or $\eta^{x,y}$) and edges are inserted independently for each pair of vertices $u$ and $v$ with probability $\connf(u,v)$. 
Adopting notation from \cite{HeyHofLasMat19}, we denote by $\xi$ (resp.\ $\xi^x$ or $\xi^{x,y}$) the entire model that includes the underlying point process as well as the additional randomness required to form edges. 
See Figure~\ref{fig:BooleanSimulation}. 
For Poisson points $u,v\in\eta$ we write $\adja uv{\xi}$ whenever there is an edge between $u$ and $v$, and further write $u\longleftrightarrow v$ whenever $u=v$ or $u$ and $v$ are (graph-)connected through a finite path of consecutive edges in $\xi$. 
Our interest lies in the geometric properties of the connected components. These can best be studied under the Palm measure: 
\[\C\left(\orig,\xi^\orig\right):=\{x\in\eta^\orig:\conn \orig x \xi^\orig\}.\]
This is also called the \emph{cluster} of $\orig$. 
Given a countable set $A$, let $\#A$ denote the number of elements of $A$. The \emph{susceptibility}, or expected size of the cluster of $\orig$ is 
\begin{equation}
    \chi\left(\lambda\right) := \ela\left[\#\C\left(\orig,\xi^{\orig}\right)\right].
\end{equation}
A second quantity is the percolation probability, which is the probability that a designated point, $\orig$ say, lies within an infinitely large cluster: 
\begin{equation}
    \theta\left(\lambda\right) := \pla\left(\#\C\left(\orig,\xi^\orig\right)=\infty\right).
\end{equation}

\begin{figure}
    \centering
    \begin{subfigure}[b]{0.49\textwidth}
        \centering
        \includegraphics[width=\linewidth]{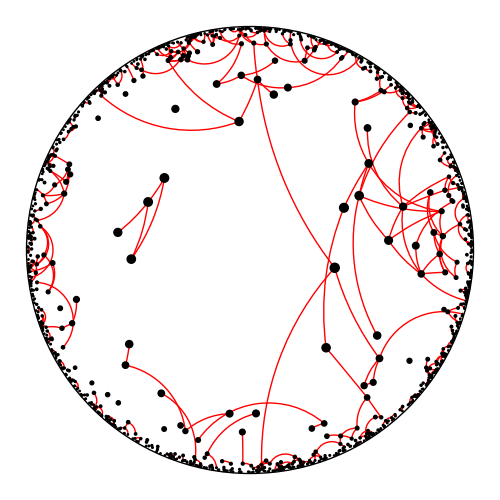}
    \end{subfigure}
    \centering
    \begin{subfigure}[b]{0.49\textwidth}
        \centering
        \includegraphics[width=\linewidth]{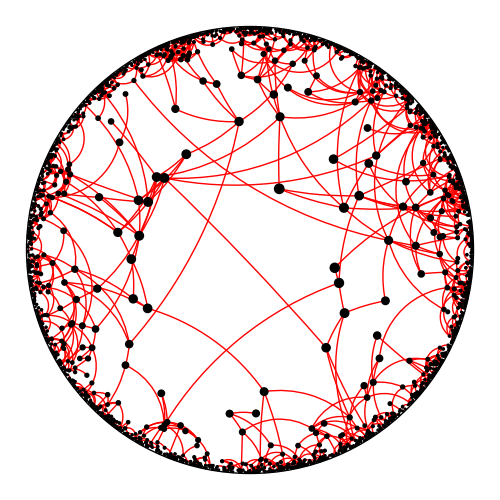}
    \end{subfigure}
    \caption{Simulations of the random connection model on the Poincar{\'e} disc model of $\HypTwo$ with $\connf\left(x,y\right) = \exp\left(-2\dist{x,y}\right)$ and different intensities.}
    \label{fig:BooleanSimulation}
\end{figure}

While we generally make very few assumptions on $\connf$, we do assume throughout that 
\begin{equation}
\label{eqn:longdistanceassumption}
    \lim_{R\to\infty}\esssup_{x\not\in B_R\left(\orig\right)}\connf(x,\orig)<1.
\end{equation}
This avoids a number of technical issues in the proof. 
A prominent example for the choice of $\connf$ is the indicator function $\connf(u,v)=\mathbf1_{\dist{u,v}\le R}$ for certain fixed $R>0$, this gives the \emph{Poisson Boolean model}. It might be depicted by considering a (hyperbolic) ball of radius $R/2$ around each vertex and connecting vertices by an edge whenever their respective balls overlap. 

One consequence of the invariance assumption \eqref{eq:translInv} is that $\int_\HypDim\connf(\orig,x)\dd x >0$ is sufficient to guarantee a certain irreducibility: for all $x,y\in\HypDim$ and $\lambda>0$ we have that the two-point function 
\begin{equation}\label{eq:TwoPt}
    \tlam(x,y):=\pla\big(\conn xy{\xi^{x,y}}\big)
\end{equation} 
satisfies $\tlam(x,y)>0$. See Proposition~\ref{lem:NonTrivialPhaseTrans} for related statements.

\subsection{Results}
We first address the question of a ``non-trivial phase transition'': under which assumptions do we know that there exists a critical density threshold such that there is no percolation below, and there is percolation above such threshold? To this end, we introduce two critical values: 
\begin{align}
    \lambda_T &:= \inf\left\{\lambda>0\colon \chi\left(\lambda\right)=\infty\right\},\\
    \lambda_c &:= \inf\left\{\lambda>0\colon \theta\left(\lambda\right)>0\right\}.
\end{align}
It is clear that $\lambda_T\le \lambda_c$ always, because $\theta\left(\lambda\right)>0$ implies $\chi\left(\lambda\right)=\infty$. 
Our first result is a non-triviality criterion for these critical values: 
\begin{restatable}{prop}{NonTrivialThresholds}
\label{lem:NonTrivialPhaseTrans}
    Let $d\geq 2$ and consider a RCM on $\HypDim$. Then $\lambda_T,\lambda_c<\infty$ if and only if $\int_\HypDim\connf(\orig,x)\dd x >0$, and $\lambda_T,\lambda_c>0$ if and only if $\int_\HypDim\connf(\orig,x)\dd x <\infty$.
\end{restatable}

The behaviour at (and near) the critical intensity $\lambda_c$ (resp.\ $\lambda_T$) is best described through the notion of critical exponents. This goes back a long way, see e.g.\ \cite{AizNew84,BarAiz91,Ngu87}, and \cite{hughes1996random} for physics references. In high-dimensional Euclidean space, the values of these critical exponents agree with the corresponding exponents for branching random walk resp.\ branching jump processes, cf.\ \cite[Section 1.2]{HeyHof17}. Geodesic lines in hyperbolic space separate quickly, and we thus expect the same mean-field values for the critical exponents. The main critical exponents relevant for us are the following: 
\begin{align}
    \gamma&=-\lim_{\lambda\nearrow\lambda_T}\frac{\log \chi(\lambda)}{\log(\lambda_T-\lambda)},\\
    \beta&=\quad\lim_{\lambda\searrow\lambda_c}\frac{\log \theta(\lambda)}{\log(\lambda-\lambda_c)},\\
    \delta&=-\lim_{n\to\infty}\frac{\log n}{\log\p_{\lambda_c}\left(\#\C\left(\orig,\xi^\orig\right)\geq n\right)},\\
    \Delta&=-\lim_{\lambda\nearrow\lambda_T}\frac{\log{\ela\left[\#\C\left(\orig,\xi^\orig\right)^{n+1}\right]}-\log{\ela\left[\#\C\left(\orig,\xi^\orig\right)^{n}\right]}}{\log(\lambda_T-\lambda)}.\label{eqn:defDelta}
\end{align}
The existence of these limits is not a priori clear, and part of the hypothesis (though one might use limit superior instead for definiteness). Also the fact that $\Delta$ is independent of $n$ in \eqref{eqn:defDelta} is an implicit claim. 
Our main result is that, for the RCM on $\HypTwo$ and $\HypThree$, these critical exponents exist (in a fairly strong sense) and take on their mean-field values: 

\begin{theorem}[Critical exponents]
\label{thm:CritExponents}
    Let $d=2,3$ and assume \eqref{eqn:longdistanceassumption} as well as $\int_\HypDim\connf(\orig,x)\dd x \in(0,\infty)$. For the RCM on $\HypDim$, there exist $0<C\leq C'<\infty$ and $\varepsilon>0$ such that: 
    \begin{enumerate}[(a)]
    \item \label{thm:CritExponents - Susceptibility}
    \begin{equation}
        C\left(\lambda_T-\lambda\right)^{-1} \leq \chi(\lambda) \leq C'\left(\lambda_T-\lambda\right)^{-1}
    \end{equation}
    for all $\lambda<\lambda_T$. That is, the critical exponent $\gamma=1$.

    \item \label{thm:CritExponents - Percolation}
    \begin{equation}
        C\left(\lambda-\lambda_c\right)_+ \leq \theta(\lambda) \leq C'\left(\lambda-\lambda_c\right)_+
    \end{equation}
    for all $\lambda<\lambda_c+\varepsilon$. That is, the critical exponent $\beta=1$. Furthermore, $\lambda_c=\lambda_T$.
    \item \label{thm:CritExponents - tailexponent}
    For all $n\in\N$,
    \begin{equation}
        Cn^{-\frac{1}{2}}\leq \p_{\lambda_c}\left(\#\C\left(\orig,\xi^\orig\right)\geq n\right) \leq C'n^{-\frac{1}{2}}.
    \end{equation}
    That is, the critical exponent $\delta=2$.
    \item \label{thm:CritExponents - SusceptibilityMoment}
    For all $n\geq 1$, there exist $0<C_n\leq C'_n<\infty$ such that
    \begin{equation}
    \label{eqn:thmpartd}
        C_n\left(\lambda_c-\lambda\right)^{-2} \leq \frac{\ela\left[\#\C\left(\orig,\xi^\orig\right)^{n+1}\right]}{\ela\left[\#\C\left(\orig,\xi^\orig\right)^{n}\right]} \leq C'_n\left(\lambda_c-\lambda\right)^{-2}
    \end{equation}
    for all $\lambda<\lambda_c$. That is, the critical exponent $\Delta=2$.
    \end{enumerate}
\end{theorem}

The lower bounds on $\chi(\lambda)$ and $\theta(\lambda)$ hold in very great generality. The arguments of \cite{caicedo2023critical} -- phrased on the product space of the Euclidean $\Rd$ with a general mark space -- make no use of the geometry of the ambient space and transfer across directly to our hyperbolic case. The remaining bounds are more subtle and do not hold in such generality, and here we use specific properties of hyperbolic space (in particular related to isoperimetric properties).

Interestingly, the same critical exponents have been obtained recently for the random connection model on high-dimensional \emph{Euclidean space} $\Rd$. Indeed, in \cite{HeyHofLasMat19}, a lace expansion argument was developed to prove the so-called \emph{triangle condition}, and \cite{caicedo2023critical} proved that the triangle condition implies critical exponents very much like in Theorem~\ref{thm:CritExponents}. 
In the present work, we use the geometric properties resulting from negative curvature of the ambient space to prove the critical exponents directly, rather than going via the triangle diagram. 
Our proof is inspired by techniques on hyperbolic graphs by Madras and Wu \cite{madras2010trees}. 
Similar to their result, our theorem is limited to dimensions $2$ and $3$. The reason for this lies in Lemma~\ref{lem:DeterministicBound}, where we provide a deterministic bound on the ratio of the number of surface vertices of convex hyperbolic polytopes to the hyperbolic volume of that region. This is not true \emph{deterministically} for $\HypDim$ with $d\ge4$ (see also Remark~\ref{rem:dge4}). Even though we consider $d=2$ to be the most interesting part of the paper, we see no reason to believe that the main result is limited to these dimensions, and possibly a probabilistic version of Lemma~\ref{lem:DeterministicBound} could be obtained for $d\ge4$. 

An alternative proof approach is laid out in the work of Hutchcroft \cite{Hutch19}, who develops an improved version of Benjamini and Schramm's `magic lemma' \cite{benjaminischramm2001recurrence} and uses it to show mean-field behaviour for bond percolation on quasi-transitive hyperbolic graphs. Similarly to our Lemma~\ref{lem:DeterministicBound}, it is based on the observation that in hyperbolic geometry a typical point in a compact set is close to the boundary of its convex hull. However, a direct application of his Proposition~4.1 would require the graph to be highly regular -- including having an upper bound on the degrees and a positive lower bound on the separation between vertices. One would thus need to get extra control on configurations with clustered vertices as well as dealing with possible long-range connections. Note that the probabilistic control required by this proof approach would be necessary for all $d\geq 2$, not just $d\geq 4$.

The derivation of cluster moment exponents as in \eqref{eqn:thmpartd} \emph{on lattices} proceeds via the tree graph inequality \cite{Ngu87}, and the constants such as $C'_n$ are thus related to tree counting. In contrast, for continuum models labelled vertices in the tree representation may appear as internal vertices as well, and we therefore have to count \emph{Greg trees} instead. For details we refer to Section~\ref{sec:clustermomentexponent}.

\subsection{Related work}
We first summarise the results regarding \emph{graphs} embedded into the hyperbolic plane. Lalley \cite{Lalle98} showed that for site percolation on the dual Dirichlet tiling graph of a co-compact Fuchsian group of genus $d\ge2$, there is a parameter range for which there are infinitely many infinite clusters. This is a remarkable geometric fact, because on amenable spaces the Burton-Keane argument rules out such possibilities. 
Such non-uniqueness phases have been established by Benjamini and Schramm \cite{BenjaSchra96} for nonamenable transitive graphs with one end, and by Pak and Smirnova-Nagniebeda \cite{PakSmirn00} for nonamenable Cayley-graphs. The conjecture that non-uniqueness implies mean-field critical exponents would follow from some conjectures and results in a series of papers by Hutchcroft \cite{Hutch20,Hutch20b,hutchcroft2022derivation}.

In continuum models, Tykesson \cite{Tykes07} proved that such a non-uniqueness phase exists for the Poisson Boolean model on $\HypDim$. Dickson \cite{dickson2024NonUniqueness} extended this not only to general connectivity functions $\varphi$, but covers even inhomogeneous percolation models by admitting vertex weights. His main tool is the spherical transform.
However, both the results of \cite{Tykes07} in dimension $d\ge3$ as well as in \cite{dickson2024NonUniqueness} require sufficient stretched-out properties of the connectivity function $\connf$. 

Direct arguments towards mean-field criticality on hyperbolic graphs without the deviation of non-uniqueness have been established by Schonmann \cite{Schon01,Schon02}. 
An alternative approach has been developed by Madras and Wu \cite{madras2010trees}, and their approach has been inspirational for the present work. In fact, the geometric arguments underlying their approach is arguably coming out more clearly in the continuous setup persuaded in the present paper. On the other hand, the (possible) slow decay of the connectivity function $\connf$ requires involved additional separation arguments, that were not necessary in the graph case. See Section~\ref{sec:mean-field-behaviour}. 
Hutchcroft \cite{Hutch19} developed an alternative approach to mean-field criticality on connected, locally-finite, nonamenable, quasi-transitive Gromov-hyperbolic graphs via a `hyperbolic magical lemma'.

Finally, we remark that in a different line of research, Poisson-Voronoi percolation on $\HypTwo$ has been investigated \cite{BenjaSchra01,HansenMuller_2022_Voronoi}. Here the focus has been on the asymptotic properties of the critical values, but the nature of the phase transition (in terms of critical exponents) has not been derived.

\paragraph{Connection with hyperbolic random graphs.}
Hyperbolic random graphs are a class of spatial random graphs that exhibit many features that are typical for real world networks; this includes scale-free degree distributions, small-world distances and geometric clustering \cite{krioukov2010hyperbolic}. Given a model parameter $\nu$, we sample $n$ points uniformly w.r.t.\ hyperbolic tensor measure on a hyperbolic disc of radius $R=2\log(n/\nu)$. That is, a given point's distance from the origin has density $\frac{\sinh(r)}{\cosh(R)-1}\mathbf{1}_{\{r\le R\}}\,dr$. See Section 9.5.2 of \cite{Hofstad2024random}. 
We then connect two points whenever their hyperbolic distance is at most $R$, i.e. \[\connf(x,y)=\mathbf{1}_{\{\dist{x,y}\le R\}}.\]

There is a link between hyperbolic random graphs and a random graph model known as \emph{scale-free percolation} on Euclidean space; cf.\ \cite{Lodewijks2020Explosion} and \cite[Theorem 9.31]{Hofstad2024random}. Scale-free percolation can be viewed as a weight-dependent random connection model on $\Rd$ with a connection kernel of product form. More generally, hyperbolic random graphs are special instances of geometric inhomogeneous random graphs \cite{BringKeuscLengl19}. We have recently worked out critical exponents for weight-dependent random connection models on $\Rd$ for sufficiently large $d$ \cite{DicHey2022triangle}. However, the case of unbounded weights (as it appears for scale-free percolation) is not covered by the conditions in \cite{DicHey2022triangle}. 

Despite a seeming similarity, the results in this paper do not carry over to hyperbolic random graphs (and their related weight-dependent Euclidean models). The reason is that hyperbolic random graphs, as defined above, do not have a non-trivial  limit within our framework. Indeed as the number of vertices $n$ tends to infinity, the connection function $\connf(x,y)=\mathbf{1}{\{\dist{x,y}\le 2\log(n/\nu)\}}$ goes pointwise to $1$ and simultaneously the vertex density vanishes. Even if we rescale so that the intensity remains constant, the curvature then diverges in the limit.

While results for geometric inhomogeneous random graphs are typically obtained through a careful analysis of vertex weights, the methods in the present paper is genuinely geometric.

\subsection{Outline of the Paper}
\label{sec:Outline}
Section~\ref{sec:NonTrivialPhaseTransition} is concerned with proving Proposition~\ref{lem:NonTrivialPhaseTrans}. This is a short section, but allows us to proceed with confidence that the $\lambda_T$ and $\lambda_c$ we are talking about are indeed strictly positive and finite numbers.

In Section~\ref{sec-MFbounds} we recap the arguments that prove the susceptibility and percolation mean-field lower bounds. These bounds are very general and hold outside our hyperbolic setting. The general form of the arguments first appeared for Bernoulli bond percolation in \cite{AizNew84,AizBar87}, and have previously been applied to random connection models on $\Rd$ in \cite{HeyHofLasMat19,caicedo2023critical}. The proof for the percolation lower bound also introduces the magnetisation function, which will be the main tool by which we will access the percolation function and -- later -- the cluster tail distribution.

To find the remaining bounds in Theorem~\ref{thm:CritExponents}, our main tool will be to restrict clusters to particular half-spaces.
In Proposition~\ref{prop:BoundHalf-SpaceSusceptMagnet} we state that this restriction has only a bounded impact on the susceptibility and magnetisation. 
In the remaining parts of Section~\ref{sec:mean-field-behaviour} we take these bounds on the restricted susceptibility and magnetisation and derive from them upper bounds on the susceptibility and percolation functions, and complementary upper and lower bounds on the cluster tail distribution and the higher moments of the cluster size. 
In Section~\ref{sec:SusceptibilityUpperBound} the space is split into two distant half-spaces to derive a lower bound on the derivative of the susceptibility in terms of two factors of the susceptibility, while in Section~\ref{sec:PercolationUpperBounds} the space is split into three distant half-spaces to derive a lower bound for the magnetisation in terms of two factors of the magnetisation and one factor of the susceptibility. This bound is then used to derive an upper bound for the percolation function. We further use the upper bound on the magnetisation to derive an upper bound for the cluster tail distribution, while the lower bound on the magnetisation from earlier provides the lower bound on the cluster tail distribution. Section~\ref{sec:clustermomentexponent} uses the lower bound on the derivative of the susceptibility to derive both the upper and lower bounds on the moments of the cluster size.

Finally, in Section~\ref{sec:clusterHalfspaceProofs} we prove Proposition~\ref{prop:BoundHalf-SpaceSusceptMagnet}. It is in Lemma~\ref{lem:DeterministicBound} that the requirement that $d=2,3$ comes in.

\section{Non-triviality of the phase transition}
\label{sec:NonTrivialPhaseTransition}

In the following lemma we construct a \emph{deterministic} binary tree embedded in $\HypDim$. In the proof of Proposition~\ref{lem:NonTrivialPhaseTrans} that follows we then couple a subgraph of the percolation cluster of the RCM with Bernoulli bond percolation on this tree. 

\begin{lemma}
\label{lem:splittingLength}
    Let $A,B\in\HypTwo$ be such that $A$, $B$, and $\orig$ are not co-linear, and let $\theta$ denote the angle at $\orig$ subtended by $A$ and $B$. Also define
    \begin{equation}
    \label{eqn:avoidancelength}
        L\left(\theta\right) := \arcosh\left(\frac{1-\cos\theta\cos\frac{\theta}{2}}{\sin\theta\sin\frac{\theta}{2}}\right).
    \end{equation}
    If $\dist{\orig,A}>L(\theta)$ and $\dist{\orig,B}>L(\theta)$, then 
    the Cayley tree, $(V,E)$, in $\HypTwo$ generated by the translations $t_{\orig,A}$ and $t_{\orig,B}$
    is such that 
    $\inf_{u,v\in V}\dist{u,v}>0$.
\end{lemma}

\begin{proof}
    Let $\gamma_1$ and $\gamma_2$ be two geodesic curves passing through $\orig$ at angle $\theta$ to each other and bounding a cone in $\HypTwo$. At hyperbolic length $l$ along $\gamma_1$ we branch off another geodesic curve $\gamma_3$ into the cone at angle $\theta$ to $\gamma_1$, and at hyperbolic length $l$ along $\gamma_2$ we branch off another geodesic curve $\gamma_4$ into the cone at angle $\theta$ to $\gamma_2$. If $l$ is too small then $\gamma_3$ and $\gamma_4$ will intersect inside the cone, but if $l$ is sufficiently large then then will not. We first want to find the transition length $l=L(\theta)$, at which $\gamma_3$ and $\gamma_4$ just touch at the boundary. This critical case is sketched in the Poincar{\'e} disc model in Figure~\ref{fig:Cone-Splitting}. The hyperbolic quadrilateral bounded by $\gamma_1$, $\gamma_2$, $\gamma_3$, and $\gamma_4$ can be partitioned into two hyperbolic triangles along the radial line of reflection symmetry. We then have a triangle with internal angles $\frac{\theta}{2}$, $\pi-\theta$, and $0$, whose opposite side lengths are $\infty$, $\infty$, and $L(\theta)$ respectively. By applying the second cosine rule for hyperbolic triangles (see Appendix~\ref{app:hyperbolictriangles}), we get
    \begin{equation}
        \cosh\left(L(\theta)\right)\sin\left(\pi-\theta\right)\sin\frac{\theta}{2} = 1 + \cos\left(\pi-\theta\right)\cos\frac{\theta}{2}.
    \end{equation}
    This can be rearranged to get the expression for $L(\theta)$ above.

    When we embed the binary tree generated by $A$ and $B$ in $\HypTwo$, let $\gamma_1$ be the geodesic passing through $\orig$ and $A$, and $\gamma_2$ be the geodesic passing through $\orig$ and $B$. Then let $\gamma_3$ split off at $A$ and $\gamma_4$ split off at $B$. Every descendent of $A$ will then be contained in the cone bounded by $\gamma_1$ and $\gamma_3$, and every descendent of $B$ will then be contained in the cone bounded by $\gamma_2$ and $\gamma_4$. If $\dist{\orig,A}>L(\theta)$ and $\dist{\orig,B}>L(\theta)$ then these cones cannot intersect each other and the vertices in each cone cannot get close to the vertices in the other cone. We now only need to consider vertices in the same cone, which must be descendants and ancestors of each other. If we consider $\orig$, we see that it is not in either of the cones descended from $A$ or $B$ and therefore there is a uniform lower bound on how close a descendent can be. From the translation invariance, this can be repeated for every vertex to show that no vertex has an arbitrarily close descendent. The argument can also be run in reverse to show that there is a lower bound on how close an ancestor is.
\end{proof}

\begin{figure}
    \centering
    \begin{tikzpicture}[scale=2]
        \draw (5,0) arc (0:60:5);
        \draw[very thick] (5,0) -- (0,0) -- (2.5,4.33);
        \draw (4.33,2.5) arc(120:150:6.83);
        \draw (4.33,2.5) arc(300:270:6.83);
        \filldraw (0,0) circle (1pt);
        \filldraw (1.83,0) circle (1pt);
        \filldraw (0.915,1.585) circle (1pt);
        \draw (0.3,0) node[above left]{$\theta$} arc(0:60:0.3);
        \draw (2.13,0) node[above left]{$\theta$} arc(0:60:0.3);
        \draw (1.215,1.585) node[above left]{$\theta$} arc(0:60:0.3);
        \draw[<->] (0,-0.2) -- (1.83,-0.2);
        \draw (0.915,-0.1) node{$L(\theta)$};
        \draw (0.915,0.2) node{$\gamma_1$};
        \draw (0.45,0.7) node[above left]{$\gamma_2$};
        \draw (2.7,0.8) node{$\gamma_3$};
        \draw (2.2,2) node{$\gamma_4$};
    \end{tikzpicture}
    \caption{Sketch showing construction of the length \eqref{eqn:avoidancelength} in the Poincar{\'e} disc model of $\HypTwo$.}
    \label{fig:Cone-Splitting}
\end{figure}

We are now prepared to prove the non-triviality of the phase transition.

\NonTrivialThresholds*

\begin{proof}
    If $\int_{\HypDim}\connf\left(x,\orig\right)\dd x =\infty$, then we can write $\connf=\sum^\infty_{i=1}\connf_i$ where $\connf_i$ all satisfy our assumptions of an adjacency function and $\int_{\HypDim}\connf_i\left(x,\orig\right)\dd x =1$ for all $i\geq 1$. Fix $\lambda>0$ and let the random variable $X_m$ denote $\#\C\left(\orig,\xi^\orig\right)$ for the RCM with intensity $\lambda$ and connection function $\sum^m_{i=1}\connf_i$, and let $Y_m$ denote the number of neighbours of $\orig$ in the same model. For finite $m$, the $Y_m$ is a Poisson random variable with mean $\lambda\sum^m_{i=1}\int_\HypDim\connf_i(x,\orig)\dd x = \lambda m$ (by Mecke's formula; \cite[(6.12)]{LasZie17}). Therefore for all $n\geq 0$ we have $\p\left(Y_m = n\right)\to 0$ as $m\to\infty$. Since $X_m\geq Y_m$ for all configurations and $m$ by monotonicity, we have $\p\left(X_m = n\right)\to 0$ as $m\to\infty$. Now let $X$ denote $\#\C\left(\orig,\xi^\orig\right)$ for the RCM with intensity $\lambda$ and full connection function $\connf$. Since $X\geq X_m$ for all configurations and $m$, we have $\p\left(X=n\right)=0$ for all $n\geq 1$. Therefore $\p\left(X=\infty\right)=1$ and $\lambda\geq \lambda_c$. Since this argument only required $\lambda>0$, we have $\lambda_c=0$. Since $\lambda_T\leq \lambda_c$, we also have $\lambda_T=0$.

    If $\int\connf(\orig,x)\dd x<\infty$, then we can use the ``method of generations'' and compare to a branching process to show $\lambda_T>0$. We first let $\orig$ be the root of our cluster (generation $0$), and the vertices adjacent to $\orig$ (generation $1$) are distributed according to a Poisson point process with intensity density $\lambda\connf\left(\cdot, \orig\right)$ with respect to the hyperbolic measure. For $k\geq2$, we then iteratively define generation $k$ to be those vertices adjacent to a vertex in generation $k-1$, but no vertex from an earlier generation. If we assign an arbitrary ordering to each generation, then we can also assign a genealogical tree to the vertices. Given a vertex from generation $k$, we can call the vertices in generation $k+1$ immediate descendants of this vertex if they are adjacent to this vertex but from no earlier vertex in generation $k$. The immediate descendants of a vertex $v$ are then distributed according to a thinned Poisson point process whose intensity density is bounded above by $\lambda\connf\left(\cdot,v\right)$. We can therefore bound the two-point function, recall \eqref{eq:TwoPt}, by the kernel of a spatial branching process:
    \begin{equation}
        \tlam\left(x,\orig\right) \leq \connf\left(x,\orig\right) + \sum^\infty_{k=1}\lambda^k\int_{\left(\HypDim\right)^k}\prod^{k+1}_{i=1}\connf\left(v_{i},v_{i-1}\right)\dd v_{\left[1:k\right]},
    \end{equation}
    where $v_{0}=\orig$ and $v_{k+1}=x$. From Tonelli's theorem and the translation symmetry of $\connf$, we then have
    \begin{equation}
        \lambda\int_\HypDim\tlam\left(x,\orig\right)\dd x \leq \sum^\infty_{k=1}\lambda^k\left(\int_{\HypDim}\connf\left(v,\orig\right)\dd v\right)^{k}.
    \end{equation}
    Then Mecke's formula implies that $\chi\left(\lambda\right)= 1+ \lambda\int_\HypDim\tlam\left(x,\orig\right)\dd x$, and for $\lambda<\frac{1}{\int_{\HypDim}\connf\left(v,\orig\right)\dd v}$ we have
    \begin{equation}
        \chi\left(\lambda\right) \leq \frac{1}{1-\lambda\int_{\HypDim}\connf\left(v,\orig\right)\dd v}.
    \end{equation}
    Therefore for sufficiently small $\lambda$ we have $\chi(\lambda)<\infty$ and therefore $\lambda_T>0$ (and $\lambda_c>0$). Note that this also proves that $\lambda_T=\infty$ if $\int_{\HypDim}\connf\left(x,\orig\right)\dd x =0$, and furthermore $\lambda_c=\infty$ if $\int_{\HypDim}\connf\left(x,\orig\right)\dd x =0$.

    For the reverse direction, we adapt an argument of Penrose \cite{Pen91} to the hyperbolic setting. 
    If $\int\connf(\orig,x)\dd x>0$, then there exist $a,b\in\HypDim$ and $\delta,\varepsilon>0$ such that $a$, $b$, and $\orig$ are not co-linear, the sets $B_\delta(a)$, $B_\delta(b)$, and $B_\delta(\orig)$ are disjoint, and 
    \begin{equation}
        \essinf_{x\in B_\delta(\orig),y\in B_\delta(a)\cup B_\delta(b)}\connf(x,y)>\varepsilon.
    \end{equation}
    Let $\theta$ be the angle subtended at $\orig$ by $a$ and $b$. Then define
    \begin{equation}
        L\left(\theta\right) := \arcosh\left(\frac{1-\cos\theta\cos\frac{\theta}{2}}{\sin\theta\sin\frac{\theta}{2}}\right),
    \end{equation}
    and $m:=\ceil*{\frac{L(\theta)}{\dist{\orig,a}}}+1$ and $n:=\ceil*{\frac{L(\theta)}{\dist{\orig,b}}}+1$. Then we can define $A:= t_{\orig,a}^m(\orig)$ as the $m$-fold application of $t_{\orig,a}$ to $\orig$, and similarly $B:= t_{\orig,b}^n(\orig)$. Note that $\dist{\orig,A}>L(\theta)$ and $\dist{\orig,B}>L(\theta)$. Since we have these bounds, Lemma~\ref{lem:splittingLength} proves that $A$ and $B$ generate a binary tree $T=(V,E)$ embedded in $\HypDim$ for which the $\HypDim$-distances between vertices do not become arbitrarily small. We are then free to choose $\delta>0$ small enough that $\left\{B_\delta(v)\right\}_{v\in V}$ are disjoint.
    
    By using Mecke's formula, given $u\in B_\delta(\orig)$, the number of vertices in $B_\delta(a)$ that are adjacent to $u$ is a Poisson random variable with mean:
\begin{equation}
    \E_{\lambda}\left[\#\left\{y\in\eta\cap B_\delta(a) \colon y \sim u\right\}\right] = \lambda \int_{B_\delta(a)} \connf\left(x,u\right)\dd x \geq \lambda \varepsilon \habs{B_\delta(a)} = 4\pi\lambda\varepsilon\left(\sinh \delta\right)^2.
\end{equation}
It then follows that the probability that there exists such a neighbour has the lower bound
\begin{equation}
    \pla\left(\exists v\in\eta\cap B_\delta(a) \colon v \sim u\right) \\= 1 - \exp\left(-\E_{\lambda}\left[\#\left\{y\in\eta\cap B_\delta(a) \colon y \sim u\right\}\right]\right) \geq 1- \e^{-4\pi\lambda\varepsilon\left(\sinh \delta\right)^2}.
\end{equation}
In particular, this bound is uniform in the vertex $u\in B_\delta(\orig)$. From the translation invariance and independence properties of the model, we can repeat this argument $m$ times to show that the probability that there exists a path from $u\in B_\delta(\orig)$ to some vertex in $B_\delta(A)$ is bounded below by
\begin{equation}
    \pla\left(\exists v\in\eta\cap B_\delta(A) \colon \conn{v}{u}{\xi^u}\right) \geq \left(1- \e^{-4\pi\lambda\varepsilon\left(\sinh \delta\right)^2}\right)^m.
\end{equation}
Similarly, we can find that
\begin{equation}
    \pla\left(\exists v\in\eta\cap B_\delta(B) \colon \conn{v}{u}{\xi^u}\right) \geq \left(1- \e^{-4\pi\lambda\varepsilon\left(\sinh \delta\right)^2}\right)^n.
\end{equation}

If $\min\left\{\left(1- \e^{-4\pi\lambda\varepsilon\left(\sinh \delta\right)^2}\right)^m,\left(1- \e^{-4\pi\lambda\varepsilon\left(\sinh \delta\right)^2}\right)^m\right\}>\frac{1}{2}$, then the cluster $\C(\orig,\xi^\orig)$ will have a positive probability of having a vertex in infinitely many of the sets $\left\{B_\delta(v)\right\}_{v\in V}$. That is, if $\lambda>\frac{-1}{4\pi\varepsilon\left(\sinh\delta\right)^2}\log\left(1-2^{-\frac{1}{\max\left\{m,n\right\}}}\right)$ then $\theta\left(\lambda\right)>0$. Therefore
\begin{equation}
    \lambda_c \leq \frac{-1}{4\pi\varepsilon\left(\sinh\delta\right)^2}\log\left(1-2^{-\frac{1}{\max\left\{m,n\right\}}}\right)<\infty.
\end{equation}
In turn this also proves that $\lambda_T<\infty$.
\end{proof}

\section{Mean-Field Bounds}
\label{sec-MFbounds}

The following mean-field bounds are expected to hold in great generality. In particular, the proofs make no use of the hyperbolic structure of the ambient space. They also apply for flat Euclidean spaces such as $\Rd$ (see \cite{HeyHofLasMat19,caicedo2023critical}). The arguments are adapted from the arguments for transitive Bernoulli bond percolation models in \cite{AizNew84,AizBar87}.

\subsection{Susceptibility Lower Bound}

The proof of the mean field lower bound for the susceptibility can be found in \cite{HeyHofLasMat19}. The proof is worded for the RCM on $\Rd$, but the argument for the RCM on $\HypTwo$ proceeds identically.

The key points are in the following lemma.
\begin{lemma}
    The map $\lambda\mapsto \frac{1}{\chi\left(\lambda\right)}$ is continuous on $\left[0,\infty\right)$, and for all $\lambda\in\left(0,\lambda_T\right)$
    \begin{equation}\label{eq:chi2lowerbd}
        \frac{\dd \chi}{\dd \lambda}\left(\lambda\right) \leq \frac{1}{\lambda}\chi\left(\lambda\right)^2.
    \end{equation}
\end{lemma}

\begin{proof}
    See \cite[Lemma~2.3 and Proof of Theorem~1.3]{HeyHofLasMat19}. The key ingredients are a Margulis-Russo formula and the BK inequality. The continuity result follows from showing that a sequence of truncated versions is equicontinuous. 
\end{proof}

This can be used to prove the claim in Theorem~\ref{thm:CritExponents}\ref{thm:CritExponents - Susceptibility} that there exists $C>0$ such that
\begin{equation}
    \chi\left(\lambda\right) \geq C\left(\lambda_T-\lambda\right)^{-1}
\end{equation}
for $\lambda<\lambda_T$.

\begin{proof}[Proof of the lower bound in Theorem~\ref{thm:CritExponents}\ref{thm:CritExponents - Susceptibility}]
    The continuity of $\lambda\mapsto \frac{1}{\chi\left(\lambda\right)}$ implies that $\frac{1}{\chi\left(\lambda_T\right)}=0$, and therefore for $\lambda<\lambda_T$
    \begin{equation}
        -\frac{1}{\chi\left(\lambda\right)} = \frac{1}{\chi\left(\lambda_T\right)}-\frac{1}{\chi\left(\lambda\right)} = \int^{\lambda_T}_\lambda \frac{\dd}{\dd s}\left(\frac{1}{\chi(s)}\right)\dd s = -\int^{\lambda_T}_\lambda \frac{1}{\chi(s)^2}\frac{\dd \chi}{\dd s}\left(s\right)\dd s \leq -\int^{\lambda_T}_\lambda\frac{1}{s}\dd s = -\log\frac{\lambda_T}{\lambda}.
    \end{equation}
    Therefore $\chi\left(\lambda\right) \geq \frac{1}{\log \lambda_T - \log \lambda}$, and by taking a Taylor approximation of $\log \lambda$ at $\lambda_T$ there exist $C>0$ and $\varepsilon>0$ such that 
    \begin{equation}
        \chi\left(\lambda\right) \geq C\left(\lambda_T-\lambda\right)^{-1}
    \end{equation}
    for $\lambda\in\left(\lambda_T-\varepsilon,\lambda_T\right)$ as required.
\end{proof}

\subsection{Percolation Lower Bound}

The proof of the mean field lower bound for the percolation proceeds as in \cite{caicedo2023critical}. The proof there is given for the marked RCM on $\Rd$, but the argument proceeds identically for $\HypDim$, and reduces to our case if the set of marks is a singleton.

\begin{definition}
The following proposition provides bounds for the susceptibility and \emph{magnetisation} functions. The magnetisation function will be important when we will use it to derive bounds on the percolation function. Given a new parameter $q\in\left[0,1\right]$, we label each vertex in $\eta$ as a ``ghost'' vertex independently with probability $q$. Note that we take the convention that when we augment a configuration with additional vertices, these are also independently labelled as a ghost vertex with probability $q$. The set of ghost vertices is then denoted $\Gcal$. Then the magnetisation function is defined by
\begin{equation}
    M\left(\lambda,q\right) := \plaq\left(\conn{\orig}{\Gcal}{\xi^\orig}\right),
\end{equation}
and the `ghost free' susceptibility is defined by
\begin{equation}
    \chi\left(\lambda,q\right) := \elaq\left[\#\C\left(\orig,\xi^{\orig}\right)\Id_{\left\{\C\left(\orig,\xi^{\orig}\right)\cap \Gcal = \emptyset\right\}}\right].
\end{equation}
\end{definition}

The following lemma provides some preliminary properties of $M\left(\lambda,q\right)$ and $\chi\left(\lambda,q\right)$.
\begin{lemma}
\label{lem:MagnetPrelim}
    For all $\lambda\geq 0$, the function $q\mapsto M\left(\lambda,q\right)$ is analytic on $\left(0,1\right)$,
    \begin{align}
        \lim_{q\searrow0}M\left(\lambda,q\right) &= \theta\left(\lambda\right),\\
        \left(1-q\right)\frac{\partial M}{\partial q}\left(\lambda,q\right) &= \chi\left(\lambda,q\right), \qquad\forall q\in\left(0,1\right).\label{eqn:magnetisationDerivative}
    \end{align}
    Furthermore, for all $q\in\left(0,1\right]$ the function $\lambda\mapsto M\left(\lambda,q\right)$ is analytic on $\left(0,\infty\right)$.
\end{lemma}

\begin{proof}
    These are proven in \cite[Lemmas~3.1, 3.4, 3.7]{caicedo2023critical}.
\end{proof}

The following lemma introduces two differential inequalities that control $M\left(\lambda,q\right)$. These are used in the proofs of Lemmas~\ref{lem:magnetlowerbound} and \ref{lem:protoPercolationLowerbound} below, and in Section~\ref{sec:PercolationUpperBounds} later.
\begin{lemma}
    Let $q\in\left(0,1\right)$ and $\lambda>0$. Then
    \begin{align}
        \frac{\partial M}{\partial \lambda}\left(\lambda,q\right) &\leq \frac{1-q}{\lambda}M\left(\lambda,q\right)\frac{\partial M}{\partial q}\left(\lambda,q\right) \label{eqn:magnetisationDiffEqnPtOne}\\
        M\left(\lambda,q\right) &\leq q\frac{\partial M}{\partial q}\left(\lambda,q\right) + M\left(\lambda,q\right)^2 + \lambda M\left(\lambda,q\right)\frac{\partial M}{\partial \lambda}\left(\lambda,q\right).
    \end{align}
\end{lemma}

\begin{proof}
    See \cite[Lemma~3.6]{caicedo2023critical}.
\end{proof}

Now define $\chi^\mathrm{f}$ to be the susceptibility conditioned on the cluster of the origin being finite:
\begin{equation}
    \chi^{\rm f}\left(\lambda\right) := \ela\left[\#\C\left(\orig,\xi^\orig\right)\mid \#\C\left(\orig,\xi^\orig\right)<\infty\right].
\end{equation}

\begin{lemma}
\label{lem:magnetlowerbound}
    Let $q\in\left(0,1\right)$, and $\lambda^*>0$ be such that $\chi^{\rm f}\left(\lambda^*\right)=\infty$. Then
    \begin{equation}
        M\left(\lambda^*,q\right)\geq \frac{1}{\sqrt{2}}q^{\frac{1}{2}}.
    \end{equation}
\end{lemma}

\begin{proof}
    See \cite[Corollary~3.8]{caicedo2023critical}.
\end{proof}

\begin{lemma}
\label{lem:protoPercolationLowerbound}
    If $\lambda^*>0$ satisfies $\theta(\lambda^*)=0$ and $\chi^{\rm f}\left(\lambda^*\right)=\infty$, then
    \begin{equation}
        \theta(\lambda) \geq \frac{1}{2\lambda}\left(\lambda-\lambda^*\right)_+.
    \end{equation}
\end{lemma}

\begin{proof}
    See \cite[Proposition~4.3]{caicedo2023critical}.
\end{proof}

These lemmas lead to a proof of the claim in Theorem~\ref{thm:CritExponents}\ref{thm:CritExponents - Percolation} that there exist $C>0$ and $\varepsilon>0$ such that
\begin{equation}
    \theta\left(\lambda\right) \geq C\left(\lambda-\lambda_c\right)_+
\end{equation}
for $\lambda<\lambda_c+\varepsilon$.

\begin{proof}[Proof of the lower bound in Theorem~\ref{thm:CritExponents}\ref{thm:CritExponents - Percolation}]
We first aim to prove the claim that there exists $C>0$ and $\varepsilon>0$ such that
\begin{equation}
    \theta\left(\lambda\right) \geq C\left(\lambda-\lambda_T\right)_+
\end{equation}    
for $\lambda< \lambda_T+\varepsilon$. First note that by Theorem~\ref{thm:CritExponents}\ref{thm:CritExponents - Susceptibility} we know that $\chi\left(\lambda_T\right)=\infty$. If $\theta\left(\lambda_T\right)=0$ then this implies $\chi^\mathrm{f}\left(\lambda_T\right) = \chi\left(\lambda_T\right)=\infty$, and Lemma~\ref{lem:protoPercolationLowerbound} proves the claim. On the other hand, if $\theta\left(\lambda_T\right)>0$ then the claim immediately follows for some sufficiently small choice of $\varepsilon$.

Now note that if $\lambda>\lambda_T$ we have proven that $\theta\left(\lambda\right)>0$ and therefore $\lambda\geq \lambda_\mathrm{c}$. Since we clearly have $\lambda_\mathrm{c}\geq \lambda_T$ from their definitions, this proves $\lambda_\mathrm{c}=\lambda_T$ and therefore that there exist $C>0$ and $\varepsilon>0$ such that
\begin{equation}
    \theta\left(\lambda\right) \geq C\left(\lambda-\lambda_\mathrm{c}\right)_+
\end{equation}    
for $\lambda< \lambda_\mathrm{c}+\varepsilon$ as required.
\end{proof}

\section{Mean-Field Behaviour}
\label{sec:mean-field-behaviour}
While the bounds in Section~\ref{sec-MFbounds} are valid in great generality (e.g., also in Euclidean setting), we now prove the corresponding upper bounds, and here we are crucially relying on hyperbolic geometry.

\subsection{Clusters in half spaces}
\label{sec:ClustersHalfSpace}
    \begin{definition}
        Given a set $S\subset \HypDim$, let $h_\varepsilon\left(S\right)$ denote the $\varepsilon$-halo of $S$. That is, $h_\varepsilon\left(S\right) = \cup_{x\in S}B_\varepsilon(x)$. Also let $\convex{S}$ denote the convex hull of $S$. That is, $\convex{S}$ equals the intersection of all \emph{closed} half-spaces that contain $S$.

    Let $H\subset \HypDim$ be a closed half-space with boundary hyperplane $\partial H\subset \HypDim$, and let $x\in H$. For all $\varepsilon\in\left(0,\dist{\partial H,x}\right)$, let $\Mcal_\varepsilon\left(H,x\right)$ denote the union of all half spaces, $M\subset \HypDim$, such that $x\in M\subset H$ and $\dist{\partial M,x}<\varepsilon$:
    \begin{equation}
        \Mcal_{\varepsilon}\left(H,x\right) := \bigcup_{M\colon x\in M\subset H, \dist{\partial M,x}<\varepsilon} M.
    \end{equation}
    Let $x_\perp\in\partial H$ denote the (unique) nearest point in $\partial H$ to $x$. Clearly $\Mcal_{\varepsilon}\left(H,x\right)$ is symmetric under rotations about the geodesic that passes through $x$ and $x_\perp$. Let $\theta_\varepsilon\left(H,x\right)$ denote the angle this geodesic forms with the boundary of $\Mcal_{\varepsilon}\left(H,x\right)$, and let $m_\varepsilon(H,x)$ denote the point at which the geodesic and the boundary meet.
    See Figure~\ref{fig:PulledHalfSpace}.
    
    The following lemma presents some expressions for $m_\varepsilon = m_\varepsilon(H,x)$ and $\theta_\varepsilon = \theta_\varepsilon\left(H,x\right)$ that can be derived from hyperbolic trigonometric rules (the tangent, sine and cosine rules can be found in Appendix~\ref{app:hyperbolictriangles}). The proofs are omitted here because the relations are not required for our arguments other that in noting that for $x\not\in \partial H$ and $\varepsilon\in\left(0,\dist{\partial H,x}\right)$ we have $\theta_\varepsilon\left(H,x\right)\in\left(0,\pi\right]$ - a result that is ``clear'' from Figure~\ref{fig:PulledHalfSpace}.

    \end{definition}

    \begin{figure}
        \centering
        \begin{tikzpicture}
            \begin{scope}
            \clip (0,0) circle (5);
            \fill[black!10] (-10,0)  circle (8.66);
            \draw (-10,0)  circle (8.66);
            \end{scope}
            \draw (0,0) circle (5);
            \fill[black!30] (0,0) -- (-2.5,-4.33) arc(-120:120:5) -- (0,0);
            \draw[dashed] (-5,0) -- (5,0);
            \draw (-2.5,-4.33) -- (0,0)node[above left]{$m_{\varepsilon}$} -- (-2.5,4.33);
            \draw[dashed] (2.5,4.33) -- (0,0) -- (2.5,-4.33);
            \draw (-1.34,0) node[above left]{$x_\perp$};
            \draw (-3,2) node{$H^\mathrm{c}$};
            \filldraw (0,0) circle (2pt);
            \filldraw (-1.34,0) circle (2pt);
            \filldraw (2.5,0) circle (2pt) node[above right]{$x$};
            \draw (1.2,0) node[above]{$\theta_\varepsilon$};
            \draw (1.56,0.34) arc(160:180:1);
            \draw[<->] (2.5,0) arc(253:237.5:8);
            \draw (1.5,0.7) node{$\varepsilon$};
            \draw (2.5,2) node{$\Mcal_\varepsilon\left(H,x\right)$};
            \end{tikzpicture}
        \caption{Sketch of the set $\Mcal_\varepsilon(H,x)$ in the Poincar{\'e} Disc model for $\HypTwo$.}
        \label{fig:PulledHalfSpace}
    \end{figure}

\begin{lemma}
\label{lem:distancecalculation}
    For all half-spaces $H\subset \HypDim$, $x\in H$, and $\varepsilon\in\left(0,\dist{\partial H,x}\right)$,
    \begin{align}
        \dist{\partial H,m_\varepsilon( H,x)} &= \dist{\partial H,x} - \frac{1}{2}\log\left(\frac{1+\e^{\dist{\partial H,x}}\sinh \varepsilon}{1-\e^{-\dist{\partial H,x}}\sinh\varepsilon}\right),\\
        \theta_\varepsilon\left(\partial H,x\right) &= \arccos\left(\frac{\tanh\varepsilon}{\tanh\left(\dist{\partial H,x}-\dist{\partial H,m_\varepsilon( H,x)}\right)}\right).
    \end{align}
\end{lemma}

Recall the \emph{susceptibility} and \emph{magnetisation} functions:
\begin{equation}
    \chi\left(\lambda\right) = \ela\left[\#\C\left(\orig,\xi^\orig\right)\right],\qquad M\left(\lambda,q\right) = \plaq\left(\conn{\orig}{\Gcal}{\xi^\orig}\right).
\end{equation}
Clearly, if we take the susceptibility or magnetisation and introduce the indicator function $\Id_{\left\{\C\left(\orig,\xi^{\orig}\right)\subset H\right\}}$ for some half-space $H$ into the expectation or probability, then the original susceptibility (resp. magnetisation) is an upper bound for our new restricted susceptibility (resp. magnetisation). The following proposition shows that in fact the reverse inequalities are also true at the cost of only a $\lambda$-independent factor (in the sub-critical regime). It also shows corresponding statements for `ghost free' susceptibility.

\begin{prop}\label{prop:BoundHalf-SpaceSusceptMagnet}
    Let $d=2,3$ and $H\ni \orig$ be a half-space such that $\dist{\partial H,\orig}>0$. Then there exists $\kappa=\kappa\left(d,\connf,\dist{\partial H,\orig}\right)\in\left(0,\infty\right)$ such that for all $\lambda\in\left[0,\lambda_c\right)$ and $q\in\left(0,1\right)$,
    \begin{align}
        \ela\left[\#\C\left(\orig,\xi^{\orig}\right)\right] &\leq \kappa\ela\left[\#\C\left(\orig,\xi^{\orig}\right) \Id_{\left\{\C\left(\orig,\xi^{\orig}\right)\subset H\right\}}\right]\\
        \elaq\left[\#\C\left(\orig,\xi^{\orig}\right)\Id_{\left\{\C\left(\orig,\xi^{\orig}\right)\cap \Gcal = \emptyset\right\}}\right] &\leq \kappa\elaq\left[\#\C\left(\orig,\xi^{\orig}\right)\Id_{\left\{\C\left(\orig,\xi^{\orig}\right)\cap \Gcal = \emptyset\right\}} \Id_{\left\{\C\left(\orig,\xi^{\orig}\right)\subset H\right\}}\right]\\
        \plaq\left(\conn{\orig}{\Gcal}{\xi^\orig}\right) &\leq \kappa\plaq\left(\conn{\orig}{\Gcal}{\xi^\orig},\C\left(\orig,\xi^{\orig}\right)\subset H\right).
    \end{align}
\end{prop}

\begin{remark}
    Even though the actual value of $\kappa$ is of no importance to our argument, we observe that by carefully following our argument, one can find $K = K\left(d,\connf\right)<\infty$ such that
    \begin{equation}
        \kappa = \inf_{\varepsilon\in\left(0,\dist{\partial H,\orig}\right)}\frac{K}{\habsd{B_\varepsilon\left(\orig\right)}}\left(\pi + \habsd{B_\varepsilon\left(\orig\right)}\right) \times \begin{cases}
            \frac{\pi}{\theta_\varepsilon\left( H,\orig\right)} &\colon d=2\\
            \frac{2}{1-\cos\left(\theta_\varepsilon\left( H,\orig\right)\right)}&\colon d=3.
        \end{cases}
    \end{equation}
    As $\dist{\partial H,\orig}\to 0$ (and hence $\varepsilon\to 0$), we have $\theta_\varepsilon\left(H,\orig\right)\to 0$. This is ``clear'' from Figure~\ref{fig:PulledHalfSpace}, and Lemma~\ref{lem:distancecalculation} can in fact be used to show $\theta_\varepsilon\left(H,\orig\right) = \LandauBigO{\dist{\partial H, \orig}}$. This means that $\kappa\to\infty$ as $\dist{\partial H,\orig}\to 0$.
\end{remark}

Proposition~\ref{prop:BoundHalf-SpaceSusceptMagnet} allows us ``at acceptable cost'' to condition on clusters being located inside disjoint regions, and therefore be independent. This independence is the crucial ingredient in deriving mean-field exponents. 

When investigating the susceptibility and the magnetisation, we shall break up the cluster through conditioning on a certain separation event $\mathcal S$, see \eqref{eq:DefSn}. This breaks the cluster into two clusters living on half spaces. Proposition~\ref{prop:BoundHalf-SpaceSusceptMagnet} is then employed to relate the half-space clusters to full space clusters.

\subsection{Susceptibility Upper Bound}
\label{sec:SusceptibilityUpperBound}

The key bound for the proof of the upper bound on the susceptibility in Theorem~\ref{thm:CritExponents}\ref{thm:CritExponents - Susceptibility} is the following complementary bound to \eqref{eq:chi2lowerbd}. 

\begin{lemma}
    \label{lem:susceptibilityderivativeLowerBound}
    For $d=2,3$ there exist $\varepsilon,K>0$ such that for all $\lambda\in\left(\lambda_c-\varepsilon,\lambda_c\right)$
    \begin{equation}\label{eq:Kchi2bd}
        \frac{\dd \chi}{\dd \lambda} \geq K\chi\left(\lambda\right)^2.
    \end{equation}
\end{lemma}
Indeed, \eqref{eq:Kchi2bd} readily implies the upper bound in Theorem~\ref{thm:CritExponents}, as we show next:  
\begin{equation}
    \chi\left(\lambda\right) \leq C'\left(\lambda_T-\lambda\right)^{-1}
\end{equation}
for all $\lambda<\lambda_T$.

\begin{proof}[Proof of the upper bound in Theorem~\ref{thm:CritExponents}\ref{thm:CritExponents - Susceptibility}]
Now we have Lemma~\ref{lem:susceptibilityderivativeLowerBound}, the argument proceeds similarly to the proof of the lower bound in Theorem~\ref{thm:CritExponents}\ref{thm:CritExponents - Susceptibility}. From the continuity of $\lambda\mapsto\frac{1}{\chi\left(\lambda\right)}$ implying $\frac{1}{\chi\left(\lambda_c\right)}=0$, we have
\begin{equation}
    \frac{1}{\chi\left(\lambda\right)} = \int^{\lambda_c}_{\lambda}\frac{1}{\chi\left(s\right)^2}\frac{\dd \chi}{\dd s}\left(s\right)\dd s \geq K\int^{\lambda_c}_\lambda\dd s = K\left(\lambda_c-\lambda\right)
\end{equation}
for $\lambda\in\left(\lambda_c-\varepsilon,\lambda_c\right)$. Since $\chi\left(\lambda\right)\geq 1$, we can find a constant $C'$ such that 
\begin{equation}
    \chi\left(\lambda\right) \leq C'\left(\lambda_c-\lambda\right)^{-1}
\end{equation}
for all $\lambda<\lambda_c$. The result then follows from the equality $\lambda_T=\lambda_c$ (proven using lower bound in Theorem~\ref{thm:CritExponents}\ref{thm:CritExponents - Percolation}).
\end{proof}

In the remainder of the subsection we prove Lemma~\ref{lem:susceptibilityderivativeLowerBound}. We start with an auxiliary statement.

\begin{lemma}
\label{lem:Two-point-derivatives}
    Let $x\in\HypDim$ and $\lambda\in\left[0,\lambda_c\right)$. Then $\lambda\mapsto \tlam\left(\orig,x\right)$ and $\lambda\mapsto \int_{\HypDim}\tlam\left(\orig,x\right)\dd x$ are differentiable at $\lambda$ with
    \begin{align}
        \frac{\dd}{\dd \lambda}\tlam\left(\orig,x\right) &= \int_\HypDim\pla\left(u\in\piv{\orig,x;\xi^{\orig,x}}\right)\dd u \label{eqn:pointwiseDerivative}\\
        \frac{\dd}{\dd \lambda}\int_{\HypDim}\tlam\left(\orig,x\right)\dd x &= \int_\HypDim\int_\HypDim\pla\left(u\in\piv{\orig,x;\xi^{\orig,x}}\right)\dd u\dd x.\label{eqn:integralDerivative}
    \end{align}
\end{lemma}

\begin{proof}
    Equations \eqref{eqn:pointwiseDerivative} and \eqref{eqn:integralDerivative} are proven in the same way as \cite[Lemma~2.2, 2.3]{HeyHofLasMat19} respectively - the hyperbolic ambient space makes not difference to the argument. 
    We work with a truncated version of the two-point function: we let $B_n(\orig)$ be the ball around $\orig$ in $\HypDim$ with radius $n\geq 1$ in the hyperbolic metric, and define
\begin{equation}
    \tlam^n\left(\orig,x\right):= \pla\left(\conn{\orig}{x}{\xi^{\orig,x}_{B_n\left(\orig\right)}}\right).
\end{equation}
    Then the Margulis-Russo formula is applied to show that this truncated function $\lambda\mapsto\tlam^n(\orig,x)$ is differentiable on $\left[0,\lambda_c-\varepsilon\right]$ for all $\varepsilon>0$. Then it is shown that $\tlam^n\left(\orig,x\right)\to\tlam\left(\orig,x\right)$ and $\frac{\dd}{\dd \lambda}\tlam^n\left(\orig,x\right) \to \frac{\dd}{\dd \lambda}\tlam\left(\orig,x\right)$ uniformly in $\lambda\in\left[0,\lambda_c-\varepsilon\right]$. This uniform convergence then justifies the exchange of the derivative and the $n\to\infty$ limit to get \eqref{eqn:pointwiseDerivative}.

    To show \eqref{eqn:integralDerivative}, we first note that $\tlam(\orig,x)$ is non-decreasing in $\lambda$ for any $x\in\HypDim$, and therefore for any $\lambda<\lambda_c$ we can dominate $\tlam$ with $\tau_{\lambda_c-\varepsilon}$ for sufficiently small $\varepsilon>0$. Then by the measure-theoretic form of Leibniz's integral rule we can exchange the derivative and the integral to get the result.
\end{proof}

We now establish some notation to refer to geometric objects that we will use in the following lemmas. Define
\begin{equation}
    R_* := \sup\left\{r>0\colon \esssup_{x\not\in B_r\left(\orig\right)}\connf(x,\orig) = 1\right\}.
\end{equation}
Observe that the assumption $\lim_{r\to\infty}\esssup_{x\not\in B_{r}\left(\orig\right)}\connf(x,\orig)<1$ implies that $R_*<\infty$. Now since $\int_{\HypDim}\connf(x,\orig)\dd x>0$, there exist $u\in\HypDim$ and $\delta,\varepsilon>0$ such that both
\begin{equation}
    \dist{u,\orig} > \max\left\{ 4\delta, R_* - 2\delta\right\},
\end{equation}
\begin{equation}
    \essinf_{x\in B_\delta(u), y\in B_\delta(\orig)}\connf\left(x,y\right) > \varepsilon.
\end{equation}
Let $\hat{\gamma}$ denote the (unique) geodesic containing both $\orig$ and $u$, and let $\hat{\gamma}(s)$ denote the isometric parametrisation of $\hat{\gamma}$ such that $\hat{\gamma}(0)=\orig$ and $\hat{\gamma}(l)=u$, where $l:= \dist{u,\orig}$. 

For $S\in\R$ let $H^+_S$ denote the minimal closed half-space containing $\bigcup_{s>S}\hat{\gamma}\left(s\right)$ with boundary $\partial H^+_S$ orthogonal to $\hat{\gamma}$, and similarly let $H^-_S$ denote the minimal closed half-space containing $\bigcup_{s<S}\hat{\gamma}\left(s\right)$ with boundary $\partial H^-_S$ orthogonal to $\hat{\gamma}$. Now given $n\in\N$ we can partition our space using 
\begin{equation}
    H_1 := H^-_{2\delta}, \qquad H_2:= H^+_{nl-2\delta},
\end{equation}
and
\begin{equation}
    \partial H_1:= \partial H^-_{2\delta},\qquad \partial H_2:=\partial H^+_{nl - 2\delta}.
\end{equation}
It will also be convenient in the proof of Lemma~\ref{lem:susceptibilityderivativeLowerBound} to make use of
\begin{equation}
    H_3 := H^-_{\frac{1}{2}nl},\qquad\partial H_3:=\partial H^-_{\frac{1}{2}nl}.
\end{equation}
We now construct ``stepping stones'' that will allow us to connect $\orig\in H_1$ to vertices in $H_2$. For all $k\in\N$, define
\begin{equation}\label{eq:vkVk}
    V_k:= B_\delta\left(\hat{\gamma}(kl)\right).
\end{equation}

\begin{figure}
    \centering
    \begin{tikzpicture}[scale=1.3]
        \begin{scope}
            \clip (-5,-2.82) rectangle (5,2.82);
            \draw[dashed,fill=gray!15] (-11,0)  circle (8.66);
            \draw[dashed,fill=gray!15] (11,0)  circle (8.66);
            \draw[thick] (-6,0) -- (6,0);
            \draw[dashed] (0,-2.82) node[above left]{$\partial H_3$} -- (0,5);
        \end{scope}
        \draw (-3,1) node{$H_1$};
        \draw (3,1) node{$H_2$};
        \draw (-4.5,0) node[above]{$\hat\gamma$};
        \filldraw (-3,0)node[above]{$\orig$} circle (2pt);
        \draw[fill=gray!50,thick] (-1.8,0) circle (10pt) node{$V_1$};
        \draw[fill=gray!50,thick] (-0.6,0) circle (10pt);
        \draw[fill=gray!50,thick] (0.6,0) circle (10pt);
        \draw[fill=gray!50,thick] (1.8,0) circle (10pt);
        \draw[fill=gray!50,thick] (3,0) circle (10pt) node{$V_n$};
    \end{tikzpicture}
    \caption{Sketch of the `stepping stones' used to connect $\orig$ to a far half-space in Section~\ref{sec:SusceptibilityUpperBound}.}
    \label{fig:SteppingStones}
\end{figure}

\begin{lemma}\label{lem:SteppingStones}
    Let $\lambda_{\min}>0$ and $d\geq 2$. Then for sufficiently large $n\geq 1$ , there exist $C_1=C_1(\lambda_{\min},n)>0$ and $C_2=C_2(\lambda_{\min},n)>0$ such that for all $\lambda\in\left[\lambda_{\min},\lambda_c\right)$
    \begin{multline}
        \frac{\dd}{\dd \lambda}\int_{\HypDim}\tlam\left(\orig,x\right)\dd x \geq C_1\int_{H_1}\int_{H_2}\int_{V_n} \pla\left(\C\left(y,\xi^{y,v}\right)\subset H_1, \conn{y}{\orig}{\xi^{y,\orig}},\right.\\\left. \C\left(x,\xi^{x,v}\right)\subset  H_2,\conn{v}{x}{\xi^{x,v}}\right)\dd v \dd y \dd x - C_2.
    \end{multline}
\end{lemma}
The crux of the lemma is to write the pivotality in the RHS of \eqref{eqn:integralDerivative} as an integral over half-spaces only (to which subsequently we apply Proposition~\ref{prop:BoundHalf-SpaceSusceptMagnet}). 
\begin{proof}
    We carry out the proof in three steps.

    \vskip.5em
    \noindent\textbf{Step 1.} We first prove that there is a constant $c=c(\lambda_{\min},n)$ such that for sufficiently large $n$ we have 
    \begin{align}
        &\frac{\dd}{\dd \lambda}\int_{\HypDim}\tlam\left(x,\orig\right)\dd x\nonumber\\
        &\hspace{1cm}\geq c\int_{H_1}\int_{H_2} \int_{V_n}\ldots\int_{V_1}\pla\left(\C\left(y,\xi^{y,v_1,\ldots,v_n}\right)\subset H_1,\conn{y}{\orig}{\xi^{y,\orig}}, \C\left(x,\xi^{x,v_1,\ldots,v_n}\right)\subset  H_2,\right.\nonumber\\
        &\hspace{8cm}\left.\conn{v_n}{x}{\xi^{x,v_n}}\right)\dd v_1\ldots\dd v_n\dd y\dd x.\label{eq:step1}
    \end{align}
    To start with, by Lemma~\ref{lem:Two-point-derivatives} we get 
    \begin{align}
        \frac{\dd}{\dd \lambda}\int_{\HypDim}\tlam\left(x,\orig\right)\dd x
        &= \int_{\HypDim}\int_{\HypDim} \pla\left(u\in \piv{\orig,x;\xi^{\orig,u,x}}\right)\dd u\dd x\nonumber\\
        &= \int_{\HypDim}\int_{\HypDim} \pla\left(\orig\in \piv{x,y;\xi^{\orig,x,y}}\right)\dd x\dd y,
    \end{align}
    where in the last equality we used the translation invariance of the model under the hyperbolic measure.

    Observe that $\orig\in \piv{x,y;\xi^{\orig,x,y}}$ if and only if any path from $x$ to $y$ in the graph $\xi^{\orig,x,y}$ passes through $\orig$. We want a lower bound on this so we will start from the event that $\conn{x}{\orig}{\xi^{\orig,x}}$ and $\conn{y}{\orig}{\xi^{\orig,y}}$ and introduce restrictions that mean that $\nconn{x}{y}{\xi^{x,y}}$. Firstly, we cannot have $\adja{x}{y}{\xi^{x,y}}$. This is independent of everything else in the model so will immediately factorise out. We then ask that there is a sequence of vertices $\orig,v_1,v_2,\ldots,v_n$ (where each $v_i\in V_i$) that forms a bridge between $H_1$ and $ H_2$, and that $y$ is connected to $\orig$ within $H_1$ while $x$ is connected to $v_n$ within $ H_2$. To be precise, we bound
    \begin{align}
        \pla\left(\orig\in \piv{x,y;\xi^{\orig,x,y}}\right) &\geq (1-\connf(x,y))\nonumber\\
        &\hspace{0.5cm}\times\pla\left( \exists \left\{v_i\right\}^n_{i=1}\colon v_i\in \eta\cap V_i,\C\left(y,\xi^{y}\right)\subset H_1,\C\left(x,\xi^{x}\right)\subset  H_2\cup\left\{v_1,\ldots,v_n\right\},\right.\nonumber\\
        &\hspace{4cm}\left.\conn{y}{\orig}{\xi^{y,\orig}},\adja{\orig}{v_1}{\xi^{\orig}},\adja{v_1}{v_2}{\xi},\ldots,\right. \nonumber\\
        &\hspace{4cm}\left.\adja{v_{n-1}}{v_n}{\xi},\conn{v_n}{x}{\xi^x}\right).
    \end{align}
    Now we condition on these stepping stone vertices $\left\{v_1,\ldots,v_n\right\}$ being in the configuration, and note that they are all independently and uniformly distributed over the $V_i$. Therefore
    \begin{align}
        &\pla\left( \exists \left\{v_i\right\}^n_{i=1}\colon v_i\in \eta\cap V_i,\C\left(y,\xi^{y}\right)\subset H_1,\C\left(x,\xi^{x}\right)\subset  H_2\cup\left\{v_1,\ldots,v_n\right\},\conn{y}{\orig}{\xi^{y,\orig}},\right.\nonumber\\
        &\hspace{5cm}\left.\adja{\orig}{v_1}{\xi^{\orig}},\adja{v_1}{v_2}{\xi},\ldots,\adja{v_{n-1}}{v_n}{\xi},\conn{v_n}{x}{\xi^x}\right)\nonumber\\
        &\hspace{1cm} \geq \left(\prod^n_{i=1}\pla\left(\#\eta\cap V_i \geq 1\right)\right)\nonumber\\
        &\hspace{2cm}\times\int_{V_n}\ldots\int_{V_1}\pla\left(\C\left(y,\xi^{y}\right)\subset H_1, \conn{y}{\orig}{\xi^{y,\orig}}, \C\left(x,\xi^{x}\right)\subset  H_2,\adja{\orig}{v_1}{\xi^{\orig}},\right.\nonumber\\
        &\hspace{5cm}\left.\adja{v_1}{v_2}{\xi},\ldots,\adja{v_{n-1}}{v_n}{\xi},\conn{v_n}{x}{\xi^x}\;\middle|\;v_1,\ldots,v_n\in\eta\right)\nonumber\\
        &\hspace{13cm}\dd v_1\ldots\dd v_n.
    \end{align}
    Observe that in this step we also change the event $\left\{\C\left(x,\xi^{x}\right)\subset  H_2\cup\left\{v_1,\ldots,v_n\right\}\right\}$ to $\left\{\C\left(x,\xi^{x}\right)\subset  H_2\right\}$. The inequality arises because the latter event is a subset of the former.

    Since $\eta$ is a Poisson point process, conditioning on $v_1,\ldots,v_n\in\eta$ is equivalent to augmenting $\xi$ with these vertices. Therefore
    \begin{align}
        &\pla\left(\C\left(y,\xi^{y}\right)\subset H_1, \conn{y}{\orig}{\xi^{y,\orig}}, \C\left(x,\xi^{x}\right)\subset  H_2,\adja{\orig}{v_1}{\xi^{\orig}},\right.\nonumber\\
        &\hspace{4cm}\left.\adja{v_1}{v_2}{\xi},\ldots,\adja{v_{n-1}}{v_n}{\xi},\conn{v_n}{x}{\xi^x}\;\middle|\;v_1,\ldots,v_n\in\eta\right)\nonumber\\
        &\hspace{1cm} = \pla\left(\C\left(y,\xi^{y,v_1,\ldots,v_n}\right)\subset H_1, \conn{y}{\orig}{\xi^{y,\orig,v_1,\ldots,v_n}}, \C\left(x,\xi^{x,v_1,\ldots,v_n}\right)\subset  H_2,\adja{\orig}{v_1}{\xi^{\orig,v_1}},\right.\nonumber\\
        &\hspace{4cm}\left.\adja{v_1}{v_2}{\xi^{v_1,v_2}},\ldots,\adja{v_{n-1}}{v_n}{\xi^{v_{n-1},v_n}},\conn{v_n}{x}{\xi^{v_1,\ldots,v_n,x}}\right)\nonumber\\
        &\hspace{1cm} \geq \pla\left(\C\left(y,\xi^{y,v_1,\ldots,v_n}\right)\subset H_1, \conn{y}{\orig}{\xi^{y,\orig}}, \C\left(x,\xi^{x,v_1,\ldots,v_n}\right)\subset  H_2,\adja{\orig}{v_1}{\xi^{\orig,v_1}},\right.\nonumber\\
        &\hspace{4cm}\left.\adja{v_1}{v_2}{\xi^{v_1,v_2}},\ldots,\adja{v_{n-1}}{v_n}{\xi^{v_{n-1},v_n}},\conn{v_n}{x}{\xi^{v_n,x}}\right).
    \end{align}
    The inequality arises because $\left\{\conn{y}{\orig}{\xi^{y,\orig,v_1,\ldots,v_n}}\right\}\supset \left\{\conn{y}{\orig}{\xi^{y,\orig}}\right\}$  and $\left\{\conn{v_n}{x}{\xi^{v_1,\ldots,v_n,x}}\right\}\supset \left\{\conn{v_n}{x}{\xi^{v_n,x}}\right\}$. Note that many of the extra vertices are omitted from the adjacency events because these are independent of the entire configuration other than the two vertices in question. This independence can be further used to extract these events from the probability:
    \begin{align}
        &\pla\left(\C\left(y,\xi^{y,v_1,\ldots,v_n}\right)\subset H_1, \conn{y}{\orig}{\xi^{y,\orig}}, \C\left(x,\xi^{x,v_1,\ldots,v_n}\right)\subset  H_2,\adja{\orig}{v_1}{\xi^{\orig,v_1}},\right.\nonumber\\
        &\hspace{4cm}\left.\adja{v_1}{v_2}{\xi^{v_1,v_2}},\ldots,\adja{v_{n-1}}{v_n}{\xi^{v_{n-1},v_n}},\conn{v_n}{x}{\xi^{v_n,x}}\right)\nonumber\\
        &\hspace{1cm} = \pla\left(\C\left(y,\xi^{y,v_1,\ldots,v_n}\right)\subset H_1, \conn{y}{\orig}{\xi^{y,\orig}}, \C\left(x,\xi^{x,v_1,\ldots,v_n}\right)\subset  H_2,\conn{v_n}{x}{\xi^{v_n,x}}\right)\nonumber\\
        &\hspace{4cm}\times\connf\left(\orig,v_1\right)\connf\left(v_1,v_2\right)\ldots\connf\left(v_{n-1},v_n\right).
    \end{align}
    The above steps lead to the bound
    \begin{align}
        &\frac{\dd}{\dd \lambda}\int_{\HypDim}\tlam\left(x,\orig\right)\dd x\nonumber\\
        &\hspace{1cm}\geq\left(\prod^n_{i=1}\pla\left(\#\eta\cap V_i \geq 1\right)\right)\int_{\HypDim}\int_{\HypDim} (1-\connf(x,y))\int_{V_n}\ldots\int_{V_1}\pla\left(\C\left(y,\xi^{y,v_1,\ldots,v_n}\right)\subset H_1, \right.\nonumber\\
        &\hspace{4cm}\left.\conn{y}{\orig}{\xi^{y,\orig}}, \C\left(x,\xi^{x,v_1,\ldots,v_n}\right)\subset  H_2,\conn{v_n}{x}{\xi^{x,v_n}}\right)\nonumber\\
        &\hspace{4cm}\times\connf(\orig,v_1)\ldots\connf\left(v_{n-1},v_n\right)\dd v_1\ldots\dd v_n\dd y\dd x.
    \end{align}
    Because $\eta$ is a Poisson point process with intensity $\lambda$, 
    $$\pla\left(\#\eta\cap V_i \geq 1\right) = 1 - \exp\left(-\lambda\habsd{B_\delta\left(\orig\right)}\right) \geq 1 - \exp\left(-\lambda_{\min}\habsd{B_\delta\left(\orig\right)}\right).$$ 
    Since $\lim_{r\to \infty}\esssup_{x\not\in B_r(\orig)}\connf(x,y)<1$, if $n$ is sufficiently large and we restrict to $x\in H_1$ and $y\in H_2$ then there exists $c>0$ such that $1-\connf\left(x,y\right)\geq c$. Together with the lower bound $\connf\left(v_i,v_{i+1}\right)\geq\ \inf_{u\in V_i,u'\in V_{i+1}}\connf\left(u,u'\right)\geq \varepsilon$ for our stepping stones, these give \eqref{eq:step1}. 

    \vskip.5em
    \noindent\textbf{Step 2.}
    This bound is still too messy for our purposes. We want to remove the augmented vertices $v_1,\ldots,v_{n-1}$ from the integrand, so that this dependence disappears and we can integrate these variables over $V_1,\ldots,V_{n-1}$. This will make it easier to factorise the integral in the following lemma. To this end, we introduce some notation. Given $x\in H_2$ and $y\in H_1$ we define `blue' vertices to be
    \begin{equation}
        \Ycal_B := \bigcup^{n-1}_{i=1}\left\{u\in\eta^y\cap H_1\colon \adja{u}{v_i}{\xi^{y,v_1,\ldots,v_n}}\right\}\cup \bigcup^{n-1}_{i=1}\left\{u\in\eta^x\cap H_2\colon \adja{u}{v_i}{\xi^{x,v_1,\ldots,v_n}}\right\}.
    \end{equation}
    The set $\Ycal_B$ consists of the troublesome vertices for us if we wanted to remove the vertices $v_1,\ldots,v_{n-1}$ from our probability, because the vertices in the first part provide avenues through which the cluster of $y$ could `escape' $ H_1$ in $\xi^{y,v_1,\ldots,v_n}$ but not in $\xi^{y,v_n}$, and vertices in the second part provide avenues through which the cluster of $x$ could `escape' $ H_2$ in $\xi^{x,v_1,\ldots,v_n}$ but not in $\xi^{x,v_n}$. Also note that $\Ycal_B\setminus\left\{x,y\right\}$ is distributed as a Poisson point process whose intensity measure is finite and has total mass less than or equal to $(n-1)\int_\HypDim\connf(x,\orig)\dd x$.

    For $\varepsilon_1\geq0$, $n\geq 2$, and a sequence of points $\left\{v_1,\ldots,v_{n-1}\right\}\subset \HypDim$, we define the measurable set 
    \begin{equation}
        \Ucal_{\varepsilon_1}:= \left\{z\in\HypDim\colon 1-\prod^{n-1}_{i=1}\left(1-\connf\left(z,v_i\right)\right)\geq 1-\varepsilon_1\right\}.
    \end{equation}
    This is the set of positions such that if a vertex was at this position then it would be adjacent to at least one of the vertices $\left\{v_1,\ldots,v_{n-1}\right\}$ with a `high probability.' Observe that $\Ucal_{\varepsilon_1}$ is non-decreasing in $\varepsilon_1$, and it has finite measure:
    \begin{equation}
        \habsd{\Ucal_{\varepsilon_1}}\leq \frac{n-1}{1-\varepsilon_1}\int_\HypDim\connf(x,\orig)\dd x<\infty.
    \end{equation}
    Also observe that since $\lim_{r\to \infty}\esssup_{x\not\in B_r(\orig)}\connf(x,y)<1$, for sufficiently small $\varepsilon_1$ we have that $\Ucal_{\varepsilon_1}$ is bounded except for possibly a null set. 

    We now make the claim that for $\varepsilon_1$ sufficiently small there exists $c=c(n,\lambda,\varepsilon_1)>0$ such that for $x\in H_2\setminus \Ucal_{\varepsilon_1}$ and $y\in H_1\setminus \Ucal_{\varepsilon_1}$
    \begin{multline}
    \label{eqn:middleClaim}
        \pla\left(\C\left(y,\xi^{y,v_1,\ldots,v_n}\right)\subset  H_1, \conn{y}{\orig}{\xi^{y,\orig}}, \C\left(x,\xi^{x,v_1,\ldots,v_n}\right)\subset  H_2,\conn{v_n}{x}{\xi^{x,v_n}}\right) \\\geq c\, \pla\left(\C\left(y,\xi^{y,v_n}\right)\subset  H_1, \conn{y}{\orig}{\xi^{y,\orig}}, \C\left(x,\xi^{x,v_n}\right)\subset  H_2,\conn{v_n}{x}{\xi^{x,v_n}}\right).
    \end{multline}
    For $x\notin H_2\setminus \Ucal_{\varepsilon_1}$ or $y\notin H_1\setminus \Ucal_{\varepsilon_1}$, we will use the trivial lower bound of $0$.

    To demonstrate the strategy, we first consider the case where $R^*=0$. Observe that this means that $\habsd{\Ucal_{\varepsilon_1}}\searrow 0$ as $\varepsilon_1\searrow0$, and in particular that $\Ucal_{\varepsilon_1}\cap  H_1$ and $\Ucal_{\varepsilon_1}\cap  H_2$ are null for sufficiently small $\varepsilon_1$. Since $x\not\in\Ucal_{\varepsilon_1}$ and $y\not\in\Ucal_{\varepsilon_1}$, they are not themselves blue with a positive probability and the Poissonian nature of the remaining vertices tells us that there exists $\tilde c=\tilde c(n,\lambda)>0$ such that $\pla\left(\#\Ycal_B=0\right)= \tilde c$. Furthermore, \emph{if there are no blue vertices}, then $\C\left(y,\xi^{y,v_1,\ldots,v_n}\right)\subset  H_1$ precisely when $\C\left(y,\xi^{y,v_n}\right)\subset  H_1$, and $ \C\left(x,\xi^{x,v_1,\ldots,v_n}\right)\subset  H_2$ precisely when $ \C\left(x,\xi^{x,v_n}\right)\subset  H_2$. Therefore when $x\not\in\Ucal_{\varepsilon_1}$ and $y\not\in\Ucal_{\varepsilon_1}$,
    \begin{align}
        &\pla\left(\C\left(y,\xi^{y,v_1,\ldots,v_n}\right)\subset  H_1, \conn{y}{\orig}{\xi^{y,\orig}}, \C\left(x,\xi^{x,v_1,\ldots,v_n}\right)\subset  H_2,\conn{v_n}{x}{\xi^{x,v_n}}\right)\nonumber\\
        &\hspace{1cm}\geq\pla\left(\C\left(y,\xi^{y,v_1,\ldots,v_n}\right)\subset  H_1, \conn{y}{\orig}{\xi^{y,\orig}}, \C\left(x,\xi^{x,v_1,\ldots,v_n}\right)\subset  H_2,\right.\nonumber\\
        &\hspace{9cm}\left.\conn{v_n}{x}{\xi^{x,v_n}}\;\middle|\;\#\Ycal_B =0\right) \pla\left(\#\Ycal_B =0\right)\nonumber\\
        &\hspace{1cm} = \tilde c \,\pla\left(\C\left(y,\xi^{y,v_n}\right)\subset  H_1, \conn{y}{\orig}{\xi^{y,\orig}}, \C\left(x,\xi^{x,v_n}\right)\subset  H_2,\right.\nonumber\\
        &\hspace{9cm}\left.\conn{v_n}{x}{\xi^{x,v_n}}\;\middle|\;\#\Ycal_B =0\right).
    \end{align}
    If $\xi\sim \pla\left(\cdot \;\middle|\;\#\Ycal_B = 0\right)$, and $\xi'\sim \mathbb{P}_{\nu}$ where $\nu$ is defined by the Radon-Nikodym derivative
    \begin{equation}
        \frac{\dd \nu}{\dd x} = \lambda\prod^{n-1}_{i=1}\left(1-\connf(x,v_i)\right),
    \end{equation}
    then $\xi\dequal \xi'$. Now recall that $\Ucal_{\varepsilon_1}\cap  H_1$ and $\Ucal_{\varepsilon_1}\cap  H_2$ are null for sufficiently small $\varepsilon_1$. Therefore $\frac{1}{\lambda}\frac{\dd \nu}{\dd x}\in\left(\varepsilon_1,1\right]$ almost everywhere. Furthermore since $\lambda-\nu$ is a finite measure (in fact, $\left(\lambda-\nu\right)\left(\HypDim\right)\leq \lambda n \int_\HypDim\connf(x,\orig)\dd x<\infty$), we can modify a configuration distributed as $\pla$ that satisfies the event so that it is now distributed as $\mathbb{P}_\nu$ at only the cost of a factor $\hat c=\hat c(n,\lambda,\varepsilon_1)$. Multiplying $\hat c$ and $\tilde c$ then gives the claimed $c=c(n,\lambda,\varepsilon_1)$.

    Now let us address the case where $R_*\in\left(0,\infty\right)$. Recall that we suppose that $x\not\in\Ucal_{\varepsilon_1}$ and $y\not\in\Ucal_{\varepsilon_1}$ for some $\varepsilon_1>0$. Once again because $\Ycal_B\setminus\left\{x,y\right\}$ is distributed as independent Poisson point processes whose intensity measure is finite, we are able to modify the configuration such that there are no blue vertices. However, since $R_*>0$ we now have the possibility that $\frac{\dd \nu}{\dd x}=0$ on a positive measure part of $ H_1$ or $ H_2$. We therefore have to be careful with reintroducing vertices to ensure that $\conn{y}{\orig}{\xi^{y,\orig}}$ and $\conn{v_n}{x}{\xi^{x,v_n}}$ still occur. Fortunately, having $R_*<\infty$ ensures that $\Ucal_{\varepsilon_1}$ is bounded by $\bigcup^{n-1}_{i=1}B_{R_*+\varepsilon_2+\delta}(v_i)$ for some $\varepsilon_2=\varepsilon_2\left(\varepsilon_1\right)>0$ where $\lim_{\varepsilon_1\searrow0}\varepsilon_2=0$. The symmetry of $\connf$ and the construction of the $V_i$ ensures that there is a positive probability that $\orig$ is adjacent to a vertex in $\eta\cap H_1\cap\Ucal^\mathrm{c}_{\varepsilon_1} $, and $v_n$ is adjacent to a vertex in $\eta\cap H_2\cap\Ucal^\mathrm{c}_{\varepsilon_1} $. There is then a positive probability that these vertices are connected to $y$ and $x$ respectively. This concludes the proof of the claim in \eqref{eqn:middleClaim}.

    \vskip.5em
    \noindent\textbf{Step 3.}
    We now have that for sufficiently large $n$ there exists $c>0$ such that
    \begin{align}
        &\frac{\dd}{\dd \lambda}\int_{\HypDim}\tlam\left(x,\orig\right)\dd x\nonumber\\
        &\hspace{1cm}\geq c \int_{H_1\setminus\Ucal_{\varepsilon_1} }\int_{H_2\setminus \Ucal_{\varepsilon_1}} \int_{V_n}\ldots\int_{V_1}\pla\left(\C\left(y,\xi^{y,v_n}\right)\subset H_1,\conn{y}{\orig}{\xi^{y,\orig}}, \C\left(x,\xi^{x,v_n}\right)\subset  H_2,\right.\nonumber\\
        &\hspace{8cm}\left.\conn{v_n}{x}{\xi^{x,v_n}}\right)\dd v_1\ldots\dd v_n\dd y\dd x\nonumber\\
        &\hspace{1cm} = c\, \habsd{B_\delta\left(\orig\right)}^{n-1}\int_{H_1\setminus\Ucal_{\varepsilon_1}}\int_{H_2\setminus \Ucal_{\varepsilon_1}} \int_{V_n}\pla\left(\C\left(y,\xi^{y,v_n}\right)\subset H_1,\conn{y}{\orig}{\xi^{y,\orig}}, \C\left(x,\xi^{x,v_n}\right)\subset  H_2,\right.\nonumber\\
        &\hspace{8cm}\left.\conn{v_n}{x}{\xi^{x,v_n}}\right)\dd v_n\dd y\dd x.
    \end{align}    
    We can then remove the $\Ucal_{\varepsilon_1}$ terms from the integral domains by performing the following bound. Since we can bound the integrand above by $1$, 
    \begin{align}       
        &\int_{H_1}\int_{H_2}  \pla\left(\C\left(y,\xi^{y,v_n}\right)\subset  H_1, \conn{y}{\orig}{\xi^{y,\orig}},\C\left(x,\xi^{x,v_n}\right)\subset  H_2,\conn{v_n}{x}{\xi^{x,v_n}}\right) \dd y\dd x\nonumber\\
        &\hspace{1cm} \leq \int_{H_1\setminus \Ucal_{\varepsilon_1}}\int_{H_2\setminus \Ucal_{\varepsilon_1}}\pla\left(\C\left(y,\xi^{y,v_n}\right)\subset  H_1, \conn{y}{\orig}{\xi^{y,\orig}},\right.\nonumber\\
        &\hspace{6cm}\left.\C\left(x,\xi^{x,v_n}\right)\subset  H_2,\conn{v_n}{x}{\xi^{x,v_n}}\right) \dd y\dd x + \habsd{\Ucal_{\varepsilon_1}}^2.
    \end{align}
    Therefore after exchanging the order of integration (by Tonelli's theorem) we have that for sufficiently large $n$ there exists $c'>0$ such that
    \begin{align}
        &\frac{\dd}{\dd \lambda}\int_{\HypDim}\tlam\left(x,\orig\right)\dd x\nonumber\\
        &\hspace{1cm}\geq c' \int_{H_1}\int_{H_2} \int_{V_n}\pla\left(\C\left(y,\xi^{y,v_n}\right)\subset H_1,\conn{y}{\orig}{\xi^{y,\orig}}, \C\left(x,\xi^{x,v_n}\right)\subset  H_2,\right.\nonumber\\
        &\hspace{6cm}\left.\conn{v_n}{x}{\xi^{x,v_n}}\right)\dd v_n\dd y\dd x - c'\habsd{B_\delta\left(\orig\right)}\habsd{\Ucal_{\varepsilon_1}}^2.
    \end{align}
    
\end{proof}

\begin{proof}[Proof of Lemma~\ref{lem:susceptibilityderivativeLowerBound}]
For the sake of readability, let us abbreviate
\begin{align}
    E(y):=&\left\{\C\left(y,\xi^{y}\right)\subset  H_1, \conn{y}{\orig}{\xi^{y,\orig}}\right\}, \\
    E(y,v_n):=&\left\{\C\left(y,\xi^{y,v_n}\right)\subset  H_1, \conn{y}{\orig}{\xi^{y,\orig}}\right\}, \\
    F(v_n,x):=&\left\{\C\left(x,\xi^{x,v_n}\right)\subset  H_2,\conn{v_n}{x}{\xi^{v_n,x}}\right\}.
\end{align}
Observe that as $n\to\infty$, we have $\pla\left(E(y,v_n)\cap F(v_n,x)\right) \to \pla\left(E(y)\cap F(v_n,x)\right)$ monotonically since the probability that $v_n$ is adjacent to any vertex in $ H_1$ tends to $0$ monotonically. Therefore for sufficiently large $n$ we have by Lemma~\ref{lem:SteppingStones} the existence of $c>0$ and $c'<\infty$ such that
\begin{equation}
    \frac{\dd}{\dd \lambda}\int_{\HypTwo}\tlam\left(x,\orig\right)\dd x \geq c\int_{H_1}\int_{H_2}\int_{V_n}\pla\left(E(y)\cap F(v,x)\right)\dd v \dd x \dd y - c'
\end{equation}

Our objective now is to condition on a separation event under which $E$ and $F$ are independent events, and therefore the probability and the integrals factorise. 
To this end, we define the separation event
\begin{equation}\label{eq:DefSn}
        \Scal_n := \left\{\xi\colon \forall a\in\eta\cap H_1, \forall b\in\eta\cap H_3^\mathrm{c}, a\not\sim b\right\}\cap\left\{\xi\colon \forall a\in\eta\cap H_3, \forall b\in\eta\cap H_2, a\not\sim b\right\}.
    \end{equation}
    This is constructed so that the configurations on $H_1$ and $H_2$ are independent when we condition on it. 
    By the law of total probability, it is immediate that
    \begin{equation}
    \label{eqn:conditioning}
        \pla\left(E(y)\cap F(v,x)\right) \geq \pla\left(E(y)\cap F(v,x)\;\middle|\;\xi^{\orig,x,y,v}\in\Scal_n\right)\pla\left(\xi^{\orig,x,y,v}\in\Scal_n\right),
    \end{equation}
    and our definition of the conditioning event ensures that for $x\in H_2$, $v\in V_n\subset H_2$, and $y\in H_1$ we have
    \begin{equation}\label{eq:FactorizedLowerBd}
        \pla\left(E(y)\cap F(v,x)\;\middle|\;\xi^{\orig,x,y,v}\in\Scal_n\right) = \pla\left(E(y)\;\middle|\;\xi^{\orig,x,y,v}\in\Scal_n\right)\pla\left(F(v,x)\;\middle|\;\xi^{\orig,x,y,v}\in\Scal_n\right).
    \end{equation}

    \vskip.5em
    We are now bounding $\pla\left(\xi^{\orig,x,y,v}\in\Scal_n\right)$ and the two factors on the RHS of \eqref{eq:FactorizedLowerBd} in order to get that there exists $c''>0$ such that for sufficiently large $n$
    \begin{multline}\label{eqn:EFpartitionbound}
        \frac{\dd}{\dd \lambda}\int_{\HypTwo}\tlam\left(x,\orig\right)\dd x \geq c''\int_{H_1}\int_{H_2}\int_{V_n}\pla\left(E(y)\right)\pla\left(F(v,x)\right)\dd v \dd x \dd y - c'\\
        = c''\left(\int_{H_1}\pla\left(E(y)\right)\dd y\right)\left(\int_{H_2}\int_{V_n}\pla\left(F(v,x)\right)\dd v \dd x\right) - c'.
    \end{multline}
    To this end, we first find a lower bound for $\pla\left(\xi^{\orig,x,y,v}\in\Scal_n\right)$. Note that since $x,v\in H_2$ and $\orig,y\in H_1$, we can decompose the event $\left\{\xi^{\orig,x,y,v}\in\Scal_n\right\}$ into the non-disjoint events
    \begin{equation}
    \label{eqn:SeparationDecomposition}
        \left\{\xi^{\orig,x,y,v}\in\Scal_n\right\} = \left\{x\not\sim y\right\}\cap \left\{v\not\sim y\right\}\cap\left\{x\not\sim \orig\right\}\cap\left\{v\not\sim y\right\}\cap\left\{\xi^{x,v}\in\Scal_n\right\}\cap \left\{\xi^{\orig,y}\in\Scal_n\right\}.
    \end{equation}
    The first four events are independent of the last two, and these last two are both decreasing events. Therefore the FKG inequality \cite[(2.19)]{HeyHofLasMat19} produces
    \begin{equation}
        \pla\left(\xi^{\orig,x,y,v}\in\Scal_n\right) \geq \left(1-\connf(x,y)\right)\left(1-\connf\left(v,y\right)\right)\left(1-\connf(x,\orig)\right)\left(1-\connf\left(v,\orig\right)\right)\pla\left(\xi^{x,v}\in\Scal_n\right)\pla\left(\xi^{\orig,y}\in\Scal_n\right).
    \end{equation}
    Recall that $\overline{\varphi}:= \lim_{r\to\infty}\esssup_{x\not\in\B_r\left(\orig\right)}\connf\left(x,\orig\right)<1$, and $\int_\HypDim\connf(x,\orig)\dd x<\infty$. Therefore Mecke's formula and dominated convergence implies that $\lim_{n\to\infty}\pla\left(\xi^{x,v}\in\Scal_n\right)=1$ uniformly in $x,v\in H_2$ and $\lim_{n\to\infty}\pla\left(\xi^{\orig,y}\in\Scal_n\right)=1$ uniformly in $y\in H_1$. Hence for sufficiently large $n$ we have
    \begin{equation}
        \pla\left(\xi^{\orig,x,y,v}\in\Scal_n\right)\geq \frac{1}{2}\left(1-\overline{\varphi}\right)^4>0
    \end{equation}
    for all $y\in H_1$ and $x,v\in H_2$.

    To bound $\pla\left(E(y)\;\middle|\;\xi^{\orig,x,y,v}\in\Scal_n\right)$ from below, we first adjust the conditioning event to finally decouple the $x$ and $v$, and $y$ dependence. By using \eqref{eqn:SeparationDecomposition} and the independence of the edges between $x$, $v$, $y$, and $\orig$ from the rest of the graph, 
    \begin{align}\label{eq:EXiSn}
        &\pla\left(E(y)\;\middle|\;\xi^{\orig,x,y,v}\in\Scal_n\right)\nonumber \\
        &\hspace{1cm}= \frac{\pla\left(E(y)\cap\left\{\xi^{\orig,x,y,v}\in\Scal_n\right\}\right)}{\pla\left(x\not\sim y\right)\pla\left(v\not\sim y\right)\pla\left(x\not\sim \orig\right)\pla\left(v\not\sim y\right)\pla\left(\left\{\xi^{x,v}\in\Scal_n\right\}\cap \left\{\xi^{\orig,y}\in\Scal_n\right\}\right)}\nonumber\\
        &\hspace{1cm}\geq \frac{\pla\left(E(y)\cap\left\{\xi^{\orig,x,y,v}\in\Scal_n\right\}\right)}{\pla\left(\left\{\xi^{x,v}\in\Scal_n\right\}\cap \left\{\xi^{\orig,y}\in\Scal_n\right\}\right)}\nonumber\\
        &\hspace{1cm} = \pla\left(E(y)\cap\left\{x\not\sim y\right\}\cap \left\{v\not\sim y\right\}\cap\left\{x\not\sim \orig\right\}\cap\left\{v\not\sim y\right\}\;\middle|\;\left\{\xi^{x,v}\in\Scal_n\right\}\cap \left\{\xi^{\orig,y}\in\Scal_n\right\}\right)
    \end{align}
    Note that as discussed above, FKG, Mecke and dominated convergence implies that
    \begin{equation}
        \pla\left(\left\{\xi^{x,v}\in\Scal_n\right\}\cap \left\{\xi^{\orig,y}\in\Scal_n\right\}\right) \geq \pla\left(\xi^{x,v}\in\Scal_n\right)\pla\left(\xi^{\orig,y}\in\Scal_n\right)\to 1
    \end{equation}
    as $n\to \infty$, uniformly for $x,v\in H_2$ and $y\in H_1$. Therefore
    \begin{equation}\label{eq:RatioPEPE}
        \frac{\pla\left(E(y)\cap\left\{x\not\sim y\right\}\cap \left\{v\not\sim y\right\}\cap\left\{x\not\sim \orig\right\}\cap\left\{v\not\sim y\right\}\;\middle|\;\left\{\xi^{x,v}\in\Scal_n\right\}\cap \left\{\xi^{\orig,y}\in\Scal_n\right\}\right)}{\pla\left(E(y)\cap\left\{x\not\sim y\right\}\cap \left\{v\not\sim y\right\}\cap\left\{x\not\sim \orig\right\}\cap\left\{v\not\sim y\right\}\right)} \to 1
    \end{equation}
    as $n\to \infty$, uniformly for $x,v\in H_2$ and $y\in H_1$. 
    Using \eqref{eq:EXiSn} with \eqref{eq:RatioPEPE} gives 
    \[\pla\left(E(y)\;\middle|\;\xi^{\orig,x,y,v}\in\Scal_n\right)
    \geq
    \frac12
    \pla\left(E(y)\cap\left\{x\not\sim y\right\}\cap \left\{v\not\sim y\right\}\cap\left\{x\not\sim \orig\right\}\cap\left\{v\not\sim y\right\}\right)
    \]
    for $n$ sufficiently large. 
    We observe independence of the events $E(y)$, $\left\{x\not\sim y\right\}$, $\left\{v\not\sim y\right\}$, $\left\{x\not\sim \orig\right\}$, and $\left\{v\not\sim y\right\}$ (so that their joint probability factors), and have that for $n$ sufficiently large (independent of $x,v\in H_2$, and $y\in H_1$)
    \begin{equation}
        \pla\left(E(y)\;\middle|\;\xi^{\orig,x,y,v}\in\Scal_n\right) \geq \frac{1}{2} \left(1-\overline{\connf}\right)^4\pla\left(E(y)\right).
    \end{equation}
    The same argument also gives
    \begin{equation}
        \pla\left(F(v,x)\;\middle|\;\xi^{\orig,x,y,v}\in\Scal_n\right) \geq \frac{1}{2} \left(1-\overline{\connf}\right)^4\pla\left(F(v,x)\right)
    \end{equation}
    for $n$ sufficiently large (again independent of $x,v\in H_2$, and $y\in H_1$). 
    This finishes the proof of \eqref{eqn:EFpartitionbound}.

    \vskip.5em
    We next claim that there exist $c'''>0$ and $c'<\infty$ such that
    \begin{equation}\label{eq:Step2'}
         \frac{\dd}{\dd \lambda}\int_{\HypDim}\tlam\left(x,\orig\right)\dd x \geq c''' \ela\left[\#\C\left(\orig,\xi^{\orig}\right) \Id_{\left\{\C\left(\orig,\xi^{\orig}\right)\subset H^-_{2\varepsilon}\right\}}\right]\ela\left[\#\C\left(\orig,\xi^{\orig}\right) \Id_{\left\{\C\left(\orig,\xi^{\orig}\right)\subset H^-_\delta\right\}}\right]- c'.
    \end{equation}
    In order to show this claim, we further bound from below the two terms in brackets on the RHS of \eqref{eqn:EFpartitionbound}. 
    We start with $\int_{H_1}\pla\left(E(y)\right)\dd y$ and write out
    \begin{align}
        \int_{H_1}\pla\left(E(y)\right)\dd y &= \int_{H_1}\pla\left(\C\left(y,\xi^{y}\right)\subset  H_1, \conn{y}{\orig}{\xi^{y,\orig}}\right)\dd y\nonumber\\
        &\geq \int_{H_1}\pla\left(\C\left(y,\xi^{\orig,y}\right)\subset  H_1, \conn{y}{\orig}{\xi^{y,\orig}}\right)\dd y\nonumber\\
        &= \int_{H_1}\pla\left(\C\left(\orig,\xi^{\orig,y}\right)\subset  H_1, \conn{y}{\orig}{\xi^{y,\orig}}\right)\dd y.
    \end{align}
    Here we first introduced the vertex $\orig$ to $\xi^y$, which produces a sub-event of the original event, and then were able to replace the first $y$ in $\C\left(y,\xi^{\orig,y}\right)$ with an $\orig$ because we are also on the event that $\conn{y}{\orig}{\xi^{y,\orig}}$ and therefore the two clusters are the same cluster. Finally we are able to use Mecke's formula to get
    \begin{equation}\label{eqn:Eintegralbound}
        \int_{H_1}\pla\left(E(y)\right)\dd y \ge \frac{1}{\lambda}\ela\left[\#\C\left(\orig,\xi^{\orig}\right) \Id_{\left\{\C\left(\orig,\xi^{\orig}\right)\subset H_1\right\}}\right].
    \end{equation}
    This is exactly what we require once we notice that $\lambda<\lambda_\mathrm{c}<\infty$ and $H_1:=H^-_{2\varepsilon}$.

    We now bound $\int_{H_2}\int_{V_n}\pla\left(F(v,x)\right)\dd v\dd x$. First we use Tonelli's lemma to exchange the order of integration, and replace an instance of $x$ in the integrand with $v$ because we know they are in the same cluster:
    \begin{multline}
        \int_{H_2}\int_{V_n}\pla\left(F(v,x)\right)\dd v \dd x = \int_{V_n}\int_{H_2}\pla\left(\C\left(x,\xi^{x,v}\right)\subset  H_2,\conn{v}{x}{\xi^{v,x}}\right) \dd x\dd v\\ = \int_{V_n}\int_{H_2}\pla\left(\C\left(v,\xi^{x,v}\right)\subset  H_2,\conn{v}{x}{\xi^{v,x}}\right) \dd x\dd v
    \end{multline}
    Now fix $v\in V_n$. For all $v\in V_n$, there exists a unique geodesic $\gamma'$ that contains $v$ and is perpendicular to $\partial H_2$ (we take the `positive' direction on $\gamma'$ to be the direction towards $\partial H_2$ from $\orig$). Let $\iota$ be an isometry that maps $v\mapsto\orig$ and $\gamma'\mapsto\hat{\gamma}$ (matching the `positive' directions). This can be constructed by composing a translation that takes $v\mapsto\orig$, and then appropriately rotating about $\orig$. If we define $H':=\iota\left(H_2\right)$, then we find
    \begin{equation}
        H' = H^-_{\dist{v,\partial H_2}}.
    \end{equation}
    Applying this isometry $\iota$ for each $v$, we find
    \begin{equation}
        \int_{H_2}\pla\left(\C\left(v,\xi^{x,v}\right)\subset  H_2,\conn{v}{x}{\xi^{v,x}}\right) \dd x =\int_{H'}\pla\left(\C\left(\orig,\xi^{x,\orig}\right)\subset  H',\conn{\orig}{x}{\xi^{x,\orig}}\right) \dd x.
    \end{equation}
    The right hand side of this equation depends upon $v$ only via $H'$, and is non-decreasing in $\dist{v,\partial H_2}$. Since our construction of $V_n$ and $H_2$ ensures that
    \begin{equation}
         \inf_{v\in V_n}\dist{v,\partial H_2} \geq \delta,
    \end{equation}
    hence $H^-_\delta\subset H'$, and we thus have 
    \begin{multline}\label{eqn:Fintegralbound}
        \int_{H_2}\int_{V_n}\pla\left(F(v,x)\right)\dd v \dd x \geq \habsd{B_\delta\left(\orig\right)}\int_{H^-_\delta}\pla\left(\C\left(\orig,\xi^{x,\orig}\right)\subset  H^-_\delta,\conn{\orig}{x}{\xi^{x,\orig}}\right) \dd x\\
        = \frac{\habsd{B_\delta\left(\orig\right)}}{\lambda}\ela\left[\#\C\left(\orig,\xi^{\orig}\right) \Id_{\left\{\C\left(\orig,\xi^{\orig}\right)\subset H^-_\delta\right\}}\right],
    \end{multline}
    where the last equality follows from Mecke's formula.
    This is excatly what we require once we notice that  $\lambda<\lambda_\mathrm{c}<\infty$.
    
    Combining \eqref{eqn:EFpartitionbound} with \eqref{eqn:Eintegralbound} and \eqref{eqn:Fintegralbound} gives \eqref{eq:Step2'}.  
    
    \vskip.5em
    In the final step we show how \eqref{eq:Step2'} yields the assertion of the lemma. With Proposition~\ref{prop:BoundHalf-SpaceSusceptMagnet}, this leads to there existing  $C>0$ and $c'<\infty$ such that
    \begin{equation}
         \frac{\dd}{\dd \lambda}\int_{\HypTwo}\tlam\left(x,\orig\right)\dd x \geq C \ela\left[\#\C\left(\orig,\xi^{\orig}\right)\right]^2- c',
    \end{equation}

    Since Mecke's formula implies $\chi\left(\lambda\right) = 1+ \lambda \int_\HypDim\tlam\left(x,\orig\right)\dd x$, this inequality implies
    \begin{equation}
        \frac{\dd \chi}{\dd \lambda} \geq \frac{1}{\lambda}\left(\chi\left(\lambda\right)-1\right) +\lambda C \chi\left(\lambda\right)^2 - \lambda c'.
    \end{equation}
    Then the lower bound in Theorem~\ref{thm:CritExponents}\ref{thm:CritExponents - Susceptibility} implies that $\lim_{\lambda\nearrow\lambda_{\mathrm{c}}}\chi\left(\lambda\right)=\infty$, and therefore there exist $\varepsilon,K>0$ such that for $\lambda\in\left(\lambda_{c}-\varepsilon,\lambda_{\mathrm{c}}\right)$
    \begin{equation}
        \frac{\dd \chi}{\dd \lambda} \geq K\chi\left(\lambda\right)^2.
    \end{equation}

\end{proof}

This now completes the proof of Theorem~\ref{thm:CritExponents}\ref{thm:CritExponents - Susceptibility}.

\subsection{Percolation Upper Bound and Cluster Tail Exponent}
\label{sec:PercolationUpperBounds}

\begin{lemma}\label{lem:magnetupperbound}
    There exists $K<\infty$ such that
    \begin{equation}
        M\left(\lambda_c,q\right)\leq K q^{\frac{1}{2}}
    \end{equation}
    for all $q\in\left(0,\frac{1}{2}\right)$.
\end{lemma}

This leads to a proof of the claim in Theorem~\ref{thm:CritExponents}\ref{thm:CritExponents - Percolation} that for $d=2,3$ there exists $C'<\infty$ and $\varepsilon>0$ such that
\begin{equation}
    \theta\left(\lambda\right) \leq C'\left(\lambda-\lambda_c\right)_+
\end{equation}
for all $\lambda<\lambda_c+\varepsilon$.

\begin{proof}[Proof of the upper bound in Theorem~\ref{thm:CritExponents}\ref{thm:CritExponents - Percolation}]
    This follows exactly the argument of \cite[Section~4.2]{caicedo2023critical}. By using \eqref{eqn:magnetisationDiffEqnPtOne} to bound a slight increase in the $\lambda$ parameter with a slight increase in the $q$ parameter, and using the bound from Lemma~\ref{lem:magnetupperbound}, a bound of the form
    \begin{equation}
        M\left(\lambda_\mathrm{c}+\varepsilon,q\right) \leq c_0 \varepsilon + c_1q^\frac{1}{2}
    \end{equation}
    follows. Then taking $q\searrow 0$ produces the required upper bound on $\theta\left(\lambda\right)$.
\end{proof}

Lemma~\ref{lem:magnetupperbound} is also the key ingredient for the bound on the cluster tail exponent in Theorem~\ref{thm:CritExponents}\ref{thm:CritExponents - tailexponent}. 
This theorem claims that for $d=2,3$ there exist $0<C\leq C'<\infty$ such that for all $n\in\N$
\begin{equation}
        Cn^{-\frac{1}{2}}\leq \p_{\lambda_c}\left(\#\C\left(\orig,\xi^\orig\right)\geq n\right) \leq C'n^{-\frac{1}{2}}.
    \end{equation}

\begin{proof}[Proof of Theorem~\ref{thm:CritExponents}\ref{thm:CritExponents - tailexponent}]
Begin by observing
\begin{multline}
    \p_{\lambda_c} \left(\#\C\left(\orig\right) \geq n\right) =  \frac{\e}{\e-1} \sum_{l \geq n} \left(1- \tfrac 1e\right) \p_{\lambda_c} \left(\#\C\left(\orig\right) = l\right) \\\leq  \frac{\e}{\e-1} \sum_{l \geq n} \left(1- \left(1-\tfrac 1n\right)^l\right) \p_{\lambda_c} \left(\#\C\left(\orig\right) = l\right) \leq \frac{\e}{\e-1} M\left(\lambda_c,\tfrac 1n\right).
\end{multline}
Lemma~\ref{lem:magnetupperbound} then implies
\begin{equation}
    \p_{\lambda_c} \left(\#\C(\orig) \geq n\right) \leq \frac{\e}{\e-1}K n^{-1/2}.
\end{equation}

To prove the lower bound, let $0 \leq q < \tilde q < 1$. Then
\begin{align}
    M\left(\lambda_c,q\right) & \leq q \sum_{l < n} l \p_{\lambda_c} \left(\#\C\left(\orig\right) = l\right) + \sum_{l \geq n} \p_{\lambda_c} \left(\#\C\left(\orig\right) = l\right) \nonumber\\
	& \leq \frac{q}{\tilde q} \sum_{l < n} \e^{l \tilde q} \left(1-\left(1-\tilde q\right)^l\right) \p_{\lambda_c} \left(\#\C\left(\orig\right) = l\right) + \p_{\lambda_c} \left(\#\C\left(\orig\right) \geq n\right) \nonumber\\
		& \leq \frac{q}{\tilde q} \e^{\tilde q n} M(\lambda_c, \tilde{q}) + \p_{\lambda_c} \left(\#\C\left(\orig\right) \geq n\right),
\end{align}
where we have used $\left(1-\left(1-q\right)^l\right) \leq lq$ and $l \tilde q \leq \e^{l \tilde q} \left( 1- \left(\e^{- \tilde q}\right)^l\right) \leq \e^{l \tilde q} \left(1- \left(1-\tilde q\right)^l\right)$. Setting $\tilde q = 1/n$ and using Lemma~\ref{lem:magnetupperbound} then gives
\begin{equation}
    \p_{\lambda_c} \left(\#\C\left(\orig\right) \geq n\right) \geq M(\lambda_c,q) - q\e Kn^{\frac{1}{2}}.
\end{equation}
We now use Lemma~\ref{lem:magnetlowerbound} to get
\begin{equation}
    \p_{\lambda_c} \left(\#\C\left(\orig\right) \geq n\right) \geq q^{\frac{1}{2}} \left(\frac{1}{\sqrt{2}} - K\e n^{\frac{1}{2}}q^{\frac{1}{2}} \right).
\end{equation}
Choosing $q=\frac{1}{8K^2 \e^2 n}$ shows that
\begin{equation}
    \p_{\lambda_c} \left(\#\C\left(\orig\right) \geq n\right) \geq C n^{-1/2}
\end{equation}
for $C := \frac{1}{8K \e}$.
\end{proof}

Next we prove Lemma~\ref{lem:magnetupperbound}.
To arrive at an upper bound for the magnetisation at $\lambda=\lambda_c$, we will derive an inequality relating the magnetisation to the percolation probability. 
To do this, we will use a ``stepping stone'' idea like we did in Section~\ref{sec:SusceptibilityUpperBound}.

Given an isometrically parametrised geodesic $\gamma$ and $S\in\R$, let $H^-_S\left(\gamma\right)$ and $H^+_S\left(\gamma\right)$ be the half-spaces produced using $\gamma$ in the same way $H^-_S$ and $H^+_S$ were produced using $\hat\gamma$ in Section~\ref{sec:SusceptibilityUpperBound}.

Recall $R_* := \sup\left\{r>0\colon \esssup_{x\not\in B_r\left(\orig\right)}\connf(x,\orig) = 1\right\}$, and $\lim_{r\to\infty}\esssup_{x\not\in B_r\left(\orig\right)}\connf\left(x,\orig\right)<1$ implies $R_*<\infty$. Let $\hat\gamma_0$ be an isometrically parametrised geodesic such that $\hat\gamma_0(0)=\orig$ (the rotational symmetry of the model makes the specific choice of $\hat\gamma_0$ unimportant).

Since $\int_\HypDim\connf\left(x,\orig\right)\dd x >0$, there exist distinct $u_1,u_2\in\HypDim$ and $\delta,\varepsilon>0$ such that
\begin{equation}
    \dist{u_1,\orig} > R_*-2\delta,\qquad \dist{u_1,H^-_{2\delta}\left(\hat\gamma_0\right)}>2\delta,
    \qquad \dist{u_2,H^-_{2\delta}\left(\hat\gamma_0\right)} > 2\delta,
\end{equation}
\begin{equation}
    \dist{u_2,\orig} =\dist{u_1,\orig},\qquad\dist{u_1,u_2} > 2\delta,
\end{equation}
and
\begin{equation}
    \essinf_{x\in B_\delta\left(u_1\right),y\in B_\delta\left(\orig\right)}\connf\left(x,y\right) >\varepsilon,\qquad \essinf_{x\in B_\delta\left(u_1\right),y\in B_\delta\left(\orig\right)}\connf\left(x,y\right) >\varepsilon.
\end{equation}
Let $l:=\dist{u_1,\orig}$. Also let $\hat\gamma_1$ denote the isometrically parametrised geodesic such that $\hat\gamma_1\left(0\right)=\orig$ and $\hat\gamma_1\left(l\right) = u_1$, and let $\hat\gamma_2$ denote the isometrically parametrised geodesic such that $\hat\gamma_2\left(0\right)=\orig$ and $\hat\gamma_2\left(l\right) = u_2$. Given $m,n\in\N$, define
\begin{equation}
    H_0 := H^-_{2\delta}\left(\hat\gamma_0\right),\qquad H_1:= H^+_{nl-2\delta}\left(\hat\gamma_1\right),\qquad H_2:= H^+_{ml-2\delta}\left(\hat\gamma_2\right).
\end{equation}
Observe that for sufficiently large $m$ and $n$, the \emph{three} half-spaces $H_0$, $H_1$, and $H_2$ are disjoint. We then construct two families of stepping stones. For all $k\in\N$, define
\begin{align}
    v_k := \hat\gamma_1\left(kl\right),& \qquad V_k:= B_\delta\left(v_k\right),\\
    w_k := \hat \gamma_2\left(kl\right),&\qquad W_k:=B_\delta\left(w_k\right).
\end{align}

\begin{figure}
    \centering
    \begin{tikzpicture}[scale=1.3]
        \begin{scope}
            \clip (-4,-4) rectangle (5,4);
            \draw[dashed,fill=gray!15] (-11,0)  circle (8.66);
            \draw[dashed,fill=gray!15] (5.5,9.53)  circle (8.66);
            \draw[dashed,fill=gray!15] (5.5,-9.53)  circle (8.66);
            \draw[thick] (-5,1.16) -- (5,-4.62);
            \draw[thick] (-5,-1.16) -- (5,4.62);
            \draw[thick] (-6,0) -- (6,0);
        \end{scope}
        \draw (-3,1) node{$H_0$};
        \draw (3.5,2.6) node{$H_1$};
        \draw (2.7,3.7) node{$\hat\gamma_1$};
        \draw (3.5,-2.6) node{$H_2$};
        \draw (2.6,-3.6) node{$\hat\gamma_2$};
        \draw (3.5,0) node[above]{$\hat\gamma_0$};
        \filldraw (-3,0)node[above]{$\orig$} circle (2pt);
        \draw[fill=gray!50,thick] (-1.5,0.87) circle (10pt) node{$V_1$};
        \draw[fill=gray!50,thick] (0,1.73) circle (10pt);
        \draw[fill=gray!50,thick] (1.5,2.60) circle (10pt) node{$V_n$};
        \draw[fill=gray!50,thick] (-1.5,-0.87) circle (10pt) node{$W_1$};
        \draw[fill=gray!50,thick] (0,-1.73) circle (10pt);
        \draw[fill=gray!50,thick] (1.5,-2.60) circle (10pt) node{$W_m$};
    \end{tikzpicture}
    \caption{Sketch of the `stepping stones' used to connect $\orig$ to two disjoint far half-spaces in Section~\ref{sec:PercolationUpperBounds}.}
    \label{fig:SteppingStonesmagnetisation}
\end{figure}

\begin{lemma}
\label{lem:TripartiteDifferentialEqnmagnetisation}
    There exists $c>0$ such that for all $q\in\left(0,\frac{1}{2}\right)$ and $\lambda<\lambda_\mathrm{c}$
    \begin{equation}
        M\left(\lambda,q\right) \geq c M\left(\lambda,q\right)^2 \chi\left(\lambda,q\right).
    \end{equation}
\end{lemma}

This lemma allows us to prove Lemma~\ref{lem:magnetupperbound}.

\begin{proof}[Proof of Lemma~\ref{lem:magnetupperbound}]
    Starting from the result of Lemma~\ref{lem:TripartiteDifferentialEqnmagnetisation} and using \eqref{eqn:magnetisationDerivative}, we get
    \begin{equation}
        M\left(\lambda,q\right) \geq c\left(1-q\right)M\left(\lambda,q\right)^2 \frac{\partial M}{\partial q}\left(\lambda,q\right) \geq \frac{c}{2}M\left(\lambda,q\right)^2 \frac{\partial M}{\partial q}\left(\lambda,q\right)
    \end{equation}
    for $\lambda<\lambda_\mathrm{c}$ and $q\in\left(0,\frac{1}{2}\right)$. Therefore 
    \begin{equation}
        \frac{\partial }{\partial q}\left[M\left(\lambda,q\right)^2\right] \leq \frac{4}{c}.
    \end{equation}
    We now want to use this differential inequality with the boundary being the limit at $q\searrow0$. Since $\lim_{q\searrow 0}M\left(\lambda,q\right) = \theta\left(\lambda\right) = 0$ for $\lambda<\lambda_\mathrm{c}$, we get
    \begin{equation}
        M\left(\lambda,q\right)^2 \leq \frac{4}{c}q,
    \end{equation}
    and
    \begin{equation}
        M\left(\lambda,q\right) \leq \frac{2}{\sqrt{c}}q^\frac{1}{2},
    \end{equation}
    for $\lambda < \lambda_\mathrm{c}$ and $q\in\left(0,\frac{1}{2}\right)$.

    Now \cite[Lemma~3.4]{caicedo2023critical} proves that $\lambda\mapsto M\left(\lambda,q\right)$ is analytic on $\left(0,\infty\right)$ for all $q\in\left(0,1\right)$. In particular it is continuous. Therefore
    \begin{equation}
        M\left(\lambda_\mathrm{c},q\right) \leq \frac{2}{\sqrt{c}}q^\frac{1}{2},
    \end{equation}
    for all $q\in\left(0,\frac{1}{2}\right)$.
\end{proof}

Now we finalise the proof of Lemma~\ref{lem:magnetupperbound} by proving Lemma~\ref{lem:TripartiteDifferentialEqnmagnetisation}.
\begin{proof}[Proof of Lemma~\ref{lem:TripartiteDifferentialEqnmagnetisation}]
    Recall $M\left(\lambda,q\right):= \plaq\left(\conn{\orig}{\Gcal}{\xi^\orig}\right)$. Now for each $x\in\HypDim$ let us define
    \begin{equation}
        F_x:=\left\{x\in\piv{\orig,\Gcal;\xi^{\orig}}\right\}\cap \left\{\dconn{x}{\Gcal}{\xi}\right\},
    \end{equation}
    and then $F:= \bigcup_{x\in\eta}F_x$. Since $F$ implies $\left\{\conn{\orig}{\Gcal}{\xi^\orig}\right\}$,
    \begin{equation}
        M\left(\lambda,q\right) \geq \plaq\left(F\right).
    \end{equation}
    Then by using Mecke's formula we get
    \begin{equation}
        \plaq\left(F\right) = \elaq\left[\sum_{x\in\eta}\Id_{\left\{F_x\text{ in }\xi\right\}}\right] =\lambda\int_{\HypDim}\plaq\left(F_x\text{ in }\xi^{x}\right)\dd x.
    \end{equation}
    To lower bound $\plaq\left(F\right)$ we will find a lower bound for $\int_{\HypDim}\plaq\left(F_x\text{ in }\xi^{x}\right)\dd x$. The aim is to express the probability as integrals of three different events (denoted $E(y)$, $F(v)$, $G(w)$), which are integrated over the three disjoint half spaces $H_0$, $H_1$, $H_2$ -- see Figure~\ref{fig:SteppingStonesmagnetisation}. This bound is formulated in \eqref{eq:EFGintBd}.

    By using the definition of $F_x$ and the translation symmetry of the model, we get
    \begin{align}
        \int_{\HypDim}\plaq\left(F_x\text{ in }\xi^{x}\right)\dd x & = \int_{\HypDim}\plaq\left(x\in\piv{\orig,\Gcal,\xi^{\orig, x}}, \dconn{x}{\Gcal}{\xi^x}\right)\dd x\nonumber\\
         & = \int_{\HypDim}\plaq\left(\orig\in\piv{y,\Gcal,\xi^{\orig, y}}, \dconn{\orig}{\Gcal}{\xi^\orig}\right)\dd y.
    \end{align}
    The objective is now to find a lower bound for the integrand. The event $\left\{\orig\in\piv{y,\Gcal,\xi^{\orig, y}}\right\}$ can be written as
    \begin{equation}
        \left\{\orig\in\piv{y,\Gcal,\xi^{\orig, y}}\right\} = \left\{\nconn{y}{\Gcal}{\xi^y}\right\}\cap\left\{\conn{y}{\Gcal}{\xi^{y,\orig}}\right\},
    \end{equation}
    that is, $y$ is connected to $\Gcal$ when $\orig$ is there, but is not connected when $\orig$ is absent. We now find a lower bound for the second of these events in the intersection by constructing a particular way in which $y$ is connected to $\Gcal$ when $\orig$ is present. First we say that $y$ is connected to $\orig$ when the latter is present (that is, $\conn{y}{\orig}{\xi^{y,\orig}}$) and that this is the only way in which $\C\left(y,\xi^{y,\orig}\right)$ leaves $H_0$ (that is, $\C\left(y,\xi^y\right)\subset H_0$). We then say that there exists a path of non-ghost vertices from $\orig$ to a vertex in $H_1$, and a (vertex) disjoint path of non-ghost vertices from $\orig$ to a vertex in $H_2$. We will label the vertices in these paths as $\left(v_1,v_2,\ldots,v_n\right)$ and $\left(w_1,w_2,\ldots,w_m\right)$ respectively. For notational convenience, we denote $\vec{v} = \left\{v_1,v_2,\ldots,v_n\right\}$ and $\vec{w}= \left\{w_1,w_2,\ldots,w_m\right\}$. We then require that $v_n$ and $w_m$ each connect to $\Gcal$ (that is, $\conn{v_n}{\Gcal}{\xi}$ and $\conn{w_m}{\Gcal}{\xi}$) without leaving $H_1$ or $H_2$ other than via the vertices $\vec{v}$ or $\vec{w}$ (that is, $\C\left(v_n,\xi\left[\eta\setminus\left(\left\{v_1,\ldots,v_{n-1}\right\}\cup\vec{w}\right)\right]\right)\subset H_1$ and $\C\left(w_m,\xi\left[\eta\setminus\left(\vec{v}\cup(\left\{w_1,\ldots,w_{m-1}\right\}\right)\right]\right)\subset H_2$). Note that in this construction we not only have $\conn{y}{\Gcal}{\xi^{y,\orig}}$, but also $\dconn{\orig}{\Gcal}{\xi^\orig}$. Therefore given $y\in\HypDim$,
    \begin{align}
        &\plaq\left(\orig\in\piv{y,\Gcal,\xi^{\orig, y}}, \dconn{\orig}{\Gcal}{\xi^\orig}\right) \nonumber\\
        &\hspace{1cm} \geq (1-q)^{n+m}\plaq\left(\nconn{y}{\Gcal}{\xi^y},\conn{y}{\orig}{\xi^{\orig,y}},
        \C\left(y,\xi^y\right)\subset H_1,\exists \left\{v_i\right\}_{i=1}^n, \left\{w_i\right\}_{i=1}^m\colon \right.\nonumber\\
        &\hspace{2cm}v_i\in\eta\cap V_i\cap \Gcal^c,w_i\in\eta\cap W_i\cap\Gcal^c,\adja{\orig}{v_1}{\xi^\orig},\adja{v_{1}}{v_2}{\xi},\ldots,\adja{v_{n-1}}{v_n}{\xi},\nonumber\\
        &\hspace{2cm}\adja{\orig}{w_1}{\xi^\orig},\adja{w_{1}}{w_2}{\xi},\ldots,\adja{w_{m-1}}{w_m}{\xi},\C\left(v_n,\xi\left[\eta\setminus\left(\left\{v_1,\ldots,v_{n-1}\right\}\cup\vec{w}\right)\right]\right)\subset H_1,\nonumber\\
        &\hspace{2cm}\left. \C\left(w_m,\xi\left[\eta\setminus\left(\vec{v}\cup(\left\{w_1,\ldots,w_{m-1}\right\}\right)\right]\right)\subset H_2,\conn{v_n}{\Gcal}{\xi}, \conn{w_m}{\Gcal}{\xi}\right).
    \end{align}
    As in the proof of Lemma~\ref{lem:SteppingStones}, we can condition on the positions of vertices $\vec{v}\cup\vec{w}$. Since the stepping stones are disjoint the point processes on these sets is independent, and this produces
    \begin{align}
        &\plaq\left(\orig\in\piv{y,\Gcal,\xi^{\orig, y}}, \dconn{\orig}{\Gcal}{\xi^\orig}\right) \nonumber\\
        &\hspace{1cm} \geq (1-q)^{n+m}\left(\prod^n_{i=1}\pla\left(\#\eta\cap V_i\geq 1\right)\right)\left(\prod^m_{i=1}\pla\left(\#\eta\cap W_i\geq 1\right)\right)\nonumber\\
        &\hspace{2cm} \times \int_{V_1}\ldots\int_{V_n}\int_{W_1}\ldots\int_{W_m}  \plaq\left(\nconn{y}{\Gcal}{\xi^y},\conn{y}{\orig}{\xi^{\orig,y}},
        \C\left(y,\xi^y\right)\subset H_0,\right.\nonumber\\
        &\hspace{6cm} \adja{\orig}{v_1}{\xi^\orig},\adja{v_{1}}{v_2}{\xi},\ldots,\adja{v_{n-1}}{v_n}{\xi},\nonumber\\
        &\hspace{6cm}\adja{\orig}{w_1}{\xi^\orig},\adja{w_{1}}{w_2}{\xi},\ldots,\adja{w_{m-1}}{w_m}{\xi},\nonumber\\
        &\hspace{6cm} \C\left(v_n,\xi\left[\eta\setminus\left(\left\{v_1,\ldots,v_{n-1}\right\}\cup\vec{w}\right)\right]\right)\subset H_1,\nonumber\\
        &\hspace{6cm} \C\left(w_m,\xi\left[\eta\setminus\left(\vec{v}\cup(\left\{w_1,\ldots,w_{m-1}\right\}\right)\right]\right)\subset H_2,\nonumber\\
        &\hspace{6cm} \left.\conn{v_n}{\Gcal}{\xi}, \conn{w_m}{\Gcal}{\xi} \;\middle|\; v_1,\ldots,v_n,w_1,\ldots,w_m\in\eta\right)\nonumber\\
        &\hspace{11cm}\dd w_m \ldots \dd w_1\dd v_n \ldots \dd v_1\nonumber\\
        &\hspace{1cm} \geq (1-q)^{n+m}\left(\prod^n_{i=1}\pla\left(\#\eta\cap V_i\geq 1\right)\right)\left(\prod^m_{i=1}\pla\left(\#\eta\cap W_i\geq 1\right)\right)\nonumber\\
        &\hspace{2cm} \times \int_{V_1}\ldots\int_{V_n}\int_{W_1}\ldots\int_{W_m}  \plaq\left(\nconn{y}{\Gcal}{\xi^{y,\vec{v},\vec{w}}},\conn{y}{\orig}{\xi^{\orig,y,\vec{v},\vec{w}}},
        \C\left(y,\xi^{y,\vec{v},\vec{w}}\right)\subset H_0,\right.\nonumber\\
        &\hspace{6cm} \adja{\orig}{v_1}{\xi^{\orig,v_1}},\adja{v_{1}}{v_2}{\xi^{v_1,v_2}},\ldots,\adja{v_{n-1}}{v_n}{\xi^{v_{n-1},v_n}},\nonumber\\
        &\hspace{6cm}\adja{\orig}{w_1}{\xi^{\orig,w_1}},\adja{w_{1}}{w_2}{\xi^{w_1,w_2}},\ldots,\adja{w_{m-1}}{w_m}{\xi^{w_{m-1},w_m}},\nonumber\\
        &\hspace{6cm} \C\left(v_n,\xi^{v_n}\right)\subset H_1, \C\left(w_m,\xi^{w_m}\right)\subset H_2,\nonumber\\
        &\hspace{6cm} \left.\conn{v_n}{\Gcal}{\xi^{\vec{v},\vec{w}}}, \conn{w_m}{\Gcal}{\xi^{\vec{v},\vec{w}}}\right)\nonumber\\
        &\hspace{11cm}\dd w_m \ldots \dd w_1\dd v_n \ldots \dd v_1,
    \end{align}
    where we have omitted the extra vertices from the adjacency events because the adjacency is independent of the other vertices. Using this independence of the adjacency events and the bounds $\connf\left(v_i,v_{i+1}\right)\geq \varepsilon$ and $\connf\left(w_i,w_{i+1}\right)\geq \varepsilon$, and removing vertices from increasing events gives
    \begin{align}
        &\plaq\left(\orig\in\piv{y,\Gcal,\xi^{\orig, y}}, \dconn{\orig}{\Gcal}{\xi^\orig}\right) \nonumber\\
        &\hspace{1cm} \geq \varepsilon^{n+m}\left(\prod^n_{i=1}\pla\left(\#\eta\cap V_i\geq 1\right)\right)\left(\prod^m_{i=1}\pla\left(\#\eta\cap W_i\geq 1\right)\right)\nonumber\\
        &\hspace{2cm} \times \int_{V_1}\ldots\int_{V_n}\int_{W_1}\ldots\int_{W_m}  \plaq\left(\nconn{y}{\Gcal}{\xi^{y,\vec{v},\vec{w}}},\conn{y}{\orig}{\xi^{\orig,y}},
        \C\left(y,\xi^{y,\vec{v},\vec{w}}\right)\subset H_0,\right.\nonumber\\
        &\hspace{7cm} \C\left(v_n,\xi^{v_n}\right)\subset H_1, \C\left(w_m,\xi^{w_m}\right)\subset H_2,\nonumber\\
        &\hspace{7cm} \left.\conn{v_n}{\Gcal}{\xi^{v_n}}, \conn{w_m}{\Gcal}{\xi^{w_m}}\right)\nonumber\\
        &\hspace{11cm}\dd w_m \ldots \dd w_1\dd v_n \ldots \dd v_1.
    \end{align}
    
By the same argument as in Lemma~\ref{lem:SteppingStones}, we can omit the stepping stone vertices from some of the events at the cost of a constant factors:
\begin{align}
        &\plaq\left(\nconn{y}{\Gcal}{\xi^{y,\vec{v},\vec{w}}},\conn{y}{\orig}{\xi^{\orig,y}},\C\left(y,\xi^{y,\vec{v},\vec{w}}\right)\subset H_0, \C\left(v_n,\xi^{v_n}\right)\subset H_1, \C\left(w_m,\xi^{w_m}\right)\subset H_2,\right.\nonumber\\
        &\hspace{8cm} \left.\conn{v_n}{\Gcal}{\xi^{v_n}},\conn{w_m}{\Gcal}{\xi^{w_m}}\right)\nonumber\\
        &\hspace{1cm} \geq c\,\plaq\left(\nconn{y}{\Gcal}{\xi^{y}},\conn{y}{\orig}{\xi^{\orig,y}},\C\left(y,\xi^{y}\right)\subset H_0, \C\left(v_n,\xi^{v_n}\right)\subset H_1, \C\left(w_m,\xi^{w_m}\right)\subset H_2,\right.\nonumber\\
        &\hspace{8cm} \left.\conn{v_n}{\Gcal}{\xi^{v_n}},\conn{w_m}{\Gcal}{\xi^{w_m}}\right) - \Id_{\left\{y\in\Ucal_\varepsilon\right\}}\nonumber\\
        &\hspace{1cm} \geq c\,\plaq\left(E\left(y\right)\cap F\left(v_n\right)\cap G\left(w_m\right)\right) - \Id_{\left\{y\in\Wcal_\varepsilon\right\}},
\end{align}
where 
\begin{align}
    E(y) &:= \left\{\nconn{y}{\Gcal}{\xi^y}\right\}\cap\left\{\conn{y}{\orig}{\xi^{y,\orig}}\right\}\cap\left\{\C\left(y,\xi^y\right)\subset H_0\right\},\\
    F(v) &:= \left\{\conn{v}{\Gcal}{\xi^v}\right\}\cap \left\{\C\left(v,\xi^v\right)\subset H_1\right\},\\
    G(w) &:= \left\{\conn{w}{\Gcal}{\xi^w}\right\}\cap \left\{\C\left(w,\xi^w\right)\subset H_2\right\}.
\end{align}

We therefore have that there exists $c,c'>0$ such that for $m$ and $n$ sufficiently large and $\lambda\in\left[\lambda_{\min},\lambda_\mathrm{c}\right)$,
\begin{equation}\label{eq:EFGintBd}
    \plaq\left(F\right) \geq c\int_{H_0}\int_{V_n}\int_{W_m}\plaq\left(E\left(y\right)\cap F\left(v\right)\cap G\left(w\right)\right)\dd w \dd v\dd y - c'.
\end{equation}

\vskip.5em
We next proceed as before and factorise the three events appearing in \eqref{eq:EFGintBd} after conditioning on a suitable separation event. 
Repeating the argument of Lemma~\ref{lem:susceptibilityderivativeLowerBound}, the events in the probability become asymptotically independent as $m,n\to\infty$. Therefore there exists $c''>0$ such that for $m$ and $n$ sufficiently large and $\lambda\in\left[\lambda_{\min},\lambda_\mathrm{c}\right)$,
\begin{align}
    &\int_{H_0}\int_{V_n}\int_{W_m}\plaq\left(E\left(y\right)\cap F\left(v\right)\cap G\left(w\right)\right)\dd w \dd v\dd y \nonumber\\
    &\hspace{5cm}\geq c''\int_{H_0}\int_{V_n}\int_{W_m}\plaq\left(E\left(y\right)\right)\plaq\left(F\left(v\right)\right)\plaq\left(G\left(w\right)\right)\dd w \dd v\dd y\nonumber\\
    &\hspace{5cm}= c''\int_{H_0}\plaq\left(E\left(y\right)\right)\dd y \int_{V_n}\plaq\left(F\left(v\right)\right) \dd v \int_{W_m}\plaq\left(G\left(w\right)\right)\dd w.
    \label{eq:EFGprodBd}
\end{align}

We now need to lower bound the three integrals on the RHS of \eqref{eq:EFGprodBd}, and do this one by one.
By Mecke's formula, 
\begin{align}
    \int_{H_0}\plaq\left(E\left(y\right)\right)\dd y &= \int_{H_0}\plaq\left(\nconn{y}{\Gcal}{\xi^y}, \conn{y}{\orig}{\xi^{y,\orig}},\C\left(y,\xi^y\right)\subset H_0\right)\dd y\nonumber\\
    &= \frac{1}{\lambda}\elaq\left[\#\left\{y\in\eta\colon \nconn{y}{\Gcal}{\xi}, \conn{y}{\orig}{\xi^{\orig}},\C\left(y,\xi\right)\subset H_0\right\}\right]\nonumber\\
    &\geq \frac{1}{\lambda_{\mathrm{c}}}\elaq\left[\#\left\{y\in\eta\colon \nconn{y}{\Gcal}{\xi}, \conn{y}{\orig}{\xi^{\orig}},\C\left(y,\xi\right)\subset H_0\right\}\right]\nonumber\\
    & \geq \frac{1}{\lambda_{\mathrm{c}}}\left(\elaq\left[\#\C\left(\orig,\xi^{\orig}\right)\Id_{\left\{\C\left(\orig,\xi^{\orig}\right)\cap \Gcal = \emptyset\right\}} \Id_{\left\{\C\left(\orig,\xi^{\orig}\right)\subset H_0\right\}}\right]-1\right).
\end{align}

    To manage the $v$ and $w$ integrals, we transform the half spaces $H_1$ and $H_2$ for each value of $v\in V_n$ and $w\in W_m$. Like at the end of the proof of Lemma~\ref{lem:susceptibilityderivativeLowerBound}, for all $v\in V_n$, there exists a unique geodesic $\gamma'$ that contains $v$ and is perpendicular to $\partial H_1$ (we take the `positive' direction on $\gamma'$ to be the direction towards $\partial H_1$ from $\orig$). Let $\iota$ be an isometry that maps $v\mapsto\orig$ and $\gamma'\mapsto\hat{\gamma}_0$ (matching the `positive' directions). If we define $H':=\iota\left(H_1\right)$, then we find
    \begin{equation}
        H' = H^-_{\dist{v,\partial H_1}}\left(\hat\gamma_0\right).
    \end{equation}
    Since $\inf_{v\in V_n}\dist{v,\partial H_1}\geq \delta$,
    \begin{multline}
        \int_{V_n}\plaq\left(F\left(v\right)\right) \dd v = \int_{V_n}\plaq\left(\conn{\orig}{\Gcal}{\xi^\orig},\C\left(\orig,\xi^\orig\right)\subset H'\right)\dd v \\\geq \habsd{B_\delta\left(\orig\right)}\plaq\left(\conn{\orig}{\Gcal}{\xi^\orig},\C\left(\orig,\xi^\orig\right)\subset H^-_{\delta}\left(\hat\gamma_0\right)\right).
    \end{multline}
    Similarly, 
    \begin{equation}
        \int_{W_m}\plaq\left(G\left(w\right)\right) \dd w \geq \habsd{B_\delta\left(\orig\right)}\plaq\left(\conn{\orig}{\Gcal}{\xi^\orig},\C\left(\orig,\xi^\orig\right)\subset H^-_\delta\left(\hat\gamma_0\right)\right).
    \end{equation}
    Together, the three bounds above yield a lower bound for \eqref{eq:EFGprodBd}. Together with \eqref{eq:EFGintBd} we thus get the following lower bound:
    there exists $c'''>0$ such that for $\lambda\in\left[\lambda_{\min},\lambda_{\mathrm{c}}\right)$, 
    \begin{multline}
        M\left(\lambda,q\right) \geq \plaq\left(F\right) \geq c''' \elaq\left[\#\C\left(\orig,\xi^{\orig}\right)\Id_{\left\{\C\left(\orig,\xi^{\orig}\right)\cap \Gcal = \emptyset\right\}} \Id_{\left\{\C\left(\orig,\xi^{\orig}\right)\subset H^-_\varepsilon\left(\hat\gamma_0\right)\right\}}\right] \\\times\plaq\left(\conn{\orig}{\Gcal}{\xi^\orig},\C\left(\orig,\xi^\orig\right)\subset H^-_\delta\left(\hat\gamma_0\right)\right)^2 - c'.
    \end{multline}
    
    \vskip.5em
    In the final step of the proof, we employ Proposition~\ref{prop:BoundHalf-SpaceSusceptMagnet} to replace the terms on the right hand side with $M\left(\lambda,q\right)$ and $\chi\left(\lambda,q\right)$ at the cost of a constant multiplicative factor. This means that there exist $\varepsilon,K>0$ such that 
    \begin{equation}
        M\left(\lambda,q\right) \geq KM\left(\lambda,q\right)^2\chi\left(\lambda,q\right)
    \end{equation}
    for all $\lambda\in\left(\lambda_c-\varepsilon,\lambda_c \right)$ and $q\in\left(0,\frac{1}{2}\right)$. To extend this to $\lambda\in\left[0,\lambda_c\right)$, observe that $M\left(\lambda,q\right)\in\left[0,1\right]$ and therefore 
    \begin{equation}
        M\left(\lambda,q\right) \leq M\left(\lambda,q\right)^2.
    \end{equation}
    For $\lambda\in\left[0,\lambda_c-\varepsilon\right]$ we have
    \begin{equation}
        \chi\left(\lambda,q\right) \leq \chi\left(\lambda\right) \leq \chi\left(\lambda_c-\varepsilon\right)<
        \infty,
    \end{equation}
    and therefore
    \begin{equation}
        M\left(\lambda,q\right) \geq \frac{1}{\chi\left(\lambda_c-\varepsilon\right)}M\left(\lambda,q\right)^2\chi\left(\lambda,q\right).
    \end{equation}
    Therefore for all $\lambda\in\left[0,\lambda_c\right)$ and $q\in\left(0,\frac{1}{2}\right)$
    \begin{equation}
        M\left(\lambda,q\right) \geq c M\left(\lambda,q\right)^2\chi\left(\lambda,q\right)
    \end{equation}
    with $c=\min\left\{K,\chi\left(\lambda_c-\varepsilon\right)^{-1}\right\}>0$.
\end{proof}

This now completes the proof of Theorem~\ref{thm:CritExponents}\ref{thm:CritExponents - Percolation} and \ref{thm:CritExponents}\ref{thm:CritExponents - tailexponent}.

\subsection{Cluster Moment Exponent}
\label{sec:clustermomentexponent}
The derivation of the cluster moment exponent here closely follows the analogous argument in \cite{Ngu87}, with versions of lemmas from \cite{AizNew84} and \cite{durrett1985thermodynamic}.

We first bound the $n$-point function using tree diagrams. Given $n\geq 2$, $\lambda>0$, and $x_1,\ldots,x_n\in\HypDim$, define the $n$-point function to be
    \begin{equation}
        \tau_{n,\lambda}\left(x_1,\ldots,x_n\right) := \pla\left(x_1\longleftrightarrow x_2 \longleftrightarrow \ldots \longleftrightarrow x_n \text{ in }\xi^{x_1,\ldots,x_n}\right).
    \end{equation}
To bound $\tau_{n,\lambda}$, we will make use of Greg trees. A \emph{Greg tree with $n$ labelled vertices} is defined to be a tree with two colours of vertices: there are $n$ `blue' labelled vertices with degrees $\geq 1$, and $m$ `red' unlabelled vertices with degree $\geq 3$. These can be constructed iteratively in the manner described in \cite{flight1990many}. The only Greg tree with $1$ labelled vertex is one blue vertex by itself. Given the set of Greg trees with $n$ labelled vertices, the set of Greg trees with $n+1$ labelled vertices can be constructed by taking each Greg tree with $n$ labelled vertices and doing one of the following with the extra labelled (blue) vertex:
\begin{enumerate}
    \item \label{enum:GregOptionOne}attach the new blue vertex to any pre-existing vertex (blue or red),
    \item \label{enum:GregOptionTwo} insert the new blue vertex into the middle of an edge,
    \item \label{enum:GregOptionThree} insert a red vertex into the middle of an edge and attach the new blue vertex to that red vertex,
    \item \label{enum:GregOptionFour} replace a pre-existing red vertex with the new blue vertex.
\end{enumerate}
The following diagrams represent some examples of option~\ref{enum:GregOptionOne} for the iterative step:
\begin{equation}
    \GregFourEg \mapsto \GregFourEgStepOneVOne \qquad\text{ or }\qquad \GregFourEg \mapsto \GregFourEgStepOneVTwo.
\end{equation}
Here is an example of option~\ref{enum:GregOptionTwo}:
\begin{equation}
    \GregFourEg \mapsto \GregFourEgStepTwo.
\end{equation}
Here is an example of option~\ref{enum:GregOptionThree}:
\begin{equation}
    \GregFourEg \mapsto \GregFourEgStepThree.
\end{equation}
Here is an example of option~\ref{enum:GregOptionFour}:
\begin{equation}
    \GregFourEg \mapsto \GregFourEgStepFour.
\end{equation}
Observe that these options each add at most $2$ edges to the tree (attained by option~\ref{enum:GregOptionThree}). Therefore a Greg tree with $n$ labelled vertices has at most $2n-3$ edges (for $n\geq 2$). In particular, this means that the number of Greg trees with $n$ labelled vertices, $N_n$, is finite, and it has been proven in \cite{foulds1980determining} that
\begin{equation}
    N_n = \frac{1}{\sqrt{2}}\frac{1}{n^2}\e^{\frac{3}{4}-n}\left(2\e^{-\frac{1}{2}}-1\right)^{\frac{3}{2}-n}n^n\left(1+o\left(1\right)\right)
\end{equation}
as $n\to\infty$\footnote{The sequence $N_n$ is sequence $A005263$ on the Online Encyclopedia of Integer Sequences.}.

Let $G_n$ denote the set of Greg trees with $n$ labelled vertices. Then given $g\in G_n$, define $r(g)$ to be the number of red (unlabelled) vertices in $g$, and $E(g)$ to be the set of edges in $g$. Then for $n\geq 2$, $\lambda>0$, and $x_1,\ldots,x_n\in \HypDim$ define
\begin{equation}
    T_{n,\lambda}\left(x_1,\ldots,x_n\right):= \sum_{g\in G_n}\lambda^{r(g)}\int_{\left(\HypDim\right)^{r(g)}}\prod_{\left\{z,z'\right\}\in E(g)}\tlam\left(z,z'\right)\dd y_1 \ldots y_{r(g)}.
\end{equation}
If $r(g)=0$, the `integral' just returns the value of the integrand.

\begin{lemma}
\label{lem:nPointvsTree}
    For all $n\geq 2$, $\lambda>0$, and $x_1,\ldots,x_n\in\HypDim$,
    \begin{equation}
        \tau_{n,\lambda}\left(x_1,\ldots,x_n\right) \leq T_{n,\lambda}\left(x_1,\ldots,x_n\right).
    \end{equation}
\end{lemma}

\begin{proof}
    Suppose that the vertices $x_1,\ldots,x_n$ are in the same cluster in $\xi^{x_1,\ldots,x_n}$. Then there exists a spanning tree of this cluster in $\xi^{x_1,\ldots,x_n}$ that, of course, contains $\left\{x_1,\ldots,x_n\right\}$. Let us prune away the branches of the tree beyond any of the vertices $\left\{x_1,\ldots,x_n\right\}$ - they will be irrelevant for our considerations. For the remaining tree, let $\left\{y_1,\ldots,y_m\right\}$ denote the vertices that have degree $\geq 3$ that are not one of the initial $\left\{x_1,\ldots,x_n\right\}$. The remaining vertices of our tree will then lie on one of the disjoint paths between elements of $\left\{x_1,\ldots,x_n,y_1,\ldots,y_m\right\}$. Contracting these paths to edges will then produce a tree graphs that details how $\left\{x_1,\ldots,x_n,y_1,\ldots,y_m\right\}$ can be connected by disjoint paths. If $\left\{E_i\right\}_{i\in I}$ is a finite set of events, let us denote
    \begin{equation}
        \underset{i\in I}{\bigcirc}E_i := E_{i_1}\circ\left( E_{i_2}\circ\ldots\circ \left(E_{i_{n-1}}\circ E_{i_n}\right)\ldots\right),
    \end{equation}
    where $n=\# I$, and $E\circ F$ denotes the vertex disjoint occurrence of the events $E$ and $F$ (an associative binary operation on increasing events) -- see \cite[Section~2.3]{HeyHofLasMat19} for details. Then by a union bound and Mecke's formula,    
\begin{align}
    \tau_{n,\lambda}\left(x_1,\ldots,x_n\right) &\leq \pla\left(\exists g\in G_n, \left\{y_1,\ldots,y_m\right\}\subset \eta\colon V(g)=\left\{x_1,\ldots,x_n,y_1,\ldots,y_m\right\} \right.\nonumber\\
    &\hspace{5cm}\left.\text{ and } \underset{\left\{z,z'\right\}\in E(g)}{\bigcirc}\left\{\conn{z}{z'}{\xi^{x_1,\ldots,x_n}}\right\}\right)\nonumber\\
    &\leq \sum_{g\in G_n}\ela\left[\sum_{\left(y_1,\ldots,y_{r(g)}\right)\in\eta^{\left(r(g)\right)}}\Id_{\left\{\underset{\left\{z,z'\right\}\in E(g)}{\bigcirc}\left\{\conn{z}{z'}{\xi^{x_1,\ldots,x_n}}\right\}\right\}}\right]\nonumber\\
    & = \sum_{g\in G_n}\lambda^{r(g)}\int_{\left(\HypDim\right)^{r(g)}} \pla\left(\underset{\left\{z,z'\right\}\in E(g)}{\bigcirc}\left\{\conn{z}{z'}{\xi^{x_1,\ldots,x_n,y_1,\ldots,y_{r(g)}}}\right\}\right) \dd y_1\ldots \dd y_{r(g)}
\end{align}
Now because we are taking the disjoint occurrence of connection events inside the probability, all the augmenting vertices other than $z$ and $z'$ can be omitted from the configuration $\xi^{x_1,\ldots,x_n,y_1,\ldots,y_{r(g)}}$ - they are being used on other paths. That is,
\begin{equation}
    \pla\left(\underset{\left\{z,z'\right\}\in E(g)}{\bigcirc}\left\{\conn{z}{z'}{\xi^{x_1,\ldots,x_n,y_1,\ldots,y_{r(g)}}}\right\}\right)
    =  \pla\left(\underset{\left\{z,z'\right\}\in E(g)}{\bigcirc}\left\{\conn{z}{z'}{\xi^{z,z'}}\right\}\right).
\end{equation}
Then by the BK inequality \cite[Thm.\ 2.1]{HeyHofLasMat19},
\begin{equation}
    \pla\left(\underset{\left\{z,z'\right\}\in E(g)}{\bigcirc}\left\{\conn{z}{z'}{\xi^{z,z'}}\right\}\right) \leq \prod_{\left\{z,z'\right\}\in E(g)}\pla\left(\conn{z}{z'}{\xi^{z,z'}}\right)
\end{equation}
This then produces
\begin{multline}
    \tau_{n,\lambda}\left(x_1,\ldots,x_n\right)  \leq \sum_{g\in G_n}\lambda^{r(g)}\int_{\left(\HypDim\right)^{r(g)}} \prod_{\left\{z,z'\right\}\in E(g)}\pla\left(\conn{z}{z'}{\xi^{z,z'}}\right) \dd y_1\ldots \dd y_{r(g)}\\
     = T_{n,\lambda}\left(x_1,\ldots,x_n\right)\label{eqn:TreeGraphInequality}
\end{multline}
as required.
\end{proof}

\begin{lemma}
\label{lem:nclusterUpper}
    For all $n\geq 1$, there exists $C_n <\infty$ such that
    \begin{equation}
        \ela\left[\#\C\left(\orig,\xi^{\orig}\right)^n\right] \leq C_n \chi\left(\lambda\right)^{2n-1}
    \end{equation}
    for all $\lambda\geq0$.
\end{lemma}

\begin{proof}    
    By a binomial expansion and Mecke's formula,
    \begin{align}
        \ela\left[\#\C\left(\orig,\xi^{\orig}\right)^n\right] &= \ela\left[\left(1+\sum_{x\in\eta}\Id_{\left\{\conn{\orig}{x}{\xi^{\orig}}\right\}}\right)^n\right]\nonumber\\
        &= 1+ \sum_{k=1}^n \binom{n}{k} \ela\left[\sum_{\left(x_1,\ldots,x_k\right)\in\eta^{(k)}}\Id_{\left\{\orig\longleftrightarrow x_1 \longleftrightarrow \ldots \longleftrightarrow x_k \text{ in }\xi^\orig\right\}}\right]\nonumber\\
        &= 1  + \sum^n_{k=1}\binom{n}{k}\lambda^k\int_{\left(\HypDim\right)^k}\tau_{k+1,\lambda}\left(\orig,x_1,\ldots,x_k\right)\dd x_1\ldots \dd x_k.\label{eqn:NMomentEquality}
    \end{align}
From Lemma~\ref{lem:nPointvsTree} we therefore have
\begin{equation}
    \ela\left[\#\C\left(\orig,\xi^{\orig}\right)^n\right] \leq 1  + \sum^n_{k=1}\binom{n}{k}\lambda^k\int_{\left(\HypDim\right)^k}T_{k+1,\lambda}\left(\orig,x_1,\ldots,x_k\right)\dd x_1\ldots \dd x_k.
\end{equation}

By translation invariance and `peeling off' leaves from the Greg trees $g\in G_{k+1}$ in definition of $T_{k+1,\lambda}$, we can arrive at
\begin{multline}
    \lambda^k\int_{\left(\HypDim\right)^k}T_{k+1,\lambda}\left(\orig,x_1,\ldots,x_k\right) \dd x_1\ldots \dd x_k  = \sum_{g\in G_{k+1}}\lambda^{k+r(g)}\left(\int_{\HypDim}\tlam\left(\orig,z\right)\dd z\right)^{\#E(g)}\\
    = \sum_{g\in G_{k+1}}\lambda^{\#E(g)}\left(\int_{\HypDim}\tlam\left(\orig,z\right)\dd z\right)^{\#E(g)},
\end{multline}
where we have used the fact that the Greg trees we are summing over have $k+1+r(g)$ vertices in total and (since they are trees) will have $\# E(g)=k+r(g)$. Mecke's formula then leads to
\begin{multline}
    \sum_{g\in G_{k+1}}\lambda^{\#E(g)}\left(\int_{\HypDim}\tlam\left(\orig,z\right)\dd z\right)^{\#E(g)} = \sum_{g\in G_{k+1}}\left(\ela\left[\#\C\left(\orig,\xi^{\orig}\right)\right]-1\right)^{\#E(g)}\\
    \leq\sum_{g\in G_{k+1}}\ela\left[\#\C\left(\orig,\xi^{\orig}\right)\right]^{\#E(g)}.
\end{multline}
Recall that from the iterative construction of Greg trees that for any $k\geq 1$ and $g\in G_{k+1}$, we have $\# E(g) \leq 2k-1$. Since $\ela\left[\#\C\left(\orig,\xi^{\orig}\right)\right]\geq 1$, we therefore have
\begin{equation}
    \lambda^k\int_{\left(\HypDim\right)^k}T_{k+1,\lambda}\left(\orig,x_1,\ldots,x_k\right) \dd x_1\ldots \dd x_k \leq N_{k+1}\ela\left[\#\C\left(\orig,\xi^{\orig}\right)\right]^{2k-1}.
\end{equation}

Combining this inequality with \eqref{eqn:NMomentEquality} and \eqref{eqn:TreeGraphInequality} then gives
\begin{equation}
    \ela\left[\#\C\left(\orig,\xi^{\orig}\right)^n\right] \leq 1 + \sum^n_{k=1}\binom{n}{k}N_{k+1}\ela\left[\#\C\left(\orig,\xi^{\orig}\right)\right]^{2k-1} \leq \left(n+1\right)!N_{n+1}\ela\left[\#\C\left(\orig,\xi^{\orig}\right)\right]^{2n-1},
\end{equation}
as required.
\end{proof}

This will be sufficient for the upper bound on $\ela\left[\#\C\left(\orig,\xi^{\orig}\right)^n\right]$. We need to prepare one more lemma to get our lower bound.

\begin{lemma}
\label{lem:SecondMomentLowerBound}
    For all $\lambda\in\left(0,\lambda_T\right)$,
    \begin{equation}
        \ela\left[\#\C\left(\orig,\xi^{\orig}\right)^2\right] \geq \frac{\lambda^2}{\chi\left(\lambda\right)}\left(\frac{\dd \chi}{\dd \lambda}\right)^2.
    \end{equation}
\end{lemma}

\begin{proof}
    In \cite{caicedo2023critical} the following expression was derived: for all $k\in\N_0$ and $\lambda> 0$
    \begin{equation}
        \pla\left(\#\C\left(\orig,\xi^{\orig}\right) = k+1\right) = \frac{\lambda^k}{k!} \int_{\left(\HypDim\right)^k}f_k\left(\vec{x}\right)\exp\left(-\lambda g_k\left(\vec{x}\right)\right)\dd x_1\ldots \dd x_{k},
    \end{equation}
    where $\vec{x}=\left(x_1,\ldots,x_{k}\right)$,
    \begin{align}
        f_k\left(\vec{x}\right) &:= \sum_{G\in\mathrm{Gr}_{k+1}}\left(\prod_{\left\{i,j\right\}\in E\left(G\right)}\connf\left(x_i,x_j\right)\right)\left(\prod_{\left\{i,j\right\}\not\in E\left(G\right)}\left(1-\connf\left(x_i,x_j\right)\right)\right),\\
        g_k\left(\vec{x}\right) &:= \int_{\HypDim}\left(1 - \prod_{m\in\left\{0,\ldots,k\right\}}\left(1-\connf\left(y,x_m\right)\right)\right) \dd y,
    \end{align}
    $\mathrm{Gr}_{k+1}$ is the set of simple connected graphs on $\left\{0,\ldots, k\right\}$, $E\left(G\right)$ is the set of edges in the graph $G$, and $x_0=\orig$. Then
    \begin{align}
        \ela\left[\#\C\left(\orig,\xi^{\orig}\right) \;\middle|\;\#\C\left(\orig,\xi^{\orig}\right)<\infty\right] &= \sum^\infty_{k=0}\left(k+1\right)\frac{\lambda^k}{k!} \int_{\left(\HypDim\right)^k}f_k\left(\vec{x}\right)\exp\left(-\lambda g_k\left(\vec{x}\right)\right)\dd x_1\ldots \dd x_{k}\\
        \frac{\dd}{\dd \lambda} \ela\left[\#\C\left(\orig,\xi^{\orig}\right) \;\middle|\;\#\C\left(\orig,\xi^{\orig}\right)<\infty\right] &= \sum^\infty_{k=0}\left(k+1\right)\frac{\lambda^k}{k!} \int_{\left(\HypDim\right)^k}\left(\frac{k}{\lambda}-g_k\left(\vec{x}\right)\right)f_k\left(\vec{x}\right)\exp\left(-\lambda g_k\left(\vec{x}\right)\right)\dd x_1\ldots \dd x_{k}.
    \end{align}
    We are able to exchange differentiation with the summation and integration for the same reasons as in \cite[Lemma~3.5]{caicedo2023critical} in which they considered the analogous argument for the magnetisation. We omit the details here. By the Cauchy-Schwarz inequality we then get
    \begin{align}
        &\left(\frac{\dd}{\dd \lambda} \ela\left[\#\C\left(\orig,\xi^{\orig}\right) \;\middle|\;\#\C\left(\orig,\xi^{\orig}\right)<\infty\right]\right)^2\nonumber\\
        &\hspace{1cm}\leq \left(\sum^\infty_{k=0}\left(k+1\right)^2\frac{\lambda^k}{k!} \int_{\left(\HypDim\right)^k}f_k\left(\vec{x}\right)\exp\left(-\lambda g_k\left(\vec{x}\right)\right)\dd x_1\ldots \dd x_{k}\right)\nonumber\\
        &\hspace{5cm}\times\left(\sum^\infty_{k=0}\frac{\lambda^k}{k!} \int_{\left(\HypDim\right)^k}\left(\frac{k}{\lambda}-g_k\left(\vec{x}\right)\right)^2f_k\left(\vec{x}\right)\exp\left(-\lambda g_k\left(\vec{x}\right)\right)\dd x_1\ldots \dd x_{k}\right)\nonumber\\
        &\hspace{1cm} =\ela\left[\#\C\left(\orig,\xi^{\orig}\right)^2 \;\middle|\;\#\C\left(\orig,\xi^{\orig}\right)<\infty\right]\nonumber\\
        &\hspace{5cm}\times\left(\sum^\infty_{k=0}\frac{\lambda^k}{k!} \int_{\left(\HypDim\right)^k}\left(\frac{k}{\lambda}-g_k\left(\vec{x}\right)\right)^2f_k\left(\vec{x}\right)\exp\left(-\lambda g_k\left(\vec{x}\right)\right)\dd x_1\ldots \dd x_{k}\right).
    \end{align}
    To manipulate this second factor, observe that
    \begin{equation}
        \theta\left(\lambda\right) = 1 - \sum^\infty_{k=0}\frac{\lambda^k}{k!} \int_{\left(\HypDim\right)^k}f_k\left(\vec{x}\right)\exp\left(-\lambda g_k\left(\vec{x}\right)\right)\dd x_1\ldots \dd x_{k},
    \end{equation}
    and that for $\lambda<\lambda_T$ (so $\theta\left(\lambda\right)=0$)
    \begin{align}
        0 = \frac{\dd \theta}{\dd \lambda} &= - \sum^\infty_{k=0}\frac{\lambda^k}{k!} \int_{\left(\HypDim\right)^k}\left(\frac{k}{\lambda}-g_k\left(\vec{x}\right)\right)f_k\left(\vec{x}\right)\exp\left(-\lambda g_k\left(\vec{x}\right)\right)\dd x_1\ldots \dd x_{k},\\
        0 = \frac{\dd^2 \theta}{\dd \lambda^2} &= \sum^\infty_{k=0}\frac{\lambda^k}{k!} \int_{\left(\HypDim\right)^k}\left(\frac{k}{\lambda^2}\right)f_k\left(\vec{x}\right)\exp\left(-\lambda g_k\left(\vec{x}\right)\right)\dd x_1\ldots \dd x_{k} \nonumber\\
        &\hspace{3cm}- \sum^\infty_{k=0}\frac{\lambda^k}{k!} \int_{\left(\HypDim\right)^k}\left(\frac{k}{\lambda}-g_k\left(\vec{x}\right)\right)^2f_k\left(\vec{x}\right)\exp\left(-\lambda g_k\left(\vec{x}\right)\right)\dd x_1\ldots \dd x_{k}.
    \end{align}
    That is, for $\lambda<\lambda_T$
    \begin{multline}
        \sum^\infty_{k=0}\frac{\lambda^k}{k!} \int_{\left(\HypDim\right)^k}\left(\frac{k}{\lambda}-g_k\left(\vec{x}\right)\right)^2f_k\left(\vec{x}\right)\exp\left(-\lambda g_k\left(\vec{x}\right)\right)\dd x_1\ldots \dd x_{k}\\
        = \frac{1}{\lambda^2}\sum^\infty_{k=0}k\pla\left(\#\C\left(\orig,\xi^{\orig}\right) = k+1\right) = \frac{1}{\lambda^2}\left(\ela\left[\#\C\left(\orig,\xi^{\orig}\right) \;\middle|\;\#\C\left(\orig,\xi^{\orig}\right)<\infty\right] -1\right).
    \end{multline}
    Since $\lambda<\lambda_T$ implies $\pla\left(\#\C\left(\orig,\xi^{\orig}\right)<\infty\right) =1$, we now have
    \begin{equation}
        \left(\frac{\dd \chi}{\dd \lambda}\right)^2 \leq \frac{1}{\lambda^2}\chi\left(\lambda\right)\ela\left[\#\C\left(\orig,\xi^{\orig}\right)^2\right].
    \end{equation}
    Rearranging this then recovers the result.
\end{proof}

The following proposition proves Theorem~\ref{thm:CritExponents}\ref{thm:CritExponents - SusceptibilityMoment}, and therefore completes the entire proof of Theorem~\ref{thm:CritExponents}.
\begin{prop}
\label{prop:ClusterMomentExponent}
    For each $n\geq 1$, there exist $0<C_n\leq C_n'<\infty$ such that
    \begin{equation}
        C_n\left(\lambda_T-\lambda\right)^{1-2n} \leq \ela\left[\#\C\left(\orig,\xi^{\orig}\right)^n\right] \leq C_n'\left(\lambda_T-\lambda\right)^{1-2n}
    \end{equation}
    for all $\lambda<\lambda_T$.
\end{prop}

\begin{proof}
    The upper bound directly follows from Lemma~\ref{lem:nclusterUpper} and the upper bound in Theorem~\ref{thm:CritExponents}\ref{thm:CritExponents - Susceptibility}: for all $\lambda<\lambda_T$ and $n\geq 1$
    \begin{equation}
        \ela\left[\#\C\left(\orig,\xi^{\orig}\right)^n\right] \leq C_n\left(C'\left(\lambda_T-\lambda\right)^{-1}\right)^{2n-1}.
    \end{equation}

    For the lower bound, first observe that by the Cauchy-Schwarz inequality
    \begin{multline}
        \left(\ela\left[\#\C\left(\orig,\xi^{\orig}\right)^{n-1}\right]\right)^2 = \left(\ela\left[\#\C\left(\orig,\xi^{\orig}\right)^{\frac{n}{2}}\#\C\left(\orig,\xi^{\orig}\right)^{\frac{n}{2}-1}\right]\right)^2\\
        \leq \ela\left[\#\C\left(\orig,\xi^{\orig}\right)^n\right]\ela\left[\#\C\left(\orig,\xi^{\orig}\right)^{n-2}\right]
    \end{multline}
    for all $n\geq 1$ and $\lambda\geq 0$ (recall $\#\C\left(\orig,\xi^{\orig}\right)\geq 1$ everywhere for all $\lambda\geq 0$). Therefore by induction
    \begin{equation}
        \frac{\ela\Big[\#\C\left(\orig,\xi^{\orig}\right)^{n{}}\Big]}{\ela\left[\#\C\left(\orig,\xi^{\orig}\right)^{n-1}\right]} \geq \frac{\ela\left[\#\C\left(\orig,\xi^{\orig}\right)^{n-1}\right]}{\ela\left[\#\C\left(\orig,\xi^{\orig}\right)^{n-2}\right]}\geq \ldots \geq \frac{\ela\left[\#\C\left(\orig,\xi^{\orig}\right)^2\right]}{\ela\Big[\#\C\left(\orig,\xi^{\orig}\right)\Big]}=: S\left(\lambda\right),
    \end{equation}
    and
    \begin{equation}
    \label{eqn:nClusterLowerbound}
        \ela\left[\#\C\left(\orig,\xi^{\orig}\right)^n\right] \geq \chi\left(\lambda\right) S\left(\lambda\right)^{n-1}
    \end{equation}
    for all $n\geq 1$ and $\lambda\geq 0$.

    From Lemma~\ref{lem:susceptibilityderivativeLowerBound} and Lemma~\ref{lem:SecondMomentLowerBound}, we know there exists $K'<\infty$ such that for $\lambda<\lambda_T$
    \begin{equation}
        S\left(\lambda\right) \geq K'\chi\left(\lambda\right)^2.
    \end{equation}
    Therefore the lower bound in Theorem~\ref{thm:CritExponents}\ref{thm:CritExponents - Susceptibility} and \eqref{eqn:nClusterLowerbound} imply that for each $n\geq 1$ there exists $C_n>0$ such that
    \begin{equation}
        \ela\left[\#\C\left(\orig,\xi^{\orig}\right)^n\right] \geq C_n \left(\lambda_T-\lambda\right)^{-1}\left(\left(\lambda_T-\lambda\right)^{-2}\right)^{n-1} = C_n\left(\lambda_T-\lambda\right)^{1-2n}
    \end{equation}
    for all $\lambda<\lambda_T$.
\end{proof}

\begin{remark}
    Observe that both the upper and lower bounds in Proposition~\ref{prop:ClusterMomentExponent} made use of the hyperbolic space. For the upper bound this came via the upper bound in Theorem~\ref{thm:CritExponents}\ref{thm:CritExponents - Susceptibility}, and for the lower bound this came via the lower bound on the derivative of $\chi\left(\lambda\right)$ in Lemma~\ref{lem:susceptibilityderivativeLowerBound}.
\end{remark}

\section{Clusters in Half-spaces: Proofs}\label{sec:clusterHalfspaceProofs}
Here we prove Proposition~\ref{prop:BoundHalf-SpaceSusceptMagnet} thereby finishing the proof of Theorem~\ref{thm:CritExponents}.
As a first step in the proof, we bound the expected number of vertices in the cluster by the expected hyperbolic area of its $\varepsilon$-halo. 
This step in the proof is based on an analogous argument for the Boolean disc model on $\R^2$ in \cite{roy1990russo}.

\begin{lemma}\label{lem:haloBound}
    Let $d=2,3$. For all $\varepsilon,\lambda^*>0$, there exists a constant $0<K_1(\varepsilon,\lambda^*)<\infty$ such that for all $\lambda\in\left[0,\lambda^*\right]$,
    \begin{equation}
        \ela\left[\#\C\left(\orig,\xi^{\orig}\right)\right] \leq K_1 \ela\left[\habsd{h_\varepsilon\left(\C\left(\orig,\xi^{\orig}\right)\right)}\right]
    \end{equation}
\end{lemma}

\begin{proof}
    We first state the proof for $\HypTwo$, and then explain the necessary adaptations for $\HypThree$. 
    
    Partition $\HypTwo$ by the regular triangular lattice with seven equilateral triangles meeting at a vertex, where we choose to have $\orig$ contained in the interior of one of the triangular cells. This lattice is denoted by the Schl{\"a}fli symbol $\left\{3,7\right\}$ (see Figure~\ref{fig:Schlaefi37lattice}). Let $C$ denote a cell in this lattice, and note that $\habs{C}=\frac{\pi}{7}$ (from Lemma~\ref{lem:areaoftriangles}) and the side length of this cell $L=2\arcosh\big(\frac{1}{2\sin\frac{\pi}{7}}\big)$ (by the second cosine rule - see Lemma~\ref{lem:secondcosinerule}). Recall that $\eta_{\HypTwo\setminus C}$ denotes the vertex Poisson point process outside the cell $C$, and $\C\big(\orig,\xi^\orig_{\HypTwo\setminus C}\big)$ denotes the cluster of $\orig$ that is formed without using any vertices in $C$ (except perhaps $\orig$). 
    
    For cells $C$ not containing $\orig$, we have
    \begin{align}
        &\ela\left[\#\left(\C\left(\orig,\xi^\orig\right)\cap C\right) \mid\eta_{\HypTwo\setminus C}\right] \leq \sum_{k=1}^\infty k\pla\left(\#\eta_C=k \text{ and }\exists y\in\eta_C\colon y\sim \C\left(\orig,\xi^\orig_{\HypTwo\setminus C}\right)\mid \eta_{\HypTwo\setminus C}\right) \nonumber\\
        &\hspace{2cm} = e^{-\lambda\habs{C}}\sum_{k=1}^\infty \frac{k}{k!}\lambda^k\habs{C}^k\left(1-\left(1-\frac{1}{\habs{C}}\int_C\pla\left(x\sim\C\left(\orig,\xi^\orig_{\HypTwo\setminus C}\right)\mid \eta_{\HypTwo\setminus C}\right)\dd x\right)^k\right)\nonumber\\
        &\hspace{2cm}\leq e^{-\lambda\habs{C}}\sum_{k=1}^\infty \frac{k}{k!}\lambda^k\habs{C}^k\times \frac{k}{\habs{C}}\int_C\pla\left(x\sim\C\left(\orig,\xi^\orig_{\HypTwo\setminus C}\right)\mid \eta_{\HypTwo\setminus C}\right)\dd x \nonumber\\
        &\hspace{2cm}= \lambda\habs{C}\left(\lambda\habs{C}+1\right)\left(\frac{1}{\habs{C}}\int_C\pla\left(x\sim\C\left(\orig,\xi^\orig_{\HypTwo\setminus C}\right)\mid \eta_{\HypTwo\setminus C}\right)\dd x \right).
    \end{align}
    If a cell $C$ contains a vertex we want to find the minimal area this vertex can contribute {to the $\varepsilon$-halo}: $v_\varepsilon:=\inf_{x\in C}\habs{B_\varepsilon(x)\cap C}$. If $\varepsilon$ is sufficiently small then $v_\varepsilon$ is found by placing $x$ in a corner of the triangle, in which case the intersection is $\frac{1}{7}$ of the $\varepsilon$-ball ($v_\varepsilon=\frac{4\pi}{7}\left(\sinh \varepsilon\right)^2$). This fails once $\varepsilon$ is big enough to reach outside the far side of the triangle. By the hyperbolic sine rule, this distance $l$ between a vertex and the far side satisfies $\sinh l = \sin\left(\frac{2\pi}{7}\right)\sinh L$. In such a case, we simply truncate the $\varepsilon$-ball to the $l$-ball. Therefore
    \begin{equation}
        v_\varepsilon=\frac{4\pi}{7}\left(\min\left\{\sinh \varepsilon, \sin\left(\frac{2\pi}{7}\right)\sinh L\right\}\right)^2.
    \end{equation}
    By restricting to the event that there is exactly one vertex in $C$, we therefore have
    \begin{equation}
        \ela\left[\habs{h_\varepsilon\left(\C\left(\orig,\xi^\orig\right)\right)\cap C} \mid\eta_{\HypTwo\setminus C}\right] \geq v_\varepsilon \lambda\habs{C}\e^{-\lambda\habs{C}}\frac{1}{\habs{C}}\int_C\pla\left(x\sim\C\left(\orig,\xi^\orig_{\HypTwo\setminus C}\right)\mid \eta_{\HypTwo\setminus C}\right)\dd x.
    \end{equation}
    Therefore for $C$ not containing $\orig$, we have
    \begin{equation}
        \ela\left[\#\left(\C\left(\orig,\xi^\orig\right)\cap C\right) \mid\eta_{\HypTwo\setminus C}\right] \leq \frac{1}{v_\varepsilon}\left(\lambda\habs{C}+1\right)\e^{\lambda\habs{C}}\ela\left[\habs{h_\varepsilon\left(\C\left(\orig,\xi^\orig\right)\right)\cap C} \mid\eta_{\HypTwo\setminus C}\right].
    \end{equation}
    If $C$ is the cell containing $\orig$, then we use the trivial bound
    \begin{equation}
        \ela\left[\#\left(\C\left(\orig,\xi^\orig\right)\cap C\right) \mid\eta_{\HypTwo\setminus C}\right] \leq \ela\left[\#\eta^0_C\right] = 1+\lambda\habs{C}.
    \end{equation}

    By taking the expectations over $\eta_{\HypTwo\setminus C}$ and then summing over $C$, we arrive at
    \begin{equation}
        \ela\left[\#\C\left(\orig,\xi^\orig\right)\right] \leq  1+\lambda\habs{C} + \frac{1}{v_\varepsilon}\left(\lambda\habs{C}+1\right)\e^{\lambda\habs{C}}\ela\left[\habs{h_\varepsilon\left(\C\left(\orig,\xi^\orig\right)\right)}\right].
    \end{equation}
    Since $\ela\left[\habs{h_\varepsilon\left(\C\left(\orig,\xi^\orig\right)\right)}\right] \geq 4\pi\left(\sinh\varepsilon\right)^2$, this proves the result.

    For $\HypThree$, we simply change the cells that form the lattice. It is shown in \cite{brandts2024regular} that there is no regular hyperbolic tetrahedral space-filler for $\HypThree$, but the regular hyperbolic cubic space-filler serves our purposes instead.
\end{proof}

\begin{remark}
    We will only require the above result for $d=2,3$, due to limitations in Lemma~\ref{lem:DeterministicBound}. Nevertheless, it should be possible to extend the above argument to higher dimensions. The only complication is in the choice of the cells used to fill the space. \cite{brandts2024regular} show that there are only two bounded regular hyperbolic space-filler cells for $\mathbb{H}^4$ (the hyperbolic $120$ and $600$-cells) and none for $\HypDim$ for $d\geq 5$. For $d=4$, the above argument can just be repeated with the $120$ or $600$-cells, but for $d\geq 5$ one would need to use more than one type of cell (preferably only finitely many types) to fill the space.
\end{remark}

\begin{figure}
    \centering
    \includegraphics[width=0.5\linewidth]{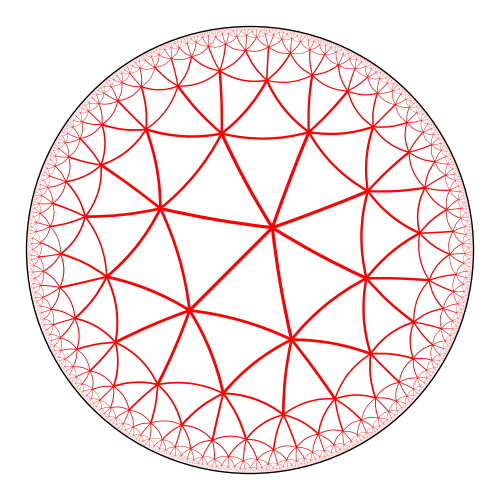}
    \caption{A tiling of $\HypTwo$ (in the Poincar{\'e} disc model) with Schl{\"a}fi symbol $\left\{3,7\right\}$, used in the proof of Lemma~\ref{lem:haloBound}.}
    \label{fig:Schlaefi37lattice}
\end{figure}

Given a set $S\subset\HypDim$ with only finitely many elements, we define $\partial_\varepsilon\convex{S}:=\big\{x\in S:B_\varepsilon(x)\not\subset\convex{S}\big\}$ to be the elements of $S$ whose $\varepsilon$-balls are not contained in $\convex{S}$. Also let $\partial\convex{S} := \bigcap_{\varepsilon>0}\partial_\varepsilon\convex{S}$, the elements of $S$ that lie on the boundary of $\convex{S}$.

\begin{lemma}
\label{lem:DeterministicBound}
    Let $d=2,3$. For any finite $S\subset\HypDim$,
    \begin{equation}
    \label{eqn:keystone}
        \habsd{\convex{S}} \leq 
        \pi\left(\#\partial\convex{S}-2\right)
    \end{equation}
    For all $\varepsilon>0$, there exists a constant $\pi<K_\varepsilon<\infty$ such that for all finite $S\subset\HypTwo$,
    \begin{equation}
        \habsd{h_\varepsilon\left(S\right)} \leq K_\varepsilon \#\partial_\varepsilon \convex{S}.
    \end{equation}
\end{lemma}

\begin{proof}
    We first prove \eqref{eqn:keystone}. 
    For $d=2$, since $S$ is finite, $\convex{S}$ is a bounded convex polygon. Any bounded convex polygon with $n$ vertices can be viewed as the union of $n-2$ triangles. A hyperbolic triangle can have at most a hyperbolic area of $\pi$, and this produces the first bound.

    For $d=3$, Madras and Wu \cite[Lemma 8]{madras2010trees} show that 
    \begin{equation}
        \habst{\convex{S}}\le V_3 \left(2\#\partial\convex{S}-4\right),
    \end{equation}
    where $V_3$ is the (hyperbolic) volume of a regular tetrahedra, whose volume is bounded above by $\pi/2$, cf. \cite[Prop. 2]{haagerup1981simplices}. We thus arrive at the same upper bound. 

    For the second bound, we first bound
    \begin{equation}
        \habsd{h_\varepsilon\left(S\right)\cap \convex{S}} \leq \habsd{\convex{S}},
    \end{equation}
    and account for the elements of $S$ whose $\varepsilon$-balls are not contained in $\convex{S}$ by bounding
    \begin{equation}
        \habsd{h_\varepsilon\left(S\right)\setminus \convex{S}} \leq \habsd{B_\varepsilon\left(\orig\right)} \#\partial_\varepsilon \convex{S}.
    \end{equation}
    Combining these bounds and using \eqref{eqn:keystone} gives
    \begin{equation}
        \habsd{h_\varepsilon\left(S\right)} \leq \habsd{\convex{S}} + \habsd{B_\varepsilon\left(\orig\right)} \#\partial_\varepsilon \convex{S} \leq \left(\pi + \habsd{B_\varepsilon\left(\orig\right)}\right)\#\partial_\varepsilon \convex{S}.
    \end{equation}
    Note that we have used $\partial \convex{S}\subset \partial_\varepsilon\convex{S}$ here. The result then follows with $K_\varepsilon = \pi+\habsd{B_\varepsilon\left(\orig\right)}$.
\end{proof}

\begin{remark}\label{rem:dge4}
    Finding an appropriate replacement for \eqref{eqn:keystone} is where the argument runs into issues for $d\geq 4$. As noted by \cite{madras2010trees} in their arguments, \cite{grunbaum2003convex} shows that if you consider a class of polytope called \emph{cyclic polytopes} then even in $d=4$ you can find examples with $n$ vertices but $\frac{1}{2}\left(n^2-3n\right)$ `$3$-dimensional facets.'
\end{remark}

We now state two lemmas, which together imply Proposition~\ref{prop:BoundHalf-SpaceSusceptMagnet}.

\begin{lemma}
\label{lem:numboundaryequalsgivenorigboundary}
    For all $d\geq 2$ and $\lambda<\lambda_c$,
    \begin{equation}
    \label{eqn:convHulltoclustersize}
        \ela\left[\#\partial_\varepsilon\convex{\C\left(\orig,\xi^{\orig}\right)}\right] = \ela\left[\#\C\left(\orig,\xi^{\orig}\right)\Id_{\left\{\orig\in\partial_\varepsilon\convex{\C\left(\orig,\xi^{\orig}\right)}\right\}}\right]
    \end{equation}
\end{lemma}

In the following lemma, $\mathfrak{S}_{d-1}=\frac{2\pi^{d/2}}{\Gamma(d/2)}$ denotes the Lebesgue measure of the whole sphere $S_{d-1}$.
\begin{lemma}\label{lem:rotationFactor}
    For all $d\geq 2$, half-spaces $H\ni\orig$, $\varepsilon\in\left(0,\dist{\partial H,\orig}\right)$, and $\lambda<\lambda_c$,
    \begin{equation}
        \ela\left[\#\C\left(\orig,\xi^{\orig}\right)\Id_{\left\{\orig\in\partial_\varepsilon\convex{\C\left(\orig,\xi^{\orig}\right)}\right\}}\right] \leq \frac{\mathfrak{S}_{d-1}}{\mathfrak{S}_{d-2}}\frac{\ela\left[\#\C\left(\orig,\xi^{\orig}\right) \Id_{\left\{\C\left(\orig,\xi^{\orig}\right)\subset\Mcal_\varepsilon\left( H,\orig\right)\right\}}\right]}{\int^{\theta_\varepsilon\left(H,\orig\right)}_0\left(\sin t\right)^{d-2}\dd t}.
    \end{equation}
\end{lemma}

\begin{proof}[Proof of Proposition~\ref{prop:BoundHalf-SpaceSusceptMagnet}]
    The bound for the susceptibility follows from Lemmas~\ref{lem:haloBound}-\ref{lem:rotationFactor}. The bound for the ghost free susceptibility starts in the same way with Lemmas~\ref{lem:haloBound} and \ref{lem:DeterministicBound}. Repeating the proofs of Lemma~\ref{lem:numboundaryequalsgivenorigboundary} and \ref{lem:rotationFactor} with $\elaq$ in the place of $\ela$ and with the extra $\Id_{\left\{\C\left(\orig,\xi^\orig\right)\cap\Gcal=\emptyset\right\}}$ then produces
    \begin{equation}
        \elaq\left[\#\partial_\varepsilon\convex{\C\left(\orig,\xi^{\orig}\right)}\Id_{\left\{\C\left(\orig,\xi^\orig\right)\cap\Gcal=\emptyset\right\}}\right] = \elaq\left[\#\C\left(\orig,\xi^{\orig}\right)\Id_{\left\{\orig\in\partial_\varepsilon\convex{\C\left(\orig,\xi^{\orig}\right)}\right\}}\Id_{\left\{\C\left(\orig,\xi^\orig\right)\cap\Gcal=\emptyset\right\}}\right]
    \end{equation}
    for all $d\geq 2$, $q\in\left(0,1\right)$, and $\lambda<\lambda_\mathrm{c}$. Similarly, for all $d\geq 2$, half-spaces $H\ni\orig$, $\varepsilon\in\left(0,\dist{\partial H,\orig}\right)$, $q\in\left(0,1\right)$, and $\lambda<\lambda_c$,
    \begin{multline}
        \elaq\left[\#\C\left(\orig,\xi^{\orig}\right)\Id_{\left\{\C\left(\orig,\xi^\orig\right)\cap\Gcal=\emptyset\right\}}\Id_{\left\{\orig\in\partial_\varepsilon\convex{\C\left(\orig,\xi^{\orig}\right)}\right\}}\right] \\
        \leq \frac{\mathfrak{S}_{d-1}}{\mathfrak{S}_{d-2}}\frac{\elaq\left[\#\C\left(\orig,\xi^{\orig}\right) \Id_{\left\{\C\left(\orig,\xi^\orig\right)\cap\Gcal=\emptyset\right\}}\Id_{\left\{\C\left(\orig,\xi^{\orig}\right)\subset\Mcal_\varepsilon\left( H,\orig\right)\right\}}\right]}{\int^{\theta_\varepsilon\left(H,\orig\right)}_0\left(\sin t\right)^{d-2}\dd t}.
    \end{multline}
    These then produce the bound for the ghost free susceptibility with the same constant $\kappa$.

    For the magnetisation bound we relate the magnetisation to the susceptibility. From Lemma~\ref{lem:MagnetPrelim} we have that for $\lambda<\lambda_\mathrm{c}$ and $q\in\left(0,1\right)$
    \begin{equation}
        \left(1-q\right)\frac{\partial}{\partial q}\plaq\left(\conn{\orig}{\Gcal}{\xi^\orig}\right) = \elaq\left[\#\C\left(\orig,\xi^\orig\right)\Id_{\left\{\C\left(\orig,\xi^\orig\right)\cap\Gcal=\emptyset\right\}}\right].
    \end{equation}
    The same idea means that for $\lambda<\lambda_\mathrm{c}$ and $q\in\left(0,1\right)$
    \begin{equation}
        \left(1-q\right)\frac{\partial}{\partial q}\plaq\left(\conn{\orig}{\Gcal}{\xi^\orig},\C\left(\orig,\xi^\orig\right)\subset H\right) = \elaq\left[\#\C\left(\orig,\xi^\orig\right)\Id_{\left\{\C\left(\orig,\xi^\orig\right)\cap\Gcal=\emptyset\right\}}\Id_{\left\{\C\left(\orig,\xi^\orig\right)\subset H\right\}}\right].
    \end{equation}
    In particular, this means that the bound for the expectations of the cluster sizes in the first part of the result imply the differential inequality
    \begin{equation}
        \frac{\partial}{\partial q}\plaq\left(\conn{\orig}{\Gcal}{\xi^\orig}\right) \leq \kappa \frac{\partial}{\partial q}\plaq\left(\conn{\orig}{\Gcal}{\xi^\orig},\C\left(\orig,\xi^\orig\right)\subset H\right).
    \end{equation}
    Since, $\lim_{q\searrow 0}\mathbb{P}_{\lambda,0}\left(\conn{\orig}{\Gcal}{\xi^\orig},\C\left(\orig,\xi^\orig\right)\subset H\right) \leq \lim_{q\searrow 0}\mathbb{P}_{\lambda,0}\left(\conn{\orig}{\Gcal}{\xi^\orig}\right) = \theta\left(\lambda\right)=0$ for all $\lambda<\lambda_c$ (see Lemma~\ref{lem:MagnetPrelim}), we have
    \begin{equation}
        \plaq\left(\conn{\orig}{\Gcal}{\xi^\orig}\right) \leq \kappa \plaq\left(\conn{\orig}{\Gcal}{\xi^\orig},\C\left(\orig,\xi^\orig\right)\subset H\right)
    \end{equation}
    for all $\lambda<\lambda_\mathrm{c}$ and $q\in\left(0,1\right)$.
\end{proof}

It remains to prove Lemmas~\ref{lem:numboundaryequalsgivenorigboundary} and \ref{lem:rotationFactor}.
\begin{proof}[Proof of Lemma~\ref{lem:numboundaryequalsgivenorigboundary}]
In the following, given a finite set $S$ we use $N_\varepsilon(S):= \#\partial_\varepsilon\convex{S}$.
First we expand the left hand side of \eqref{eqn:convHulltoclustersize}, conditioning on the size of the cluster:
    \begin{align}
        &\ela\left[\#\partial_\varepsilon\convex{\C\left(\orig,\xi^\orig\right)}\right]\nonumber\\
        &= \sum_{n=1}^\infty \pla\left(\#\C\left(\orig,\xi^\orig\right)=n\right)\ela\left[\#\partial_\varepsilon\convex{\C\left(\orig,\xi^\orig\right)}\mid \#\C\left(\orig,\xi^\orig\right)=n\right]\nonumber\\
        &= \sum_{n=1}^\infty \pla\left(\#\C\left(\orig,\xi^\orig\right)=n\right) \int_{\left(\HypDim\right)^{n-1}}N_\varepsilon\left(\left\{\orig,x_2,\ldots,x_n\right\}\right)\pla\left(\left\{\orig,x_2,\ldots,x_n\right\}\text{ is a cluster}\right)\dd x_2\ldots \dd x_n.
    \end{align}
    Note that since $\lambda<\lambda_c$, $\pla\left(\#\C\left(\orig,\xi^\orig\right)=\infty\right)=0$.

    We now similarly expand the right hand side:
    \begin{align}
        &\ela\left[\#\C\left(\orig,\xi^{\orig}\right)\Id_{\left\{\orig\in\partial_\varepsilon\convex{\C\left(\orig,\xi^{\orig}\right)}\right\}}\right]\nonumber\\
        &=\sum_{n=1}^\infty \pla\left(\#\C\left(\orig,\xi^\orig\right)=n\right)\ela\left[\#\C\left(\orig,\xi^{\orig}\right)\Id_{\left\{\orig\in\partial_\varepsilon\convex{\C\left(\orig,\xi^{\orig}\right)}\right\}}\mid \#\C\left(\orig,\xi^\orig\right)=n\right]\nonumber\\
        &=\sum_{n=1}^\infty \pla\left(\#\C\left(\orig,\xi^\orig\right)=n\right)n\int_{\left(\HypDim\right)^{n-1}}\Id_{\left\{\orig\in\partial_{\varepsilon}\convex{\left\{\orig,x_2,\ldots,x_n\right\}}\right\}}\pla\left(\left\{\orig,x_2,\ldots,x_n\right\}\text{ is a cluster}\right)\dd x_2\ldots \dd x_n.
    \end{align}
    By translation invariance, we have
    \begin{align}
        &n\int_{\left(\HypDim\right)^{n-1}}\Id_{\left\{\orig\in\partial_{\varepsilon}\convex{\left\{\orig,x_2,\ldots,x_n\right\}}\right\}}\p\left(\left\{\orig,x_2,\ldots,x_n\right\}\text{ is a cluster}\right)\dd x_2\ldots \dd x_n\nonumber\\
        &\hspace{2cm}= \sum^n_{i=1}\int_{\left(\HypDim\right)^{n-1}}\Id_{\left\{x_i\in\partial_{\varepsilon}\convex{\left\{\orig,x_2,\ldots,x_n\right\}}\right\}}\p\left(\left\{\orig,x_2,\ldots,x_n\right\}\text{ is a cluster}\right)\dd x_2\ldots \dd x_n\nonumber\\
        &\hspace{2cm}= \int_{\left(\HypDim\right)^{n-1}}N_\varepsilon\left(\left\{\orig,x_2,\ldots,x_n\right\}\right)\p\left(\left\{\orig,x_2,\ldots,x_n\right\}\text{ is a cluster}\right)\dd x_2\ldots \dd x_n,
    \end{align}
    and therefore the right hand side and the left hand side are equal.
\end{proof}

\begin{proof}[Proof of Lemma~\ref{lem:rotationFactor}]
    The idea is that for each cluster configuration with $\orig\in\partial_\varepsilon\convex{\C\left(\orig,\xi^{\orig}\right)}$ there are rotations about $\orig$ such that $\C\left(\orig,\xi^{\orig}\right)\subset\Mcal_\varepsilon\left(H,\orig\right)$, and each of these rotated configurations has the same probability density. Since $\orig\in\partial_\varepsilon\convex{\C\left(\orig,\xi^{\orig}\right)}$, there exists a half-space $M$ such that $\dist{\partial M,\orig}\leq \varepsilon$ and $\C\left(\orig,\xi^{\orig}\right)\subset M$. By associating rotations, $r$, with elements of the sphere $S_{d-1}$, the uniform probability measure of rotations such that $r(M)\subset \Mcal_\varepsilon\left(H,\orig\right)$ is given by
    \begin{equation}
        \frac{\mathfrak{S}_{d-2}}{\mathfrak{S}_{d-1}}\int^{\theta_\varepsilon\left(H,\orig\right)}_0\left(\sin t\right)^{d-2}\dd t.
    \end{equation}
    Since $\mathfrak{S}_0=2$, $\mathfrak{S}_1=2\pi$, and $\mathfrak{S}_2=4\pi$, for $d=2,3$ this takes the form
    \begin{equation}
    \label{eqn:sizeofcap}
        \begin{cases}
            \frac{1}{\pi}\theta_\varepsilon\left(H,\orig\right) & \colon d=2\\
            \frac{1}{2}\left(1- \cos \left(\theta_\varepsilon\left(H,\orig\right)\right)\right) &\colon d=3.
        \end{cases}
    \end{equation}
    Imprecisely speaking, for each configuration $\orig\in\partial_\varepsilon\convex{\C\left(\orig,\xi^{\orig}\right)}$, there are at least \eqref{eqn:sizeofcap}-many ``equally likely'' configurations such that $\C\left(\orig,\xi^{\orig}\right)\subset\Mcal_\varepsilon\left(H,\orig\right)$. This would then imply the result. The following argument gives the details for the $\HypTwo$ case, higher dimensions are analogous. 

    Recall from the proof of Lemma~\ref{lem:numboundaryequalsgivenorigboundary} that
    \begin{align}
        &\ela\left[\#\C\left(\orig,\xi^{\orig}\right)\Id_{\left\{\orig\in\partial_\varepsilon\convex{\C\left(\orig,\xi^{\orig}\right)}\right\}}\right] \nonumber\\
        &\hspace{2cm}=\e^{-\lambda\int_\HypTwo\connf(y,\orig)\dd y} + \sum_{n=2}^\infty n\pla\left(\#\C\left(\orig,\xi^\orig\right)=n\right)\int_{\left(\HypTwo\right)^{n-1}}\Id_{\left\{\orig\in\partial_{\varepsilon}\convex{\left\{\orig,x_2,\ldots,x_n\right\}}\right\}}\nonumber\\
        &\hspace{7cm}\times\pla\left(\left\{\orig,x_2,\ldots,x_n\right\}\text{ is a cluster}\right)\dd x_2\ldots \dd x_n.
    \end{align}
    Here we use the Poincar{\'e} disc radial coordinates for $\HypTwo$, so for $x\in\HypTwo$, $\rho\in\left[0,1\right)$, and $\theta\in\left[0,2\pi\right)$,
    \begin{equation}
        \dd x = \frac{4\rho}{\left(1-\rho^2\right)^2}\dd\rho\dd\theta.
    \end{equation}
    Then for $n\geq 2$ we can write
    \begin{align}
        &\int_{\left(\HypTwo\right)^{n-1}}\Id_{\left\{\orig\in\partial_{\varepsilon}\convex{\left\{\orig,x_2,\ldots,x_n\right\}}\right\}}\pla\left(\left\{\orig,x_2,\ldots,x_n\right\}\text{ is a cluster}\right)\dd x_2\ldots \dd x_n\nonumber\\
        & \hspace{1cm}= \int_{\vec{\theta}\in\left[0,2\pi\right)^{n-1}}\int_{\vec{\rho}\in\left[0,1\right)^{n-1}} \Id_{\left\{\orig\in\partial_{\varepsilon}\convex{\left\{\orig,\left(\rho_1,\theta_{1}\right),\ldots,\left(\rho_{n-1},\theta_{n-1}\right)\right\}}\right\}}\nonumber\\
        &\hspace{4cm}\times\pla\left(\left\{\orig,\left(\rho_1,\theta_{1}\right),\ldots,\left(\rho_{n-1},\theta_{n-1}\right)\right\}\text{ is a cluster}\right)\prod^{n-1}_{i=1}\frac{4\rho_i}{2\pi\left(1-\rho_i^2\right)^2}\dd \rho_i\dd\theta_i\nonumber\\
        & \hspace{1cm}= \int_{\vec{\theta}\in\left[0,2\pi\right)^{n-2}}\int_{\vec{\rho}\in\left[0,1\right)^{n-1}} \Id_{\left\{\orig\in\partial_{\varepsilon}\convex{\left\{\orig,\left(\rho_1,0\right),\ldots,\left(\rho_{n-1},\theta_{n-1}\right)\right\}}\right\}}\nonumber\\
        &\hspace{2cm}\times\pla\left(\left\{\orig,\left(\rho_1,0\right),\ldots,\left(\rho_{n-1},\theta_{n-1}\right)\right\}\text{ is a cluster}\right)\frac{4\rho_1}{\left(1-\rho_1^2\right)^2}\dd\rho_1\prod^{n-1}_{i=2}\frac{4\rho_i}{2\pi\left(1-\rho_i^2\right)^2}\dd \rho_i\dd\theta_i.
    \end{align}

    Observe that
    \begin{align}
        &\ela\left[\#\C\left(\orig,\xi^{\orig}\right)\Id_{\left\{\C\left(\orig,\xi^{\orig}\right)\subset\Mcal_\varepsilon\left(H,\orig\right)\right\}}\right] \nonumber\\
        &\hspace{2cm}=\e^{-\lambda\int_\HypTwo\connf(y,\orig)\dd y}+\sum_{n=2}^\infty n\pla\left(\#\C\left(\orig,\xi^\orig\right)=n\right)\int_{\left(\HypTwo\right)^{n-1}}\Id_{\left\{\C\left(\orig,\xi^{\orig}\right)\subset\Mcal_\varepsilon\left(H,\orig\right)\right\}}\nonumber\\
        &\hspace{7cm}\times\pla\left(\left\{\orig,x_2,\ldots,x_n\right\}\text{ is a cluster}\right)\dd x_2\ldots \dd x_n.
    \end{align}
    Then for $n\geq 2$ we can once again use the radial coordinates of the Poincar{\'e} disc model to write
    \begin{align}
        &\int_{\left(\HypTwo\right)^{n-1}}\Id_{\left\{\left\{\orig,x_2,\ldots,x_n\right\}\subset\Mcal_\varepsilon\left(H,\orig\right)\right\}}\pla\left(\left\{\orig,x_2,\ldots,x_n\right\}\text{ is a cluster}\right)\dd x_2\ldots \dd x_n\nonumber\\
        & \hspace{1cm}= \int_{\vec{\theta}\in\left[0,2\pi\right)^{n-1}}\int_{\vec{\rho}\in\left[0,1\right)^{n-1}} \Id_{\left\{\left\{\orig,\left(\rho_1,\theta_{1}\right),\ldots,\left(\rho_{n-1},\theta_{n-1}\right)\right\}\subset\Mcal_\varepsilon\left(H,\orig\right)\right\}}\nonumber\\
        &\hspace{4cm}\times\pla\left(\left\{\orig,\left(\rho_1,\theta_{1}\right),\ldots,\left(\rho_{n-1},\theta_{n-1}\right)\right\}\text{ is a cluster}\right)\prod^{n-1}_{i=1}\frac{4\rho_i}{2\pi\left(1-\rho_i^2\right)^2}\dd \rho_i\dd\theta_i\nonumber\\
        &\hspace{1cm}\geq \int_{\theta_1\in I_\varepsilon\left(H,\orig\right)}\int_{\vec{\theta}\in\left[0,2\pi\right)^{n-2}}\int_{\vec{\rho}\in\left[0,1\right)^{n-1}} \Id_{\left\{\orig\in\partial_{\varepsilon}\convex{\left\{\orig,\left(\rho_1,0\right),\ldots,\left(\rho_{n-1},\theta_{n-1}\right)\right\}}\right\}}\nonumber\\
        &\hspace{1.5cm}\times\pla\left(\left\{\orig,\left(\rho_1,0\right),\ldots,\left(\rho_{n-1},\theta_{n-1}\right)\right\}\text{ is a cluster}\right)\frac{4\rho_1}{\left(1-\rho_1^2\right)^2}\dd\rho_1\frac{1}{2\pi}\dd \theta_1\prod^{n-1}_{i=2}\frac{4\rho_i}{2\pi\left(1-\rho_i^2\right)^2}\dd \rho_i\dd\theta_i,
    \end{align}
    where $I_\varepsilon\left(H,\orig\right)$ is a measurable subset of $\left[0,2\pi\right)$. Here we have used that if $\orig\in\partial_{\varepsilon}\mathrm{conv}(\{\orig,\left(\rho_1,\theta_1\right),\ldots,\\\left(\rho_{n-1},\theta_{n-1}\right)\})$, then there are rotations $\theta^*$ of the vertex set such that $\left\{\orig,\left(\rho_1,\theta_{1}+\theta^*\right),\ldots,\left(\rho_{n-1},\theta_{n-1}+\theta^*\right)\right\}\subset\Mcal_\varepsilon\left(H,\orig\right)$. $I_\varepsilon\left(H,\orig\right)$ denotes the angles of these rotations, and will be either an interval or a union of two intervals (and hence be measurable). Furthermore $\left\{\orig\in\partial_{\varepsilon}\convex{\left\{\orig,\left(\rho_1,\theta_1\right),\ldots,\left(\rho_{n-1},\theta_{n-1}\right)\right\}}\right\}$ is rotation invariant, and $\abs*{I_\varepsilon\left(H,\orig\right)}= 2\theta_\varepsilon\left(H,\orig\right)$. We therefore have
    \begin{align}
        &\int_{\left(\HypTwo\right)^{n-1}}\Id_{\left\{\left\{\orig,x_2,\ldots,x_n\right\}\subset\Mcal_\varepsilon\left(\gamma,\orig\right)\right\}}\pla\left(\left\{\orig,x_2,\ldots,x_n\right\}\text{ is a cluster}\right)\dd x_2\ldots \dd x_n\nonumber\\
        &\hspace{1cm}\geq \frac{\theta_\varepsilon\left(H,\orig\right)}{\pi}\int_{\vec{\theta}\in\left[0,2\pi\right)^{n-2}}\int_{\vec{\rho}\in\left[0,1\right)^{n-1}} \Id_{\left\{\orig\in\partial_{\varepsilon}\convex{\left\{\orig,\left(\rho_1,0\right),\ldots,\left(\rho_{n-1},\theta_{n-1}\right)\right\}}\right\}}\nonumber\\
        &\hspace{2cm}\times\pla\left(\left\{\orig,\left(\rho_1,0\right),\ldots,\left(\rho_{n-1},\theta_{n-1}\right)\right\}\text{ is a cluster}\right)\frac{4\rho_1}{\left(1-\rho_1^2\right)^2}\dd\rho_1\prod^{n-1}_{i=2}\frac{4\rho_i}{2\pi\left(1-\rho_i^2\right)^2}\dd \rho_i\dd\theta_i,
    \end{align}
    and the result follows.
\end{proof}

\begin{appendix}
\section{Hyperbolic Triangles}
\label{app:hyperbolictriangles}
    The following lemmas are standard relations for hyperbolic triangles that are used in Sections~\ref{sec:NonTrivialPhaseTransition} and \ref{sec:ClustersHalfSpace}. The parameters $A,B,C,a,b,c,\alpha,\beta,\gamma$ correspond to the vertices, edge lengths, and angles in Figure~\ref{fig:labellingHyperbolicTriangle}.

\begin{figure}
    \centering
    \begin{tikzpicture}[scale=2]
    \begin{scope}
        \clip (0,0) -- (2,2) -- (3,0) -- (0,0);
        \draw[thick] (0,0) circle (20pt);
        \draw[thick] (2,2) circle (20pt);
        \draw[thick] (3,0) circle (20pt);
        \draw (0.4,0.15) node{$\alpha$};
        \draw (1.95,1.6) node{$\beta$};
        \draw (2.6,0.2) node{$\gamma$};
    \end{scope}
        \draw[thick] (0,0) node[left]{$A$} -- (2,2) node[above]{$B$} -- (3,0) node[right]{$C$} -- (0,0);
        \draw (1,1) node[above left]{$c$};
        \draw (2.5,1) node[above right]{$a$};
        \draw (1.5,0) node[below]{$b$};
    \end{tikzpicture}
    \caption{Labelling of vertices, angles, and side lengths of the hyperbolic triangle $\Delta ABC$ used in Lemmas~\ref{lem:tangentformula}, \ref{lem:cosineformula}, \ref{lem:sinerule}, and \ref{lem:secondcosinerule}.}
    \label{fig:labellingHyperbolicTriangle}
\end{figure}

\begin{lemma}[Tangent rule for hyperbolic triangles]\label{lem:tangentformula}
    If the hyperbolic triangle $\Delta ABC$ has a right angle at $C$, then
    \begin{equation}
        \tan \alpha = \frac{\tanh a}{\sinh b}.
    \end{equation}
\end{lemma}
\begin{proof}
    See \cite[Corollary~32.13]{martin2012foundations}.
\end{proof}

\begin{lemma}[Cosine rule for hyperbolic triangles]\label{lem:cosineformula}
    If the hyperbolic triangle $\Delta ABC$ has a right angle at $C$, then
    \begin{equation}
        \cos \alpha = \frac{\tanh b}{\tanh c}.
    \end{equation}
\end{lemma}
\begin{proof}
    See \cite[Corollary~32.13]{martin2012foundations}.
\end{proof}

\begin{lemma}[Sine rule for hyperbolic triangles]\label{lem:sinerule}
    For the hyperbolic triangle $\Delta ABC$,
    \begin{equation}
        \frac{\sin \alpha}{\sinh a} = \frac{\sin \beta}{\sinh b} = \frac{\sin \gamma}{\sinh c}.
    \end{equation}
\end{lemma}
\begin{proof}
    See \cite[Corollary~32.14]{martin2012foundations}.
\end{proof}

\begin{lemma}[Second cosine rule for hyperbolic triangles]\label{lem:secondcosinerule}
    For the hyperbolic triangle $\Delta ABC$,
    \begin{equation}
        \cosh c \sin \alpha \sin \beta = \cos \alpha \cos \beta + \cos \gamma.
    \end{equation}
\end{lemma}
\begin{proof}
    See \cite[Corollary~32.16]{martin2012foundations}.
\end{proof}

\begin{lemma}[Hyperbolic area for hyperbolic triangles]\label{lem:areaoftriangles}
    For the hyperbolic triangle $\Delta ABC$, the hyperbolic area is equal to the angle defect:
    \begin{equation}
        \habsd{\Delta ABC} = \pi - \alpha - \beta - \gamma.
    \end{equation}
\end{lemma}
\begin{proof}
    See \cite[Chapter~33.2]{martin2012foundations}.
\end{proof}
    
\end{appendix}

\subsection*{Acknowledgements.} 
MH is grateful to Johan Tykesson for introducing him to the topic and numerous stimulating discussions.  

This work is supported by NSERC of Canada and by 
\textit{Deutsche Forschungsgemeinschaft} (project number 443880457) through priority program ``Random Geometric Systems'' (SPP 2265). 
We thank the \emph{Centre de recherches math\'ematiques} Montreal for hospitality during a research visit in spring 2022 through the Simons-CRM scholar-in-residence program. 
Part of this work was done while MH was in residence at the Mathematical Sciences Research Institute in Berkeley supported by NSF Grant No. DMS-1928930.

\bibliography{bibliography}{}

\newcommand{\etalchar}[1]{$^{#1}$}
\begin{thebibliography}{HHLM22}

\bibitem[AB87]{AizBar87}
Michael Aizenman and David~J. Barsky.
\newblock Sharpness of the phase transition in percolation models.
\newblock {\em Commun. Math. Phys.}, 108(3):489--526, 1987.

\bibitem[AB91]{BarAiz91}
Michael Aizenman and David~J. Barsky.
\newblock Percolation critical exponents under the triangle condition.
\newblock {\em Ann. Probab.}, 19(4):1520--1536, 1991.

\bibitem[AN84]{AizNew84}
Michael Aizenman and Charles~M. Newman.
\newblock Tree graph inequalities and critical behavior in percolation models.
\newblock {\em J. Statist. Phys.}, 36(1-2):107--143, 1984.

\bibitem[BKL19]{BringKeuscLengl19}
Karl Bringmann, Ralph Keusch, and Johannes Lengler.
\newblock Geometric inhomogeneous random graphs.
\newblock {\em Theor. Comput. Sci.}, 760:35--54, 2019.

\bibitem[BKS24]{brandts2024regular}
Jan Brandts, Michal K{\v{r}}{\'\i}{\v{z}}ek, and Lawrence Somer.
\newblock Regular tessellations of maximally symmetric hyperbolic manifolds.
\newblock {\em Symmetry}, 16(2):141, 2024.

\bibitem[BS96]{BenjaSchra96}
Itai Benjamini and Oded Schramm.
\newblock Percolation beyond {$\mathbb Z\sp d$}, many questions and a few
  answers.
\newblock {\em Electron. Comm. Probab.}, 1:no.\ 8, 71--82 (electronic), 1996.

\bibitem[BS01a]{BenjaSchra01}
Itai Benjamini and Oded Schramm.
\newblock Percolation in the hyperbolic plane.
\newblock {\em J. Am. Math. Soc.}, 14:487--507, 2001.

\bibitem[BS01b]{benjaminischramm2001recurrence}
Itai Benjamini and Oded Schramm.
\newblock {Recurrence of Distributional Limits of Finite Planar Graphs}.
\newblock {\em Electron. J. Probab.}, 6:1 -- 13, 2001.

\bibitem[CD24]{caicedo2023critical}
Alejandro Caicedo and Matthew Dickson.
\newblock Critical exponents for marked random connection models.
\newblock {\em Electron. J. Probab.}, 29:57, 2024.
\newblock Id/No 151.

\bibitem[DH22]{DicHey2022triangle}
Matthew Dickson and Markus Heydenreich.
\newblock The triangle condition for the marked random connection model.
\newblock {\em Preprint arXiv:2210.07727 [math.PR]}, 2022.

\bibitem[Dic25]{dickson2024NonUniqueness}
Matthew Dickson.
\newblock Non-uniqueness phase in hyperbolic marked random connection models
  using the spherical transform.
\newblock {\em Advances in Applied Probability}, pages 1--48, 2025.

\bibitem[DN85]{durrett1985thermodynamic}
R.~Durrett and B.~Nguyen.
\newblock Thermodynamic inequalities for percolation.
\newblock {\em Commun. Math. Phys.}, 99:253--269, 1985.

\bibitem[Fli90]{flight1990many}
Colin Flight.
\newblock How many stemmata?
\newblock {\em Manuscripta}, 34(2):122--128, 1990.

\bibitem[FR80]{foulds1980determining}
L.~R. Foulds and R.~W. Robinson.
\newblock Determining the asymptotic number of phylogenetic trees.
\newblock In {\em Combinatorial Mathematics VII}, pages 110--126. Springer
  Berlin Heidelberg, 1980.

\bibitem[Gr{\"u}03]{grunbaum2003convex}
Branko Gr{\"u}nbaum.
\newblock {\em Convex polytopes}.
\newblock Graduate Texts in Mathematics. Springer New York, second edition,
  2003.

\bibitem[HH17]{HeyHof17}
Markus {Heydenreich} and Remco van~der Hofstad.
\newblock {\em {Progress in high-dimensional percolation and random graphs}}.
\newblock {CRM Short Courses}. Cham: Springer, 2017.

\bibitem[HHLM22]{HeyHofLasMat19}
Markus Heydenreich, Remco van~der Hofstad, G{\"u}nter Last, and Kilian Matzke.
\newblock Lace expansion and mean-field behavior for the random connection
  model, 2022.
\newblock Preprint arXiv:1908.11356 [math.PR]. To appear in Ann. Inst. H.
  Poincaré Probab. Statist.

\bibitem[HM81]{haagerup1981simplices}
Uffe Haagerup and Hans~J Munkholm.
\newblock Simplices of maximal volume in hyperbolic $n$-space.
\newblock {\em Acta Math.}, 147:1--11, 1981.

\bibitem[HM22]{HansenMuller_2022_Voronoi}
Benjamin~T. Hansen and Tobias M{\"u}ller.
\newblock The critical probability for {Voronoi} percolation in the hyperbolic
  plane tends to 1/2.
\newblock {\em Random Struct. Algorithms}, 60(1):54--67, 2022.

\bibitem[Hof24]{Hofstad2024random}
Remco van~der Hofstad.
\newblock {\em Random Graphs and Complex Networks: Volume 2}.
\newblock Cambridge University Press, 2024.

\bibitem[Hug96]{hughes1996random}
Barry~D Hughes.
\newblock {\em Random walks and random environments}.
\newblock Oxford University Press, 1996.

\bibitem[Hut19]{Hutch19}
Tom Hutchcroft.
\newblock Percolation on hyperbolic graphs.
\newblock {\em Geom. Funct. Anal.}, 29(3):766--810, 2019.

\bibitem[Hut20a]{Hutch20}
Tom Hutchcroft.
\newblock The {$L^2$} boundedness condition in nonamenable percolation.
\newblock {\em Electron. J. Probab.}, 25:Paper No. 127, 1--27, 2020.

\bibitem[Hut20b]{Hutch20b}
Tom Hutchcroft.
\newblock Nonuniqueness and mean-field criticality for percolation on
  nonunimodular transitive graphs.
\newblock {\em J. Am. Math. Soc.}, 33(4):1101--1165, 2020.

\bibitem[Hut22]{hutchcroft2022derivation}
Tom Hutchcroft.
\newblock On the derivation of mean-field percolation critical exponents from
  the triangle condition.
\newblock {\em J. Stat. Phys.}, 189(1):6, 2022.

\bibitem[KL20]{Lodewijks2020Explosion}
J\'ulia Komj\'athy and Bas Lodewijks.
\newblock Explosion in weighted hyperbolic random graphs and geometric
  inhomogeneous random graphs.
\newblock {\em Stochastic Process. Appl.}, 130(3):1309--1367, 2020.

\bibitem[KPK{\etalchar{+}}10]{krioukov2010hyperbolic}
Dmitri Krioukov, Fragkiskos Papadopoulos, Maksim Kitsak, Amin Vahdat, and
  Mari{\'a}n Bogun{\'a}.
\newblock Hyperbolic geometry of complex networks.
\newblock {\em Phys. Rev. E}, 82(3):036106, 2010.

\bibitem[Lal98]{Lalle98}
Steven~P. Lalley.
\newblock Percolation on {F}uchsian groups.
\newblock {\em Ann. Inst. H. Poincar\'e Probab. Statist.}, 34(2):151--177,
  1998.

\bibitem[LZ17]{LasZie17}
G\"unter Last and Sebastian Ziesche.
\newblock {On the {O}rnstein–{Z}ernike equation for stationary cluster
  processes and the random connection model}.
\newblock {\em Adv. Appl. Probab.}, 49(4):1260–1287, 2017.

\bibitem[Mar12]{martin2012foundations}
George~Edward Martin.
\newblock {\em The foundations of geometry and the non-Euclidean plane}.
\newblock Springer Science \& Business Media, 2012.

\bibitem[MW10]{madras2010trees}
Neal Madras and C.~Wu.
\newblock {Trees, Animals, and Percolation on Hyperbolic Lattices}.
\newblock {\em Electron. J. Probab.}, 15:2019 -- 2040, 2010.

\bibitem[Ngu87]{Ngu87}
Bao~G. Nguyen.
\newblock Gap exponents for percolation processes with triangle condition.
\newblock {\em J. Stat. Phys.}, 49(1-2):235--243, 1987.

\bibitem[{Pen}91]{Pen91}
Mathew {Penrose}.
\newblock {On a continuum percolation model}.
\newblock {\em {Adv. Appl. Probab.}}, 23(3):536--556, 1991.

\bibitem[PSN00]{PakSmirn00}
Igor Pak and Tatiana Smirnova-Nagnibeda.
\newblock On non-uniqueness of percolation on nonamenable {Cayley} graphs.
\newblock {\em C. R. Acad. Sci., Paris, S{\'e}r. I, Math.}, 330(6):495--500,
  2000.

\bibitem[Roy90]{roy1990russo}
Rahul Roy.
\newblock The {R}usso-{S}eymour-{W}elsh theorem and the equality of critical
  densities and the ``dual" critical densities for continuum percolation on
  $\mathbb{R}^2$.
\newblock {\em Ann. Probab.}, 18(4):1563--1575, 1990.

\bibitem[Sch01]{Schon01}
Roberto~H. Schonmann.
\newblock Multiplicity of phase transitions and mean-field criticality on
  highly non-amenable graphs.
\newblock {\em Commun. Math. Phys.}, 219(2):271--322, 2001.

\bibitem[Sch02]{Schon02}
Roberto~H. Schonmann.
\newblock Mean-field criticality for percolation on planar non-amenable graphs.
\newblock {\em Commun. Math. Phys.}, 225(3):453--463, 2002.

\bibitem[Tyk07]{Tykes07}
Johan Tykesson.
\newblock The number of unbounded components in the {P}oisson {B}oolean model
  of continuum percolation in hyperbolic space.
\newblock {\em Electron. J. Probab.}, 12:no. 51, 1379--1401 (electronic), 2007.

\end{thebibliography}
\bibliographystyle{alpha}

\end{document}